\documentclass[11pt]{article}
\RequirePackage{graphicx,wrapfig}
\RequirePackage{tikz,pgf}
\usepackage{amsmath}
\usepackage{amssymb}
\usepackage[citecolor=red]{hyperref}
\usepackage{latexsym}
\usepackage{amsmath, amsfonts,amssymb, amsthm, euscript,makeidx,color,mathrsfs}
\usepackage{subfig}
\usepackage{epstopdf}
\usepackage{epsfig}

\newtheorem{assumption}{Assumption}

\def\qed{ \ \vrule width.2cm height.2cm depth0cm\smallskip}

\newcommand{\la}{\langle}
\newcommand{\ra}{\rangle}

\newcommand{\dd}{\mathcal{\dagger}}
\newcommand{\ts}{\mathsf{T}}

\newcommand{\ba}{\begin{array}}
\newcommand{\ea}{\end{array}}
\newcommand{\be}{\begin{equation}}
\newcommand{\ee}{\end{equation}}
\newcommand{\bea}{\begin{eqnarray}}
\newcommand{\eea}{\end{eqnarray}}
\newcommand{\beaa}{\begin{eqnarray*}}
\newcommand{\eeaa}{\end{eqnarray*}}

\newcommand{\DKL}{\mathrm{D}_{\mathrm{KL}}}

\def\div{\mathbf{div}}

\def\Vol{\mathbf{Vol}}
\def\div{\mathbf{div}}
\def\cI{{\cal I}}
\def\cJ{{\cal J}}
\def\cK{{\cal K}}

\def\hM{\mathbb{M}}

\def\hR{\mathbb{R}}

\def\pa{\partial}
\def\cd{\cdot}

\def\qed{ \hfill \vrule width.25cm height.25cm depth0cm\smallskip}

\newcommand{\basa}{\begin{assumption}}
\newcommand{\easa}{\end{assumption}}

\newcommand{\bas}{\begin{assum}}
\newcommand{\eas}{\end{assum}}

\def\limP2{\,\mathop{\buildrel \Pi_2\over\longrightarrow\,}}

\def\pa{\partial}

 \def\cd{\cdot}

\def\1{{\bf 1}}

\def\:{\!:\!}

at 9pt

\begin{document}

\newtheorem{thm}{Theorem}[section]
\newtheorem{lem}[thm]{Lemma}
\newtheorem{nota}[thm]{Notation}
\newtheorem{cor}[thm]{Corollary}
\newtheorem{prop}[thm]{Proposition}
\newtheorem{rem}[thm]{Remark}
\newtheorem{eg}[thm]{Example}
\newtheorem{defn}[thm]{Definition}
\newtheorem{assum}[thm]{Assumption}

\renewcommand {\theequation}{\arabic{section}.\arabic{equation}}
\def\thesection{\arabic{section}}

\title{Sub-Riemannian Ricci curvature via generalized Gamma $z$ calculus}

\author{ Qi Feng\thanks{\noindent Department of
Mathematics, University of Michigan, Ann Arbor, 48109; email: qif@umich.edu}
~ and ~ Wuchen Li\thanks{ \noindent Department of
Mathematics, University of South Carolina, 29208;
email: wuchen@mailbox.sc.edu.} 
}

\maketitle
\abstract{We derive sub-Riemannian Ricci curvature tensor for sub-Riemannian manifolds. We provide examples including the Heisenberg group, displacement group, and Martinet sub-Riemannian structure with arbitrary weighted volumes, in which we establish analytical bound conditions for sub-Riemannian curvature dimension bounds and log-Sobolev inequalities. {These conditions can be used to establish the entropy dissipation results for sub-Riemannian drift diffusion processes on a compact spatial domain, in term of $L_1$ distance.} Our derivation of Ricci curvature is based on generalized Gamma $z$ calculus and $z$--Bochner's formula, where $z$ stands for extra directions introduced into the sub-Riemannian degenerate structure.}



\noindent\textbf{Thanks}: {Wuchen Li is supported by AFOSR MURI FA9550-18-1-0502.}
\section{Introduction} 
In Riemannian geometry, Ricci curvature plays essential roles in probability, geometric analysis, and functional inequalities \cite{bakryemery1985}.  Here, the lower bound, often related to the curvature dimension bound, plays crucial roles in studying convergence rate of drift-diffusion process and establishing the concentration inequalities, especially log--Sobolev inequalities. Here the major limitation for curvature dimension bound is that the metric is of Riemannian type, where the metric tensor needs to be positive definite. 

Nowadays, sub-Riemannian geometry is of great interest in Lie group, geometric analysis, optimal control, and engineering communities. Here the sub-Riemannian structure refers to the fact that the metric on the sub-bundle is degenerate. In other words, the metric is only semi-positive definite, see \cite{BaudoinGarofalo09, BBG, baudoinwang2012, Baudoin2017, BGK, baudoin2019gamma, GL2016, jungel2016entropy, agrachev2009optimal, Feng, KL, barilari2019bakry}
and many references therein. This degeneration structure brings many difficulties in the study of concentration inequalities. In these settings, the classical curvature dimension bound often does not exist. A natural question arises. {\em Does there exist a sub-Riemannian Ricci curvature tensor and its lower bound?} 

To answer these questions, Gamma calculus, also named Bakry--{\'E}mery iterative calculus, are powerful methods in deriving Ricci curvature tensor and curvature dimension bound. In Riemannian settings, the calculus provides analytical ways to compute the curvature dimension bound. However, the classical Gamma calculus relies on the fact that the metric is positive definite, which does not cover the sub-Riemannian cases.  To handle this degeneracy issue, studies in \cite{ baudoin2016wasserstein,BaudoinGarofalo09, Baudoin2017} propose Gamma $z$ calculus, where $z$ represents given extra directions. This method derives the sub-Riemannian curvature dimension bound. However, the current method still requires a sub-Riemannian structure with a special choice of $z$, and satisfying commutative iteration properties. Besides, the method requires the fact that the sub-Riemannian structure is restricted up to the step $2$ condition. 

In this paper, following the generalized Gamma $z$ calculus in \cite{FL}, we present a generalized sub-Riemannian Ricci curvature and curvature dimension bound. This method transfers the commutative iteration condition into a new quantity in Gamma $z$ calculus (see formula \eqref{Gz2}). In a compact region, our method allows us to establish the analytical bound for sub-Riemannian log--Sobolev inequalities. More concretely, we formulate analytical curvature tensor for the Heisenberg group, displacement group ($\textbf{SE}(2)$) and Martinet flat sub-Riemannian structure with general weighted volumes. {These curvature tensor bounds implies the entropy dissipation results for sub-Riemannian drift diffusion processes on a compact spatial domain, in term of $L_1$ distance.} 

In literature, a weighted Ricci curvature tensor has also been studied in sub-Riemannian manifold \cite{barilari2019bakry}. Here \cite{barilari2019bakry} introduces the other generalization of Bakry-{\'E}mery curvature tensor using the associated Riccati equation, following which they prove sub-Riemannian comparison theorems. Compared to them, our method generalizes the Gamma calculus based on sub-Riemannian Laplacian operator, following which we prove log-Sobolev and Poincar{\'e} inequalities.  {Notice that, the log-Sobolev inequality on sub-Riemannian manifolds satisfying generalized curvature dimension inequality was established in \cite{baudoin2012log} by using the Gamma-z calculus in \cite{BaudoinGarofalo09}. However, this is based on the assumption of the commutative property of the Gamma ($\Gamma$) and Gamma-z ($\Gamma^z$) operator, i.e. transverse symmetry property of the sub-Riemannian manifolds, and the convergence was established in terms of the semi-group associated with the horizontal Laplace operator. A generalized version of log-Sobolev inequality for infinite dimensional Heisenberg group 
has been proved in \cite{inglis2009logarithmic} in the sense of coercive inequalities \cite{hebisch2009coercive}. Comparing to \cite{baudoin2012log, inglis2009logarithmic}, we can compute the precise log-Sobolev inequality constant. Furthermore, our method also implies the entropy dissipation rate of the transition density associated with the sub-Riemannian drift-diffusion process in $L_1$ distance. To our best knowledge, this is the first time of establishing entropy dissipation for sub-Riemannian drift--diffusions, which extends the results in Riemannian manifolds \cite{MV,OV}.}
\subsection{Main results}
We sketch one of our main results for the Martinet sub-Riemannian structure. Here the sub-Riemannian structure is defined on $\mathbb R^3$ through the kernel of one-form $\eta :=dz-\frac{1}{2}y^2dx.$ A global orthonormal basis for the horizontal distribution $\mathcal{H}$ adapt the following differential operator representation, in local coordinates $(x,y,z)$,  
\[
X=\frac{\pa}{\pa x}+\frac{y^2}{2}\frac{\pa}{\pa z},\quad Y=\frac{\pa}{\pa y}.
\]
The commutative relation gives
\beaa
[X,Y]=-yZ,\quad [Y,[X,Y]]=-Z,\quad \text{where}\quad Z=\frac{\pa}{\pa z}.
\eeaa
Here we note the horizontal and vertical direction by
\beaa
a^T=\begin{pmatrix}
    1&0&\frac{y^2}{2}\\
    0&1&0
\end{pmatrix},\quad z^{\ts}=(0,0,1)
\eeaa
Associated with this group, we consider the drift--diffusion process whose generator is given by
\begin{equation*}
Lf=\nabla\cdot(aa^{\ts}\nabla f)-\la a\otimes \nabla a,\nabla f\ra-\la \nabla V, aa^{\ts}\nabla f\ra,
\end{equation*}
where $V\in C^{\infty}(\mathbb{R}^3)$ is a given potential function and $a\otimes\nabla a$ is defined in section \ref{section2}. Here generator $L$ induces an invariant measure associated with probability density function by $\rho^*$. 
In this paper, we shall introduce a sub-Riemannian curvature tensor $\mathfrak{R}=\mathfrak{R}_{ab}^G+\mathfrak{R}_{zb}+\mathfrak{R}_{\rho^*}$; see details in section \ref{section3}. We show that when there exists a constant $\kappa>0$, such that
\beaa
\mathfrak{R}(\nabla f, \nabla f)\geq \kappa\la \nabla f, (aa^{\ts}+zz^{\ts})\nabla f\ra,\quad \textrm{for any $f\in C^{\infty}(\mathbb{R}^3)$},
\eeaa
then the following sub-Riemannian log--Sobolev inequality holds 
\beaa
\int \rho\log\frac{\rho}{\rho^*}dx\leq \frac{1}{2\kappa}\int \Big(\nabla\log\frac{\rho}{\rho^*}. (aa^{\ts}+zz^{\ts})\nabla\log\frac{\rho}{\rho^*}\Big)\rho dx,
\eeaa 
And the entropy dissipation result is established. In particular, 
{
    \begin{equation*}
 \|\rho(t,x)-\rho^*(x)\|_{L^1(\mathbb{R}^{n+m})}\leq \sqrt{\frac{1}{\kappa}\mathcal{I}_{a,z}(\rho_0\|\rho^*) }e^{-\kappa t},
\end{equation*}
where $\rho(t,x)$ is the probability density function for the sub-Riemannian diffusion process. 
}
We derive the algebraic condition for $\kappa$. Let 
\begin{equation*}
\kappa=\lambda_{\min}(\mathsf{A}),    
\end{equation*}
where matrix $\mathsf{A}=(\mathsf{A}_{ij})_{1\leq i,j\leq 3}\in\mathbb{R}^{3\times 3}$ has the following form
\beaa
\mathsf{A}_{11}&=&\Big(\frac{\pa^2 V}{\pa x\pa x}+y^2\frac{\pa^2 V}{\pa x\pa z}+\frac{y^4}{4}\frac{\pa^2 V}{\pa z\pa z} \Big)-y^2;\\
\mathsf{A}_{22}&=&\frac{\pa^2 V}{\pa y\pa y}-y^2;\quad \mathsf{A}_{33}=\frac{y^2}{2};  \\
\mathsf{A}_{12}&=&\mathsf{A}_{21}=\frac{y}{2}\frac{\pa V}{\pa z}+(\frac{\pa^2 V}{\pa x\pa y}+\frac{y^2}{2}\frac{\pa^2 V}{\pa y\pa z});\\
\mathsf{A}_{13}&=&\mathsf{A}_{31}=\frac{1}{2}-\frac{y}{2}\frac{\pa V}{\pa y}+\frac{1}{2}(\frac{\pa^2 V}{\pa x\pa z}+\frac{y^2}{2}\frac{\pa^2 V}{\pa z\pa z});\\
\mathsf{A}_{23}&=&\mathsf{A}_{32}=\frac{1}{2}y(a^{\ts}\nabla)_1V+\frac{1}{2}\frac{\pa^2 V}{\pa y\pa z}.
\eeaa
There exists a compact region in $\mathbb{R}^3$ with $\kappa=\lambda_{\min}(A)>0$. The extension of such a lower bound to the entire space and the explicit convergence rate analysis will be left for future studies.  

This paper is organized as follows. In section \ref{section2}, we briefly review the generalized Gamma $z$ calculus and its derivation by Lyapunov methods in density space. In section \ref{section3}, we briefly recall the analytical formulas of sub-Riemannian Ricci curvature by generalized Gamma $z$ calculus from \cite{FL}.  In section \ref{section4}, we present the main result of this paper, which includes several concrete examples, including the Heisenberg group, displacement group, and Martinet sub-Riemannian structure with any weighted volumes. We leave the technical proofs in the appendices. 

\section{Generalized Gamma $z$ calculus and entropy dissipation}\label{section2}
We briefly review the generalized Gamma $z$ calculus proposed in \cite{FL}.  Here we also review its derivation by Lyapunov methods in density space in subsection \ref{section2.1}, known as the entropy dissipation method. {We apply this entropy dissipation methods to derive various decay results, epsecially for the $L^1$ distance.}

Consider a degenerate drift diffusion process
\begin{equation}\label{SDE framework1}
dX_t=-a(X_t)a(X_t)^{\ts}\nabla V(X_t)dt+\sqrt{2}a(X_t)\circ dB_t,
\end{equation}
where $n, m\in Z_+$, $a\in C^{\infty}(\mathbb{R}^{(n+m)\times n})$ is a matrix function, $V\in C^{\infty}(\mathbb{R}^{n+m})$ is a vector function. Here $B_t$ is the standard Brownian motion in $\mathbb{R}^{n}$ and $\circ$ represents the Stratonovich integral of Brownian motion.  

 We notice that the choice of matrix $a$ is based on the sub--Riemannian metric for Euclidean space $(\hM=\mathbb{R}^{n+m}, ({aa^{\ts}})^\dd)$, where $\dd$ is the pseudo inverse operator. Here $a=(a_1,a_2,\cdots,a_n)$ with each $a_i, i=1\cdots,n$, as a $n+m$-dimensional column vector. {In particular, we denote $aa^{\ts}\nabla f$ (resp. $zz^{\ts}\nabla f$) as the horizontal (resp. vertical) gradient of function $f:\hR^{n+m}\rightarrow \hR$ in terms of Euclidean gradient $\nabla$ in $\hR^{n+m}.$ (See more details in \cite{FL}[Definition 2.5]).} We notice that when $a$ is an invertiable squared matrix, i.e. $m=0$ and vectors $a_1,\cdots a_n$ are linearly independent. Then $(\hM, (aa^{\ts})^{-1)})$ is a Riemannian manifold and  \eqref{SDE framework1} corresponds to the associated Riemannian drift diffusion process. And, for general degenerate matrix $a$, SDE \eqref{SDE framework1} can be viewed as a sub--Riemannian drift diffusion process. 
 
Throughout this paper, we assume that $\{a_1,a_2,\cdots,a_n\}$ satisfies the strong H\"ormander condition or bracket generating condition. Then from sub--Riemannian theory, there exists a unique and smooth solution for the density function of a process $X_t$ in \eqref{SDE framework1}. Our goal here is to study the convergence behavior of the drift-diffusion process \eqref{SDE framework1}.

We first review some background results for the invariant measure of SDE \eqref{SDE framework1}. 
\begin{lem}[Invariant measure]
Suppose SDE \eqref{SDE framework1} with $V=0$ is associated with a unique smooth symmetric invariant measure, then there exists a function $\Vol\in C^{\infty}(\mathbb{R}^{n+m})$, such that
\beaa 
a\otimes \nabla a=-aa^{\ts}\nabla \log \Vol.
\eeaa
Assume that the SDE \eqref{SDE framework1} has a smooth invariant probability measure $\rho^*\in C^{\infty}(\mathbb{R}^{n+m})$, then 
\begin{equation*}
\rho^*=\frac{1}{Z}e^{-V}\Vol,
\end{equation*}
where $Z$ is a normalization constant such that $Z=\int_{\mathbb{R}^{n+m}}e^{-V}\Vol dx<\infty$. 
\end{lem}

We next present the iterative Gamma $z$ calculus for the convergence behavior of SDE \eqref{SDE framework1}. To do so, we denote the generator $L$ of sub-Riemannian drift--diffusion process by
\begin{equation*}
Lf=\nabla\cdot(aa^{\ts}\nabla f)-\la a\otimes \nabla a,\nabla f\ra_{\mathbb{R}^{n+m}}-\la \nabla V, aa^{\ts}\nabla f\ra_{\mathbb{R}^{n+m}},
\end{equation*}
where $f\in C^{\infty}(\mathbb{R}^{n+m})$ and 
$$a\otimes \nabla a= \Big((a\otimes \nabla a)_{\hat k}\Big)_{\hat k=1}^{n+m} =\Big(\sum_{k=1}^n\sum_{k'=1}^{n+m}a_{\hat kk}\frac{\partial}{\partial x_{k'}}a_{k'k}\Big)_{\hat k=1}^{n+m}\in \mathbb{R}^{n+m}.$$ 
\begin{defn}[Generalized Gamma $z$ calculus]\label{defn:tilde gamma 2 znew} 
Construct a smooth matrix function $z\in \mathbb{R}^{(n+m)\times m}$. Denote Gamma one bilinear forms $\Gamma_{1},\Gamma_{1}^z\colon C^{\infty}(\mathbb{R}^{n+m})\times C^{\infty}(\mathbb{R}^{n+m})\rightarrow C^{\infty}(\mathbb{R}^{n+m}) $ as 
\begin{equation*}
\Gamma_{1}(f,g)=\la a^{\ts}\nabla f, a^{\ts}\nabla g\ra_{\hR^n},\quad \Gamma_{1}^z(f,g)=\la z^{\ts}\nabla f, z^{\ts}\nabla g\ra_{\hR^m}.
\end{equation*}
Define Gamma two bilinear forms $\Gamma_{2},\Gamma_2^{z,\rho^*}\colon C^{\infty}(\mathbb{R}^{n+m})\times C^{\infty}(\mathbb{R}^{n+m})\rightarrow C^{\infty}(\mathbb{R}^{n+m}) $ by
\begin{equation*}
\Gamma_{2}(f,g)=\frac{1}{2}\Big[L\Gamma_{1}(f,g)-\Gamma_{1}(Lf, g)-\Gamma_{1}(f, Lg)\Big],    
\end{equation*}
and 
\bea
\label{old gamma two z}
\Gamma_2^{z,\rho^*}(f,g)&=&\quad\frac{1}{2}\Big[L\Gamma_1^{z}(f,g)-\Gamma_{1}^z(Lf,g)-\Gamma_{1}^z(f, Lg)\label{Gz1}\Big]\\
&& \label{new term}+\div^{{\rho^*}}_z\Big(\Gamma_{1,\nabla(aa^{\ts})}(f,g )\Big)-\div^{{\rho^*} }_a\Big(\Gamma_{1,\nabla(zz^{\ts})}(f,g )\Big)\label{Gz2}.
  \eea 
Here $\div^{{\rho^*} }_a$, $\div^{{\rho^*} }_z$ are divergence operators defined by: 
\begin{equation*}
\div^{{\rho^*}}_a(F)=\frac{1}{{\rho^*} }\nabla\cdot({\rho^*} aa^{\ts} F), \quad\div^{{\rho^*} }_z(F)=\frac{1}{{\rho^*} }\nabla\cdot({\rho^*} zz^{\ts}F),
\end{equation*}
for any smooth vector field $F\in \mathbb{R}^{n+m}$, and $\Gamma_{1, \nabla (aa^{\ts})}$, $\Gamma_{1, \nabla (zz^{\ts})}$ are vector Gamma one bilinear forms defined by 
\beaa
\Gamma_{1,\nabla(aa^{\ts)}}(f,g)&=&\la \nabla f,\nabla(aa^{\ts})\nabla g\ra=(\la \nabla f,\frac{\partial}{\partial x_{\hat k}}(aa^{\ts})\nabla g\ra)_{\hat k=1}^{n+m},\\
\Gamma_{1,\nabla(zz^{\ts)}}(f,g)&=&\la \nabla f,\nabla(aa^{\ts})\nabla g\ra=(\la \nabla f,\frac{\partial}{\partial x_{\hat k}}(zz^{\ts})\nabla g\ra)_{\hat k=1}^{n+m},
\eeaa
with 
    \beaa
\div^{{\rho^*} }_z\Big(\Gamma_{\nabla(aa^{\ts})}f,g \Big)&=&\frac{\nabla\cdot(zz^{\ts}{\rho^*} \la \nabla f,\nabla(aa^{\ts})\nabla g\ra ) }{{\rho^*} },\\
\div^{{\rho^*} }_a\Big(\Gamma_{\nabla(zz^{\ts})}f,g \Big)&=&\frac{\nabla\cdot(aa^{\ts}{\rho^*}  \la \nabla f,\nabla(zz^{\ts})\nabla g\ra) }{{\rho^*} }.
    \eeaa
\end{defn}

Given the generalized Gamma $z$ calculus, we are ready to derive the log-Sobolev inequality in sub-Riemannian manifold. Denote the Kullback--Leibler divergence by 
\begin{equation*}
\mathrm{D}_{\mathrm{KL}}(\rho\|\rho^*)=\int_{\mathbb{R}^{n+m}}\rho\log\frac{\rho}{\rho^*}dx, 
\end{equation*}
and the $a,z$--Fisher information functional 
\begin{equation*}
\mathrm{I}_{a,z}(\rho\|\rho^*)=\int_{\mathbb{R}^{n+m}}\Big(\nabla\log\frac{\rho}{\rho^*}, (aa^{\ts}+zz^{\ts})\nabla\log\frac{\rho}{\rho^*}\Big)\rho dx.
\end{equation*}
{In particular, $\mathrm{D}_{\mathrm{KL}}(\rho\|\rho^*)$ and $\mathrm{I}_{a,z}(\rho\|\rho^*)$ vanish as $\rho=\rho^*$.}
\begin{prop}[z--log--Sobolev inequalities]\label{prop1.4}
Suppose there exists a constant $\kappa>0$, such that 
\begin{equation}\label{z-log}
\Gamma_2( f,f)+\Gamma_2^{z,\rho^*}(f, f)\succeq \kappa (\Gamma_1(f, f)+\Gamma_1^z(f,f)),\quad \textrm{for any $f\in C^{\infty}(\mathbb{R}^{n+m})$}.
\end{equation}
Then the z-log-Sobolev inequalities (zLSI) holds: For any smooth density $\rho$, then
\begin{equation*}
\mathrm{D}_{\mathrm{KL}}(\rho\|\rho^*) \leq \frac{1}{2\kappa}\mathrm{I}_{a,z}(\rho\|\rho^*)\qquad(\textrm{zLSI})
\end{equation*} 
\end{prop}
\begin{rem}
We notice that formula \eqref{Gz1} was firstly introduced by \cite{BaudoinGarofalo09}. It contains a commutative iteration assumption  
$$\Gamma_{1}(f,\Gamma_{1}^z(f,f))=\Gamma_{1}^z(f,\Gamma_{1}(f,f)).$$
Here we introduce an additional term \eqref{new term}, which overcomes and removes this assumption. 
In fact, in the paper, we show that formula \eqref{new term} is exactly the new bilinear form for this assumption by the weak form in probability density space. See details in \cite{FL}.
\end{rem}
\begin{rem}
{We comment that the derived curvature tensor works on any compact region (where $\kappa >0$) with sub-Riemannian metric. In this case, our curvature is also useful in establishing the convergence rate of sub-Riemannian drift diffusion process defined in a compact region.} 
\end{rem}
\begin{rem}
It is also worth mentioning that many sub-Riemannian manifolds are non-compact. 
Hence there may not exist a positive constant $\kappa$ for both classical $\Gamma_1$ and $\Gamma^z_1$ directions in the non-compact domain. The non-compactness of the domain brings additional difficulties. To prove the associated inequalities in this case, we need to extend the result derived in \cite{baudoin2012log, wang1997logarithmic}. This is a direction for future works.  
\end{rem}

\subsection{{Ricci curvature and entropy dissipation}}\label{section2.1}
In this subsection, we apply the generalized Gamma $z$ calculus to study the convergence rate of sub-Riemannian drift diffusion process. 

{
\begin{thm}[Entropy dissipation]\label{thm11}
Suppose that there exists a constant $\kappa> 0$, satisfying \eqref{z-log}. Then the following dissipation result hold. Denote $\rho_t$ as the probability density function of sub-Riemannian drift diffusion process \eqref{SDE framework1}. Then 
\begin{itemize}
\item[(i)]
\begin{equation*}
    \mathrm{D}_{\mathrm{KL}}(\rho_t\|\rho^*)\leq \frac{1}{2\kappa}e^{-2\kappa t}\mathcal{I}_{a,z}(\rho_0\|\rho^*).
\end{equation*}
\item[(ii)]
\begin{equation*}
 \|\rho(t,x)-\rho^*(x)\|_{L^1(\mathbb{R}^{n+m})}\leq \sqrt{\frac{1}{\kappa}\mathcal{I}_{a,z}(\rho_0\|\rho^*) }e^{-\kappa t}.
\end{equation*}    
\end{itemize}
\end{thm}}
{We formulate the proof in the following orders. This proof explains the derivation of generalized Gamma $z$ calculus. 

Our method is based on the Lyapunov method in density space.} We first formulate the Fokker-Planck equation of SDE \eqref{SDE framework1}:
\begin{equation}\label{FPE}
\begin{split}
\partial_t\rho
=&\nabla\cdot(aa^{\ts}\nabla\rho)+\nabla\cdot(\rho aa^{\ts}\nabla V)+\nabla\cdot(\rho a\otimes \nabla a)\\
=&\nabla\cdot(\rho aa^{\ts}\nabla\log\rho)+\nabla\cdot(\rho aa^{\ts}\nabla V)-\nabla\cdot(\rho aa^{\ts} \nabla \log\Vol)\\
=&\nabla\cdot(\rho aa^{\ts}\nabla\log\frac{\rho}{\rho^*}),
\end{split}
\end{equation}
where we use the facts that $\nabla\rho=\rho\nabla\log\rho$ and $aa^{\ts}\nabla \log\Vol=a\otimes \nabla a$ in the second equality. 

We next construct the following Lyapunov functional for equation \eqref{FPE}. Denote $\delta \mathrm{D}_{\mathrm{KL}}=\log\frac{\rho}{\rho^*}+1$, where $\delta$ is the $L^2$ first variation. 
Then
\begin{equation*}
\mathrm{I}_{a}(\rho)= \int \Gamma_{1}(\delta  \mathrm{D}_{\mathrm{KL}}, \delta  \mathrm{D}_{\mathrm{KL}} )dx=\int \Big(\nabla\log\frac{\rho}{\rho^*}, aa^{\ts}\nabla\log\frac{\rho}{\rho^*}\Big) dx,
\end{equation*}
and
\begin{equation*}
\mathrm{I}_{z}(\rho)= \int \Gamma_{1}^z(\delta  \mathrm{D}_{\mathrm{KL}}, \delta  \mathrm{D}_{\mathrm{KL}} )dx =\int \Big(\nabla\log\frac{\rho}{\rho^*}, zz^{\ts}\nabla\log\frac{\rho}{\rho^*}\Big) dx.
\end{equation*}
With this notation, we have
 \begin{equation*}
\mathrm{I}_{a,z}(\rho):=\mathrm{I}_a(\rho)+\mathrm{I}_z(\rho)=\int \Big(\Gamma_{1}(\delta  \mathrm{D}_{\mathrm{KL}}, \delta  \mathrm{D}_{\mathrm{KL}} )+ \Gamma_{1}^z(\delta  \mathrm{D}_{\mathrm{KL}}, \delta  \mathrm{D}_{\mathrm{KL}} )\Big)dx.
\end{equation*}
We next prove the following proposition.
\begin{prop}\label{prop1}
\begin{equation*}
\frac{d}{dt}\mathrm{I}_{a,z}(\rho_t)=- 2\int \Big(\Gamma_2(\delta \mathrm{D}_{\mathrm{KL}}, \delta\mathrm{D}_{\mathrm{KL}})+ \Gamma_{2}^{z,\rho^*}(\delta \mathrm{D}_{\mathrm{KL}}, \delta\mathrm{D}_{\mathrm{KL}}))\Big)\rho_t dx, 
\end{equation*}
\end{prop}
\begin{proof}
The proof has been shown in \cite{FL}, whose motivation is presented in \cite{LiG, Li2019_diffusion}. For the self-contained purpose, we outline the major derivation below. 
Given any smooth matrix function $c\in C^{\infty}(\mathbb{R}^{(n+m)\times (n+m)})$, we use the convention that the weighted Laplacian operator is denoted by,
\begin{equation*}
\Delta_{c}=\nabla\cdot(c\nabla).
\end{equation*}
Later on, we will apply different functions of $c$. Then 
\begin{equation*}
    \begin{split}
    \frac{d}{dt}\mathrm{I}_{a,z}(\rho_t)=&\frac{d}{dt}\int (\delta\mathrm{D}_{\mathrm{KL}}, (-\Delta_{\rho_t(aa^{\ts}+zz^{\ts})}\delta\mathrm{D}_{\mathrm{KL}}) dx \\
    =&-2\int \delta^2\DKL(\Delta_{\rho_t(aa^{\ts}+zz^{\ts})}\delta \DKL) (\Delta_{\rho_t(aa^{\ts})}\delta \DKL)dx\\
    &+\int (\nabla\delta\DKL, (aa^{\ts}+zz^{\ts})\nabla\delta\DKL)\Delta_{\rho_t(aa^{\ts})}\delta\DKL dx\\
        =&-2\Big\{\int \delta^2\DKL(\Delta_{\rho_t(aa^{\ts})}\delta \DKL) (\Delta_{\rho_t(aa^{\ts})}\delta \DKL)dx\\
    &\qquad-\frac{1}{2}\int (\nabla\delta\DKL, (aa^{\ts})\nabla\delta\DKL)\Delta_{\rho_t(aa^{\ts})}\delta\DKL dx\Big\}\\
    =&-2\Big\{\int \delta^2\DKL(\Delta_{\rho_t(zz^{\ts})}\delta \DKL) (\Delta_{\rho_t(aa^{\ts})}\delta \DKL)dx\\
    &\qquad-\frac{1}{2}\int (\nabla\delta\DKL, (zz^{\ts})\nabla\delta\DKL)\Delta_{\rho_t(aa^{\ts})}\delta\DKL dx\Big\}\\
    =&- 2\int \Big(\Gamma_2(\delta \mathrm{D}_{\mathrm{KL}}, \delta\mathrm{D}_{\mathrm{KL}})+ \Gamma_{2}^{z,\rho^*}(\delta \mathrm{D}_{\mathrm{KL}}, \delta\mathrm{D}_{\mathrm{KL}})\Big)\rho_t dx,
    \end{split}    
\end{equation*}
where $\delta^2\DKL=\frac{1}{\rho}$ and the last equality follows the routine calculations shown in \cite{FL}.
\end{proof}

We are ready to prove the convergence properties in Theorem \ref{thm11} and functional inequalities for degenerate drift-diffusion processes. 

\begin{proof}[Proof of Proposition \ref{prop1.4}]
Our result follows the Lyapunov methods. Given a Lyapunov function $\mathrm{I}_{a,z}$, along the Fokker--Planck equation \eqref{FPE}, we have
\begin{equation*}
\frac{d}{dt}\mathrm{I}_{a,z}(\rho_t)=-2\int \Big(\Gamma_2(\delta\mathrm{D}_{\mathrm{KL}}, \delta\mathrm{D}_{\mathrm{KL}}) +\Gamma_2^{z,\rho^*}(\delta\mathrm{D}_{\mathrm{KL}},\delta\mathrm{D}_{\mathrm{\mathrm{KL}}}) \Big)\rho_t dx.   \end{equation*}
If $\Gamma_2(f,f)+\Gamma_2^{z,{\rho^*}}(f,f)\succeq \kappa(\Gamma_1(f,f)+\Gamma_1^z(f,f))$ with $\kappa\geq 0$, then 
\begin{equation*}\label{iq}
    \frac{d}{dt}\mathrm{I}_{a,z}(\rho_t)\leq -2\kappa \mathrm{I}_{a,z}(\rho_t). 
\end{equation*}
We next show the log--Sobolev inequality. Notice the fact that 
\begin{equation*}
-\frac{d}{dt}\mathrm{D}_{\mathrm{KL}}(\rho_t)=\mathrm{I}_a(\rho_t)\leq \mathrm{I}_{a,z}(\rho_t),
\end{equation*}
then \eqref{iq} implies the fact that
\begin{equation*}
\begin{split}
    -\mathrm{I}_{a,z}(\rho)=&\int_0^{\infty}\frac{d}{dt}\mathrm{I}_{a,z}(\rho_t)dt\\
    \leq & -2\kappa \int_0^{\infty}\mathrm{I}_{a,z}(\rho_t)dt=-2\kappa \int_0^\infty\Big(\mathrm{I}_a(\rho_t)+\mathrm{I}_z(\rho_t)\Big)dt\\
    \leq &-2\kappa \int_0^\infty\mathrm{I}_a(\rho_t)dt\\
     =&-2\kappa  \int_0^{\infty}(-\frac{d}{dt}\mathrm{D}_{\mathrm{KL}}(\rho_t))dt\\
    =&-2\kappa\mathrm{D}_{\mathrm{KL}}(\rho),
    \end{split}
\end{equation*}
where we denote $\rho_0=\rho$. Thus $\mathrm{I}_{a,z}(\rho)\geq 2\kappa \mathrm{D}(\rho)$, which finishes the proof.
\end{proof}

We are now ready to prove the main result of this paper. 

\begin{proof}[Proof of Theorem \ref{thm11}]
The exponential decay of KL divergence follows from the decay of relative Fisher information functional $\mathcal{I}_{a,z}$ and z-log-Sobolev inequality. From the $z$-log-Sobolev inequality, 
we have 
\begin{equation*}
    \mathrm{D}_{\mathrm{KL}}(\rho_t\|\rho^*)\leq \frac{1}{2\kappa}\mathcal{I}_{a,z}(\rho_t\|\rho^*)\leq \frac{1}{2\kappa}e^{-2\kappa t}\mathcal{I}_{a,z}(\rho_0\|\rho^*).
\end{equation*}
Hence we prove (i). We next apply the Pinsker's inequality, i.e the inequality between KL divergence and $L^1$ distance. Since 
\begin{equation*}
\|\rho_t-\rho^* \|_{L^1(\mathbb{R}^{n+m})}\leq \sqrt{2\mathrm{D}_{\mathrm{KL}}(\rho_t\|\rho^*)}.
\end{equation*}
Using (i), we finish the proof of (ii).
\end{proof}
\begin{rem}
Our proof follows from the facts. 
\begin{equation*}
\textrm{Gamma $z$ calculus with lower bound $\kappa$} \Rightarrow \textrm{$\mathcal{I}_{a,z}(\rho)$ decay}\Rightarrow \textrm{zLSI}\Rightarrow \textrm{$\mathrm{D}_{KL}$ decay}\Rightarrow \textrm{$L_1$ decay.}      
\end{equation*}
Our definition of Gamma $z$ calculus and its corresponding Ricci curvature lower bound can still be used to formulate the decay rate for the densities of sub-Riemannian SDEs. These decay rates can recover the classical results in Riemannian manifold. 
\end{rem}

\section{sub-Riemannian Ricci curvature}\label{section3}
In this section, we demonstrate the generalized Gamma $z$ calculus in bilinear forms, from which we derive the sub-Riemannian Ricci curvature. 
This can be viewed as a Bochner's formula associated with $z$--direction. We call it $z$--Bochner's formula. They are extensions of the corresponding ones in Riemannian manifolds. 

In \cite{FL}, we consider the drift--diffusion process 
\begin{equation*}
dX_t=b(X_t)dt+a(X_t)\circ dB_t,   
\end{equation*}
where $b$ is a given smooth drift direction. 
In this paper, we simply assume
\begin{equation*}
b=-\frac{1}{2}aa^{\ts}\nabla V.
\end{equation*}
Compared to the result of \cite{FL}, this choice of $b$ results at the same generator of drift diffusion processes up to a scale of $2$, which will not take the effect in the Gamma calculus. 

We are now ready to present the generalized Gamma $z$ calculus as follows. 
\begin{nota}\label{notation} For any smooth function $f:\hR^{n+m}\rightarrow \hR$, denoted as $f\in C^{\infty}(\hR^{n+m})$, and  $(n+m)\times n$ matrix $a$, we define matrix $Q$ as
\beaa
Q=\begin{pmatrix}
	a^{\ts}_{11}a^{\ts}_{11}&\cdots &a^{\ts}_{1(n+m)}a^{\ts}_{1(n+m)}\\
	\cdots & a^{\ts}_{i\hat i}a^{\ts}_{k\hat k}&\cdots \\
	a^{\ts}_{n 1}a^{\ts}_{n 1} &\cdots &a^{\ts}_{n(n+m)}a^{\ts}_{n(n+m)} 
	\end{pmatrix}\in \mathbb{R}^{n^2\times(n+m)^2},
\eeaa
with $Q_{ik\hat i\hat k}=a^{\ts}_{i\hat i}a^{\ts}_{k\hat k}$. {More precisely, for each row (resp. column) of $Q$, the row (resp. column)  indices of $Q_{ik\hat i\hat k}$ following $\sum_{i=1}^n\sum_{k=1}^n$  (resp. $\sum_{\hat i=1}^{n+1}\sum_{\hat k=1}^{n+m}$).} For $(n+m)\times m$ matrix $z$, we define matrix $P$ as 
\beaa
P=\begin{pmatrix}
	z^{\ts}_{11}a^{\ts}_{11}&\cdots &z^{\ts}_{1(n+m)}a^{\ts}_{1(n+m)}\\
	\cdots & z^{\ts}_{i\hat i}a^{\ts}_{k\hat k}&\cdots \\
	z^{\ts}_{m 1}a^{\ts}_{n 1} &\cdots &z^{\ts}_{m(n+m)}a^{\ts}_{n(n+m)} 
	\end{pmatrix}\in \mathbb{R}^{(nm)\times(n+m)^2},
\eeaa
with $P_{ik\hat i\hat k}=z^{\ts}_{i\hat i}a^{\ts}_{k\hat k}$. 
For any $\hat i,\hat k,\hat j=1,\cdots,n+m$ and $i,k=1,\cdots,n$ $($or $1,\cdots, m)$.
 We denote $C$ as a $(n+m)^2\times 1$ dimensional vector with components defined as 
 \beaa 
C_{\hat i\hat k}=\left[\sum_{i,k=1}^n\sum_{i'=1}^{n+m}\left(a^{\ts}_{i\hat i}a^{\ts}_{i i'} ( \frac{\partial a^{\ts}_{k\hat k}}{\partial x_{i'}})(a^{\ts}\nabla)_kf -  a^{\ts}_{k i'}a^{\ts}_{i\hat k}\frac{\partial a^{\ts}_{i \hat i}}{\partial x_{i'}}   (a^{\ts}\nabla)_kf\right)\right].
\eeaa
Here we keep the notation $(a^{\ts}\nabla)_kf=\sum_{k'=1}^{n+m}a^{\ts}_{kk'}\frac{\pa f}{\pa x_{k'}}$. Denote $D$ as a $n^2\times 1$ dimensional vector with components defined as
\beaa
D_{ik}=\sum_{\hat i,\hat k=1}^{n+m}a^{\ts}_{i\hat i}\frac{\partial a^{\ts}_{k\hat k}}{\partial x_{\hat i}}\frac{\partial f}{\partial x_{ \hat k}}.
\eeaa
Denote $F$ as a $(n+m)^2\times 1$ dimensional vector with components defined as
 \beaa 
F_{\hat i\hat k}= \left[\sum_{i=1}^n\sum_{k=1}^m\sum_{i'=1}^{n+m}\left( a^{\ts}_{i\hat i}a^{\ts}_{i i'} ( \frac{\partial z^{\ts}_{k\hat k}}{\partial x_{i'}}) (z^{\ts}\nabla) - z^{\ts}_{k i'}a^{\ts}_{i\hat k}\frac{\partial a^{\ts}_{i \hat i}}{\partial x_{i'}}   (z^{\ts}\nabla)_kf\right)\right].
\eeaa
 Denote $E$ as a $(n\times m)\times 1$ dimensional vector with components defined as
 \beaa
E_{ik}=\sum_{\hat i,\hat k=1}^{n+m}a^{\ts}_{i\hat i}\frac{\partial z^{\ts}_{k\hat k}}{\partial x_{\hat i}}\frac{\partial f}{\partial x_{ \hat k}}. 
\eeaa  
Denote $G$ as a $(n+m)^2\times 1$ dimensional vector. In local coordinates, we have 
\beaa
G_{\hat i\hat j}&=& \sum_{i=1}^n\sum_{j=1}^m \sum_{j',i'=1}^{n+m}\left[\left( z^{\ts}_{j\hat j} z^{\ts}_{j j'} \frac{\pa a^{\ts}_{i\hat i}}{\pa x_{ j'}}a^{\ts}_{ii'}\frac{\pa f}{\pa x_{i'}}+z^{\ts}_{j \hat j} z^{\ts}_{j j'} \frac{\pa a^{\ts}_{i i'}}{\pa x_{ j'}}\frac{\pa f}{\pa x_{ i'}}a^{\ts}_{i\hat i}\right)\right. \\
&&\left.\quad\quad\quad\quad\quad\quad-\left(a^{\ts}_{i\hat i}a^{\ts}_{i i'} \frac{\pa z^{\ts}_{j\hat j}}{\pa x_{ i'}}z^{\ts}_{jj'}\frac{\pa f}{\pa x_{j'}} +a^{\ts}_{i\hat i}a^{\ts}_{ii'} \frac{\pa z^{\ts}_{j j'}}{\pa x_{ i'}}\frac{\pa f}{\pa x_{ j'}}z^{\ts}_{j\hat j}\right)\right].\nonumber
\eeaa
Denote $X$ as the vectorization of the Hessian matrix of function $f$, 
\beaa
X=\begin{pmatrix}
	\frac{\pa^2 f}{\pa x_{1} \pa x_{1}}\\
	\cdots\\
		\frac{\pa^2 f}{\pa x_{\hat i} \pa x_{\hat k}}\\
		\cdots \\
		\frac{\pa^2 f}{\pa x_{n+m} \pa x_{n+m}}
	\end{pmatrix}\in \mathbb{R}^{(n+m)^2\times 1}.
\eeaa
\end{nota}{}

\begin{assum}\label{assumption:main result}
	Assume that there exists vectors $\Lambda_1, \Lambda_2\in \mathbb{R}^{(n+m)^2\times 1}$ such that
	\beaa
	(Q^{\ts}Q\Lambda_1+P^{\ts}P \Lambda_2)^{\ts}X&=&(F+C+G+Q^{\ts}D+P^{\ts}E)^{\ts}X.
	\eeaa
\end{assum}
\begin{thm}[$z$--Bochner's formula]\label{thm1}
If Assumption \ref{assumption:main result} is satisfied, then the following decomposition holds
	\beaa
\Gamma_{2}(f,f)+\Gamma_2^{z,\rho^*}(f,f)&=&|\mathfrak{Hess}_{a,z}^{G}f|^2
+\mathfrak{R}^{G}_{ab}(\nabla f,\nabla f) +\mathfrak{R}_{zb}(\nabla f,\nabla f)+\mathfrak{R}_{\rho^*}(\nabla f,\nabla f),
\eeaa 
where we define 
\beaa
|\mathfrak{Hess}_{a,z}^{G}f|^2=[X+\Lambda_1]^{\ts}Q^{\ts}Q[X+\Lambda_1]+[X+\Lambda_2]^{\ts}P^{\ts}P[X+\Lambda_2],
\eeaa
and denote the following three tensors, such that 
\beaa
\mathfrak{R}^{G}_{ab}(\nabla f,\nabla f)=\mathfrak{R}^{G}(\nabla f, \nabla f)+\mathfrak{R}_{ab}(\nabla f, \nabla f),
\eeaa
with
\beaa
\mathfrak{R}^{G}(\nabla f,\nabla f)=-\Lambda_1^{\ts}Q^{\ts}Q\Lambda_1-\Lambda_2^{\ts}P^{\ts}P\Lambda_2+D^{\ts}D+E^{\ts}E,
\eeaa
and 
\beaa
\mathfrak{R}_{ab}(\nabla f,\nabla f)
=&&\sum_{i,k=1}^n\sum_{i',\hat i,\hat k=1}^{n+m} \la a^{\ts}_{ii'} (\frac{\partial a^{\ts}_{i \hat i}}{\partial x_{i'}} \frac{\partial a^{\ts}_{k\hat k}}{\partial x_{\hat i}}\frac{\partial f}{\partial x_{\hat k}}) ,(a^{\ts}\nabla)_kf\ra_{\hR^n}  \nonumber\\
&&+\sum_{i,k=1}^n\sum_{i',\hat i,\hat k=1}^{n+m} \la a^{\ts}_{ii'}a^{\ts}_{i \hat i} (\frac{\partial }{\partial x_{i'}} \frac{\partial a^{\ts}_{k\hat k}}{\partial x_{\hat i}})(\frac{\partial f}{\partial x_{\hat k}}) ,(a^{\ts}\nabla)_kf\ra_{\hR^n} \nonumber\\
&&-\sum_{i,k=1}^n\sum_{i',\hat i,\hat k=1}^{n+m} \la a^{\ts}_{k\hat k}\frac{\partial a^{\ts}_{ii'}}{\partial x_{\hat k}} \frac{\partial a^{\ts}_{i \hat i}}{\partial x_{i'}} \frac{\partial f}{\partial x_{\hat i}})  ,(a^{\ts}\nabla)_kf\ra_{\hR^n}  \nonumber\\
&&-\sum_{i,k=1}^n\sum_{i',\hat i,\hat k=1}^{n+m} \la a^{\ts}_{k\hat k} a^{\ts}_{ii'} (\frac{\partial }{\partial x_{\hat k}} \frac{\partial a^{\ts}_{i \hat i}}{\partial x_{i'}}) \frac{\partial f}{\partial x_{\hat i}}  ,(a^{\ts}\nabla)_kf\ra_{\hR^n},  \nonumber\\
&&-2\sum_{i=1}^n \sum_{\hat i,\hat k=1}^{n+m}\la (a^T_{i\hat i}\frac{\pa b_{\hat k}}{\pa x_{\hat i}}\frac{\pa f}{\pa x_{\hat k}}-b_{\hat k}\frac{\pa a^T_{i\hat i}}{\pa x_{\hat k}} \frac{\pa f}{\pa x_{\hat i}} ),(a^T\nabla f)_i\ra_{\hR^n}\nonumber.
\eeaa
In addition, 
\beaa
\mathfrak{R}_{zb}(\nabla f,\nabla f)&=&\sum_{i=1}^n\sum_{k=}^m\sum_{i',\hat i,\hat k=1}^{n+m} \la a^{\ts}_{ii'} (\frac{\partial a^{\ts}_{i \hat i}}{\partial x_{i'}} \frac{\partial z^{\ts}_{k\hat k}}{\partial x_{\hat i}}\frac{\partial f}{\partial x_{\hat k}}) ,(z^{\ts}\nabla)_kf\ra_{\hR^m}\nonumber \\
	&&+\sum_{i=1}^n\sum_{k=}^m\sum_{i',\hat i,\hat k=1}^{n+m} \la a^{\ts}_{ii'}a^{\ts}_{i \hat i} (\frac{\partial }{\partial x_{i'}} \frac{\partial z^{\ts}_{k\hat k}}{\partial x_{\hat i}})(\frac{\partial f}{\partial x_{\hat k}}) ,(z^{\ts}\nabla)_kf\ra_{\hR^m} \nonumber \\
	&&-\sum_{i=1}^n\sum_{k=1}^m\sum_{i',\hat i,\hat k=1}^{n+m} \la z^{\ts}_{k\hat k}\frac{\partial a^{\ts}_{ii'}}{\partial x_{\hat k}} \frac{\partial a^{\ts}_{i \hat i}}{\partial x_{i'}} \frac{\partial f}{\partial x_{\hat i}})  ,(z^{\ts}\nabla)_kf\ra_{\hR^m} \nonumber \\
&&-\sum_{i=1}^n\sum_{k=1}^m\sum_{i',\hat i,\hat k=1}^{n+m} \la z^{\ts}_{k\hat k} a^{\ts}_{ii'} (\frac{\partial }{\partial x_{\hat k}} \frac{\partial a^{\ts}_{i \hat i}}{\partial x_{i'}}) \frac{\partial f}{\partial x_{\hat i}}  ,(z^{\ts}\nabla)_kf\ra_{\hR^m}\nonumber \\
&&-2\sum_{i=1}^m \sum_{\hat i,\hat k=1}^{n+m}\la (z^T_{i\hat i}\frac{\pa b_{\hat k}}{\pa x_{\hat i}}\frac{\pa f}{\pa x_{\hat k}}-b_{\hat k}\frac{\pa z^T_{i\hat i}}{\pa x_{\hat k}} \frac{\pa f}{\pa x_{\hat i}} ),(z^T\nabla f)_i\ra_{\hR^m},\nonumber
\eeaa 
and 
\beaa
\mathfrak{R}_{\rho^*}(\nabla f,\nabla f)&=&2\sum_{k=1}^m \sum_{i=1}^n\sum_{k',\hat k,\hat i,i'=1}^{n+m}\left[\frac{\pa }{\pa x_{k'}} z^{\ts}_{kk'} z^{\ts}_{k\hat k}\frac{\pa}{\pa x_{\hat k}}a^{\ts}_{i\hat i}\frac{\pa f}{\pa x_{\hat i}}a^{\ts}_{ii'}\frac{\pa f}{\pa x_{i'}}\right]\\
 &&+2\sum_{k=1}^m \sum_{i=1}^n\sum_{k',\hat k,\hat i,i'=1}^{n+m}\left[z^{\ts}_{kk'}\frac{\pa }{\pa x_{k'}} z^{\ts}_{k\hat k} \frac{\pa}{\pa x_{\hat k}}a^{\ts}_{i\hat i}\frac{\pa f}{\pa x_{\hat i}}a^{\ts}_{ii'}\frac{\pa f}{\pa x_{i'}} \right.\nonumber\\
 &&\quad\quad\quad\quad\quad\quad\quad\quad+z^{\ts}_{kk'} z^{\ts}_{k\hat k} \frac{\pa^2}{\pa x_{k'}\pa x_{\hat k}}a^{\ts}_{i\hat i}\frac{\pa f}{\pa x_{\hat i}}a^{\ts}_{ii'}\frac{\pa f}{\pa x_{i'}}\nonumber \\
&&\left.\quad\quad\quad\quad\quad\quad\quad\quad+z^{\ts}_{kk'} z^{\ts}_{k\hat k} \frac{\pa}{\pa x_{\hat k}}a^{\ts}_{i\hat i}\frac{\pa f}{\pa x_{\hat i}}\frac{\pa }{\pa x_{k'}}a^{\ts}_{ii'}\frac{\pa f}{\pa x_{i'}} \right].\nonumber \\
&&+2\sum_{k=1}^m \sum_{i=1}^n\sum_{\hat k,\hat i,i'=1}^{n+m}(z^{\ts}\nabla\log\rho^*)_k \left[ z^{\ts}_{k\hat k}\frac{\pa}{\pa x_{\hat k}}a^{\ts}_{i\hat i}\frac{\pa f}{\pa x_{\hat i}}a^{\ts}_{ii'}\frac{\pa f}{\pa x_{i'}} \right] \nonumber\\
&&-2\sum_{j=1}^m\sum_{l=1}^n\sum_{ l',\hat l,\hat j,j'=1}^{n+m}\left[\frac{\pa }{\pa x_{l'}} a^{\ts}_{ll'} a^{\ts}_{l\hat l} \frac{\pa}{\pa x_{\hat l}}z^{\ts}_{j\hat j}\frac{\pa f}{\pa x_{\hat j}}z^{\ts}_{jj'}\frac{\pa f}{\pa x_{j'}} \right] \nonumber\\
 &&- 2\sum_{j=1}^m\sum_{l=1}^n\sum_{ l',\hat l,\hat j,j'=1}^{n+m}\left[ a^{\ts}_{ll'}\frac{\pa }{\pa x_{l'}}a^{\ts}_{l\hat l} \frac{\pa}{\pa x_{\hat l}}z^{\ts}_{j\hat j}\frac{\pa f}{\pa x_{\hat j}}z^{\ts}_{jj'}\frac{\pa f}{\pa x_{j'}}  \right.\nonumber\\
 &&\quad\quad\quad\quad\quad\quad\quad\quad+a^{\ts}_{ll'}a^{\ts}_{l\hat l} \frac{\pa^2}{\pa x_{l'}\pa x_{\hat l}}z^{\ts}_{j\hat j}\frac{\pa f}{\pa x_{\hat j}}z^{\ts}_{jj'}\frac{\pa f}{\pa x_{j'}}\nonumber \\
&&\left.\quad\quad\quad\quad\quad\quad\quad\quad+a^{\ts}_{ll'}a^{\ts}_{l\hat l} \frac{\pa}{\pa x_{\hat l}}z^{\ts}_{j\hat j}\frac{\pa f}{\pa x_{\hat j}}\frac{\partial}{\pa x_{l'}}z^{\ts}_{jj'}\frac{\pa f}{\pa x_{j'}}  \right] \nonumber \\
&&-2\sum_{j=1}^m\sum_{l=1}^n\sum_{\hat l,\hat j,j'=1}^{n+m}(a^{\ts}\nabla\log\rho^*)_l \left[ a^{\ts}_{l\hat l}\frac{\pa}{\pa x_{\hat l}}z^{\ts}_{j\hat j}\frac{\pa f}{\pa x_{\hat j}}z^{\ts}_{jj'}\frac{\pa f}{\pa x_{j'}} \right].\nonumber
\eeaa
\end{thm}
{\begin{rem}
We remark that Assumption \ref{assumption:main result} is a sufficient and necessary condition for the completing square formula in Theorem \ref{thm1}. In future works, we will study the geometry formulations in Theorem \ref{thm1}, using the Bott connection in sub-Riemannian manifold as in the classical Gamma $z$ calculus. For the concreteness of presentation, later on we will provide several examples, in which we can explicitly find the first order terms, after the completing square step. These examples demonstrate the feasibility for the proposed tensor.  
\end{rem}}
\begin{rem}
 We comment that $\mathfrak{R}^{G}_{ab}+\mathfrak{R}_{zb}+\mathfrak{R}_{\rho^*}$ in the sub-Riemannian manifold plays the role of $\textrm{Ric}-\textrm{Hess}\log\rho^*$ in the Riemannian manifold. If the metric is a Riemannian metric and $z=0$, then these two formulations of curvature tensor coincide. We notice that for the sub-Riemannian manifold, we have the freedom to choose a non-degenerate direction $z$. 
\end{rem}

\begin{proof}
Here we present the main idea of the proof. More computational details are shown in \cite{FL}. By routine calculations, we have 
\begin{equation*}
\begin{split}
    &\Gamma_{2}(f,f)+\Gamma_2^{z,\rho^*}(f,f)\\
    =&(X+D)^{\ts}Q^{\ts}Q(X+D)+(X+E)^{\ts}P^{\ts}P(X+E)+2(F+C+G)^{\ts}X\\
    &+\textrm{Quadratic forms of $\nabla f$}\\
    =&X^\ts Q^{\ts}QX+X^{\ts}P^\ts PX+2(F+C+G+Q^{\ts}D+P^{\ts}E)^{\ts}X\\
    &+\textrm{Quadratic forms of $\nabla f$}.
\end{split}
\end{equation*}
Substituting assumption \ref{assumption:main result} into the above formula, we have
\begin{equation*}
\begin{split}
   &\Gamma_{2}(f,f)+\Gamma_2^{z,\rho^*}(f,f)\\      
       =&X^\ts Q^{\ts}QX+X^{\ts}P^\ts PX+2\Lambda_1^{\ts}Q^{\ts}QX+2\Lambda_2^{\ts}P^\ts P X\\
    &+\textrm{Quadratic forms of $\nabla f$}\\
           =&(X+\Lambda_1)^\ts Q^{\ts}Q(X+\Lambda_1)+(X+\Lambda_2)^{\ts}P^\ts P(X+\Lambda_2)\\
          & -\Lambda_1^{\ts}Q^{\ts}Q\Lambda_1-\Lambda_2^{\ts}P^{\ts}P\Lambda_2+\textrm{Quadratic forms of $\nabla f$}\\
          =&|\mathfrak{Hess}_{a,z}^{G}f|^2
+\mathfrak{R}^{G}_{ab}(\nabla f,\nabla f) +\mathfrak{R}_{zb}(\nabla f,\nabla f)+\mathfrak{R}_{\rho^*}(\nabla f,\nabla f).
\end{split}
\end{equation*}
Here in the last equality, we summarize the squared term involving second order terms as $|\mathfrak{Hess}_{a,z}^{G}f|^2$, and formulates all quadratic forms of $\nabla f$ by $\mathfrak{R}^{G}_{ab}(\nabla f,\nabla f) +\mathfrak{R}_{zb}(\nabla f,\nabla f)+\mathfrak{R}_{\rho^*}(\nabla f,\nabla f)$.
\end{proof}
With $z$--Bochner's formula in hand, we are ready to present the following sub-Riemannian curvature dimension bound. 
\begin{defn}[sub-Riemannian curvature dimension bound]\label{def:CD k n}
We name the generalized curvature-dimension inequality $CD(\kappa, d)$ for degenerate diffusion process generator $L$ by
\begin{equation*}
\Gamma_{2}(f,f)+\Gamma_2^{z,\rho^*}(f,f)\geq \kappa \Gamma_{1}(f,f)+\kappa \Gamma_{1}^z(f,f)+\frac{1}{d}\textrm{tr}(\mathfrak{Hess}_{a,z}f)^2,
\end{equation*}
for any $f\in C^{\infty}(\mathbb{R}^{n+m}).$
In particular, the generalized $CD(\kappa, \infty)$ condition is equivalent to  
\begin{equation*}
\mathfrak{R}^{G}_{ab}(\nabla f, \nabla f) +\mathfrak{R}_{zb}(\nabla f, \nabla f)+\mathfrak{R}_{\rho^*}(\nabla f, \nabla f)\succeq \kappa\Big(\Gamma_1(f,f)+\Gamma_1^z(f,f)\Big).    
\end{equation*}
\end{defn}
Here we summarize all the result as follows:
$$\mathfrak{R}^{G}_{ab} +\mathfrak{R}_{zb}+\mathfrak{R}_{\rho^*}\succeq \kappa(\Gamma_1+\Gamma_1^z)\Rightarrow  \Gamma_2+\Gamma_2^{z,\rho^*}\succeq \kappa(\Gamma_1+\Gamma_1^z)\Rightarrow \textrm{zLSI}.$$
For the simplicity of presentations, we formulate the curvature tensor $\mathfrak{R}^{G}_{ab} +\mathfrak{R}_{zb}+\mathfrak{R}_{\rho^*}$ into a matrix format. Denote $$\mathsf{U}=\Big(
(a^{\ts}\nabla)_1f,\cdots ,(a^{\ts}\nabla)_nf , (z^{\ts}\nabla)_1f,\cdots (z^{\ts}\nabla)_m f
\Big)_{(n+m)\times 1},$$ and denote $\mathbf{I}_{(n+m)\times (n+m)}$ as the identity matrix. In this case, our Ricci curvature tensor forms 
\beaa
\mathfrak{R}_{ab}^G(\nabla f,\nabla f)+\mathfrak{R}_{zb}(\nabla f,\nabla f)+\mathfrak{R}_{\rho^*}(\nabla f,\nabla f)=\mathsf{U}^{\ts}\cd\mathsf{A}\cd\mathsf{U},
\eeaa
where 
\beaa
\mathsf{A}&\succeq& \kappa 
\mathbf{I}_{(n+m)\times(n+m)},\\
\Rightarrow \mathfrak{R}_{ab}^G(\nabla f,\nabla f)+\mathfrak{R}_{zb}(\nabla f,\nabla f)+\mathfrak{R}_{\rho^*}(\nabla f,\nabla f)&\succeq& \kappa (\Gamma_1(f,f)+\Gamma_1^z(f,f)).
\eeaa
Later on, we present analytical formulations of sub--Riemannian Ricci curvature in the form of $A$ in examples. 

\section{Examples}\label{section4}
In this section, we provide several examples for analytical formulations of sub-Riemannian Ricci curvature tensors.
\subsection{Heisenberg group}
In this subsection, we apply our general theory to the well-known example in sub-Riemannian geometry, which is the Heisenberg group. We believe that even for Heisenberg group, the analytical bound for the $z$-LSI  is also new. A related LSI for the horizontal Wiener measure has been studied in  \cite{baudoin2015log}. Recall briefly that the Heisenberg group $\mathbb H^1$ admits left invariant vector fields: $X=\frac{\partial}{\partial x}-\frac{1}{2}y\frac{\partial}{\partial z},\quad Y=\frac{\partial}{\partial y}+\frac{1}{2}x\frac{\partial}{\partial z},\quad Z=\frac{\partial}{\partial z}$. Here $\{X,Y,Z\}$ forms an orthonormal basis for the tangent bundle of $\mathbb H^1$. In this case, $\Vol=1$. In particular, $X$ and $Y$ generate the horizontal distribution $\tau$. To fit into our general  theory from the previous section, we take matrices $a$ and $z$ as below
\beaa
a^{\ts}=\begin{pmatrix}{}1&0&-y/2\\
0&1&x/2
\end{pmatrix},\quad z^{\ts}=(0,0,1).
\eeaa
In particular, we have 
\beaa
a^{\ts}\nabla f=\Big((a^{\ts}\nabla)_1f,(a^{\ts}\nabla)_2f\Big)^{\ts},\quad (a^{\ts}\nabla)_1f=(\frac{\pa f}{\pa x}-\frac{y}{2}\frac{\pa f}{\pa z}),\quad (a^{\ts}\nabla)_2f=(\frac{\pa f}{\pa y}+\frac{x}{2}\frac{\pa f}{\pa z}).
\eeaa 

We have the following proposition for Heisenberg group following Theorem \ref{thm1}.
\begin{prop}\label{prop heisenberg} For any smooth function $f\in C^{\infty}(\mathbb{H}^1)$, one has
\beaa
\Gamma_2(f,f)+\Gamma_2^{z,\pi}(f,f)&=&\|\mathfrak{Hess}_{a,z}f\|^2+\mathfrak{R}(\nabla f,\nabla f),
\eeaa
where 
\beaa
\Lambda_1^{\ts}&=&(0,0,0,0,0,0,0,0,0); \\
\Lambda_2^{\ts}&=&(0,0,0,0,0,0,(a^{\ts}\nabla)_2f,-(a^{\ts}\nabla)_1f,0);\\
&&\mathfrak{R}_{ab}(\nabla f,\nabla f)-\Lambda_1^{\ts}Q^{\ts}Q\Lambda_1-\Lambda_2^{\ts}P^{\ts}P\Lambda_2+D^{\ts}D+E^{\ts}E\\
&=&-\Gamma_1(f,f)+\frac{1}{2}\Gamma_1^z(f,f)-(a^{\ts}\nabla)_1V\pa_z f(a^{\ts}\nabla)_2f+(a^{\ts}\nabla)_2V\pa_z f(a^{\ts}\nabla)_1f\\
&&+\Big[\frac{\pa^2 V}{\pa x\pa x}+\frac{y^2}{4}\frac{\pa^2 V}{\pa z\pa z}-y\frac{\pa^2 V}{\pa x\pa z}\Big] |(a^{\ts}\nabla)_1 f|^2\\
&&+\Big[\frac{\pa^2 V}{\pa y\pa y}+\frac{x^2}{4}\frac{\pa^2 V}{\pa z\pa z}+x\frac{\pa^2 V}{\pa y\pa z}\Big] |(a^{\ts}\nabla)_2 f|^2\\
&&+2\Big[\frac{\pa^2 V}{\pa x\pa y}+\frac{x}{2}\frac{\pa^2 V}{\pa x\pa z}-\frac{y}{2}\frac{\pa^2 V}{\pa y\pa z}-\frac{xy}{4}\frac{\pa^2 V}{\pa z\pa z}\Big] (a^{\ts}\nabla)_1 f(a^{\ts}\nabla)_2 f;\\
\mathfrak{R}_{zb}(\nabla f, \nabla f)&=& \Big(\frac{\pa^2 V}{\pa x\pa z }-\frac{y}{2}\frac{\pa^2 V}{\pa z\pa z}\Big)(a^{\ts}\nabla)_1 f (z^{\ts}\nabla)_1 f+\Big(\frac{\pa^2 V}{\pa y\pa z}+\frac{x}{2}\frac{\pa^2 V}{\pa z\pa z}\Big)(z^{\ts}\nabla)_1 f(a^{\ts}\nabla)_2 f;\\
\mathfrak{R}_{\pi}(\nabla f, \nabla f)&=&0.
\eeaa 
\end{prop}
We next formulate the curvature tensor into a matrix format. Denote $$\mathsf{U}=\Big(
(a^{\ts}\nabla)_1f , (a^{\ts}\nabla)_2f , (z^{\ts}\nabla)_1f
\Big)_{3\times 1},$$ and denote $\mathbf{I}_{3\times 3}$ as the identity matrix. 
There exists a symmetric matrix $\mathsf{A}$ such that we can represent the tensor as below.
\beaa
\mathfrak{R}_{ab}^G(\nabla f,\nabla f)+\mathfrak{R}_{zb}(\nabla f,\nabla f)+\mathfrak{R}_{\rho^*}(\nabla f,\nabla f)=\mathsf{U}^{\ts}\cd\mathsf{A}\cd\mathsf{U},
\eeaa
which implies that 
\beaa
\mathsf{A}&\succeq& \kappa 
\mathbf{I}_{3\times 3},\\
\Rightarrow \mathfrak{R}_{ab}^G(\nabla f,\nabla f)+\mathfrak{R}_{zb}(\nabla f,\nabla f)+\mathfrak{R}_{\rho^*}(\nabla f,\nabla f)&\succeq& \kappa (\Gamma_1(f,f)+\Gamma_1^z(f,f)).
\eeaa
In other words, we need to estimate the smallest eigenvalue of matrix $A$. We next present the formulation of matrix $A$ for the Heisenberg group as follows. \\ 
\begin{cor}
 The matrix $\mathsf{A}$ associated  with Heisenberg group has the following form
\beaa
\mathsf{A}_{11}&=& \Big[\frac{\pa^2 V}{\pa x\pa x}+\frac{y^2}{4}\frac{\pa^2 V}{\pa z\pa z}-y\frac{\pa^2 V}{\pa x\pa z}\Big]-1;\\
\mathsf{A}_{22}&=&\Big[\frac{\pa^2 V}{\pa y\pa y}+\frac{x^2}{4}\frac{\pa^2 V}{\pa z\pa z}+x\frac{\pa^2 V}{\pa y\pa z}\Big]-1;\quad \mathsf A_{33}=\frac{1}{2};  \\
\mathsf{A}_{12}&=&\mathsf{A}_{21}=\Big[\frac{\pa^2 V}{\pa x\pa y}+\frac{x}{2}\frac{\pa^2 V}{\pa x\pa z}-\frac{y}{2}\frac{\pa^2 V}{\pa y\pa z}-\frac{xy}{4}\frac{\pa^2 V}{\pa z\pa z}\Big];\\
\mathsf{A}_{13}&=&\mathsf{A}_{31}=\frac{1}{2}(a^{\ts}\nabla)_2V+\frac{1}{2}\Big(\frac{\pa^2 V}{\pa x\pa z }-\frac{y}{2}\frac{\pa^2 V}{\pa z\pa z}\Big);\\
\mathsf{A}_{23}&=&\mathsf{A}_{32}=-\frac{1}{2}(a^{\ts}\nabla)_1V+\frac{1}{2}\Big(\frac{\pa^2 V}{\pa y\pa z}+\frac{x}{2}\frac{\pa^2 V}{\pa z\pa z}\Big).
\eeaa
\end{cor}

\subsection{Displacement group}
In this subsection, we derive the generalized curvature dimension bound for displacement group, which is one example of three dimensional solvable Lie groups. We adapt the general setting from  \cite{baudoin2015subelliptic} below. Denote $\mathfrak{g}$ as the three dimensional solvable Lie algebra and denote $H\subset \mathfrak{g}$ as the horizontal subspace satisfying H\"ormander's condition, then for a given inner product $\la\cd ,\cd\ra$ on $H$, there exists a canonical basis $\{X,Y,Z\}$ for $(\mathfrak{g},H,\la\cd,\cd\ra)$, such that $\{X,Y\}$ forms an orthonormal basis for $H$ and satisfies the following Lie bracket generating condition for parameters $\alpha$ and $\beta\ge 0$:
\[
[X,Y]=Z,\quad [X,Z]=\alpha Y+\beta Z,\quad [Y,Z]=0.
\]
When the parameter $\alpha=0$ and $\beta\neq 0$, the Lie algebra $\mathfrak{g}$ has a faithful representation. In particular, it is shown in \cite{baudoin2015subelliptic} that the elements of $\mathfrak{g}$, in local coordinates $(\theta,x,y)$, corresponds to the following left-invariant differential operators:
\beaa
X=\frac{\partial}{\partial{\theta}},\quad Y=e^{\beta\theta}\frac{\partial}{\pa x}+\frac{\pa}{\pa y},\quad R=-\beta\frac{\pa}{\pa y},
\eeaa
with the following relation
\beaa
[X,Y]=\beta Y+R,\quad [X,R]=0,\quad [Y,R]=0.
\eeaa
In terms of local coordinates $(\theta,x,y)$, we have 
\beaa
X=\begin{pmatrix}
	1\\
	0\\
	0
\end{pmatrix},\quad Y=\begin{pmatrix}
	0\\
	e^{\beta\theta}\\
	1
\end{pmatrix},\quad R=\begin{pmatrix}
	0\\
	0\\
	-\beta 
\end{pmatrix}.
\eeaa
The corresponding Lie group of this special Lie algebra $\mathfrak{g}$ is called displacement group, denoted as $\mathsf G$. We choose $\{X,Y\}$ as the horizontal orthonormal basis for subalgebra $H$.
To fit into the general framework from the previous section, we take 
\beaa
a=(X,Y)=\begin{pmatrix}1&0\\
0&e^{\beta \theta}\\
0&1
\end{pmatrix}, \quad a^{\ts}= \begin{pmatrix}{}
1&0&0\\
0&e^{\beta \theta}&1
\end{pmatrix},\quad z^{\ts}=\begin{pmatrix}{}0&0&-g(\theta,x,y)
\end{pmatrix},
\eeaa
with $g(\theta,x,y)\neq 0$. Our focus here is to derive the curvature tensor in terms of $\rho^*=\frac{1}{Z}e^{-V}\Vol$. In this case, $$\Vol=1.$$  We then use $(aa^{\ts})^{\dd}_{|H}$ as the horizontal metric on $H$. Thus the sub-Riemannian structure is given by $(\mathsf G,H,(aa^{\ts})^{\dd}_{|H})$ and we proceed to derive the generalized curvature dimension bound following our framework in Section \ref{section3}. 
By direct computations, it is easy to show that, for general smooth function $f$, $\Gamma_1(f,\Gamma_1^z(f,f))\neq\Gamma_1^z(f,\Gamma_1(f,f))$. Hence classical Gamma $z$ calculus proposed in \cite{BaudoinGarofalo09} can not be extended for this case to derive zLSI. Thus we need to compute vector $G$ and the tensor term $\mathfrak{R}_{\rho^*}$. We have the following proposition.

Following Theorem \ref{thm1}, we have the following z-Bochner's formula for $\mathsf G$.  
\begin{prop}\label{prop se2} For any smooth function $f\in C^{\infty}(\mathsf G)$, one has
\beaa
\Gamma_2(f,f)+\Gamma_2^{z,\pi}(f,f)&=&\|\mathfrak{Hess}_{a,z}f\|^2+\mathfrak{R}(\nabla f,\nabla f),
\eeaa
where 
\beaa
\Lambda_1^{\ts}&=& (0,\beta\pa_xf,\frac{\beta\pa_yf}{2},\beta\pa_xf,0,0,\frac{\beta\pa_yf}{2},0,-\beta\pa_{\theta}f); \\
\Lambda_2^{\ts}&=& (0,0,0,0,0,0,\lambda_6,0,\lambda_9);\\
\lambda_6&=&\frac{\pa_{\theta}g\pa_yf}{g}-\frac{\beta(a^{\ts}\nabla)_2f}{g^2}-\frac{\pa_{\theta}g\pa_yf}{g};\\ \lambda_9&=&\frac{(a^{\ts}\nabla)_2g\pa_yf}{g}+\frac{\beta \pa_{\theta}f}{g^2}-\frac{(a^{\ts}\nabla)_2g\pa_yf}{g};
\eeaa
And
\beaa
&&\mathfrak{R}_{ab}(\nabla f,\nabla f)-\Lambda_1^{\ts}Q^{\ts}Q\Lambda_1-\Lambda_2^{\ts}P^{\ts}P\Lambda_2+D^{\ts}D+E^{\ts}E)\\
&=&\Gamma_1(\log g,\log g)\Gamma_1^z(f,f)-\beta^2(1+\frac{1}{g^2})\Gamma_1(f,f)+\frac{\beta^2}{2g^2}\Gamma_1^z(f,f)\\
&&+\beta^2e^{\beta \theta}\frac{\pa f}{ \pa x}(a^{\ts}\nabla )_2f +\beta e^{\beta\theta} (a^{\ts}\nabla)_2V\frac{\pa f}{\pa x}(a^{\ts}\nabla )_1 f  +\beta e^{\beta \theta}\frac{\pa V}{\pa x}(a^{\ts}\nabla)_2f(a^{\ts}\nabla)_1f \\
&&+\frac{\pa^2 V}{\pa \theta\pa \theta} |(a^{\ts}\nabla)_1 f|^2+2(e^{\beta\theta}\frac{\pa^2 V}{\pa \theta\pa x}+\frac{\pa^2 V}{\pa \theta \pa y} )(a^{\ts}\nabla)_1 f(a^{\ts}\nabla)_2 f\\
&&+\sum_{\hat i,k'=1}^3a^{\ts}_{2\hat i}a^{\ts}_{2k'}\frac{\pa^2 V}{\pa x_{\hat i}\pa x_{k'}})|(a^{\ts}\nabla)_2 f|^2-\beta e^{\beta\theta}(a^{\ts}\nabla)_1V \frac{\pa f}{\pa x}(a^{\ts}\nabla)_2f;\\
\mathfrak{R}_{zb}(\nabla f,\nabla f)&=& \sum_{i=1}^2\sum_{i',\hat i=1}^{3}  a^{\ts}_{ii'}a^{\ts}_{i \hat i} \frac{\pa^2 z^{\ts}_{13}}{\partial x_{i'}\partial x_{\hat i}}\pa_yf (z^{\ts}\nabla)_1f -\sum_{k=1}^2(a^{\ts}\nabla)_kz^{\ts}_{13}(a^{\ts}\nabla)_kV\pa_yf(z^{\ts}\nabla)_1f
\\ && -g\frac{\pa^2 V}{\pa \theta\pa y}(a^{\ts}\nabla)_1 f(z^{\ts}\nabla)_1 f-g(e^{\beta\theta}\frac{\pa^2 V}{\pa x\pa y}+\frac{\pa^2 V}{\pa y\pa y} )(a^{\ts}\nabla)_2 f(z^{\ts}\nabla)_1 f;\\
\mathfrak{R}_{\pi}(\nabla f,\nabla f)&=& -2\sum_{l=1}^2\sum_{l',\hat l=1}^3a^{\ts}_{ll'}a^{\ts}_{l\hat l}\frac{\pa^2 z^{\ts}_{13}}{\pa x_{l'}\pa x_{\hat l}}\pa_yf (z^{\ts}\nabla)_1f\\
&&-2\Gamma_1(\log \pi,\log g)|(z^{\ts}\nabla)_1f|^2-2\Gamma_1(\log g,\log g)|(z^{\ts}\nabla)_1f|^2.
\eeaa
In particular, we have  
\beaa
\sum_{\hat i,k'=1}^3a^{\ts}_{2\hat i}a^{\ts}_{2k'}\frac{\pa^2 V}{\pa x_{\hat i}\pa x_{k'}}|(a^{\ts}\nabla)_2 f|^2&=&\Big[e^{2\beta\theta}\frac{\pa^2 V}{\pa x\pa x}+2e^{\beta \theta}\frac{\pa^2 V}{\pa x\pa y}+\frac{\pa^2 V}{\pa y\pa y}\Big] |(a^{\ts}\nabla)_2 f|^2;\\
\sum_{i=1}^2\sum_{i',\hat i=1}^{3}  a^{\ts}_{ii'}a^{\ts}_{i \hat i} \frac{\pa^2 z^{\ts}_{13}}{\partial x_{i'}\partial x_{\hat i}}\pa_yf (z^{\ts}\nabla)_1f&=&\Big[\frac{\pa ^2 g}{\pa \theta\pa \theta}+e^{2\beta\theta}\frac{\pa ^2 g}{\pa x\pa x}+\frac{\pa ^2 g}{\pa y\pa y}+2e^{\beta\theta}\frac{\pa ^2 g}{\pa x\pa y}\Big]\frac{|(z^{\ts}\nabla)_1f|^2}{g}.
\eeaa
\end{prop}{}

Similarly, we formulate the curvature tensor into a matrix format of $\mathsf A$. 
\begin{cor}
The matrix $\mathsf{A}$ associated with the displacement group has the following representation
\beaa
\mathsf A_{11}&=& \frac{\pa^2 V}{\pa \theta\pa \theta}-\beta^2(1+\frac{1}{g^2});\\
\mathsf A_{22}&=& \Big[e^{2\beta\theta}\frac{\pa^2 V}{\pa x\pa x}+2e^{\beta \theta}\frac{\pa^2 V}{\pa x\pa y}+\frac{\pa^2 V}{\pa y\pa y}\Big]-\frac{\beta^2}{g^2}-\beta(a^{\ts}\nabla)_1V;\\
\mathsf A_{33}&=&\frac{\beta^2}{2g^2}-\Gamma_1(\log g,\log g)-2\Gamma_1(\log \pi,\log g)-\Gamma_1(\log g,V)-\frac{1}{g}\Big[\frac{\pa ^2 g}{\pa \theta\pa \theta}+e^{2\beta\theta}\frac{\pa ^2 g}{\pa x\pa x}+\frac{\pa ^2 g}{\pa y\pa y}+2e^{\beta\theta}\frac{\pa ^2 g}{\pa x\pa y}\Big];  \\
\mathsf A_{12}&=&\mathsf A_{21}=\frac{1}{2}\Big(\beta e^{\beta \theta}\frac{\pa V}{\pa x}+2(e^{\beta\theta}\frac{\pa^2 V}{\pa \theta\pa x}+\frac{\pa^2 V}{\pa \theta \pa y} )+\beta (a^{\ts}\nabla)_2V \Big);\\
\mathsf A_{13}&=&\mathsf A_{31}=\frac{1}{2}\Big(\frac{\beta}{g}(a^{\ts}\nabla)_2V-g\frac{\pa^2 V}{\pa \theta\pa y}\Big);\\
\mathsf A_{23}&=&\mathsf A_{32}=-\frac{1}{2}\Big(\frac{\beta}{g}(a^{\ts}\nabla)_1-\frac{\beta^2}{g} \Big)-\frac{1}{2}g(e^{\beta\theta}\frac{\pa^2 V}{\pa x\pa y}+\frac{\pa^2 V}{\pa y\pa y} ).
\eeaa
\end{cor}{}


\subsection{Martinet flat sub-Riemannian structure}
In this part, we apply our result to Martinet flat sub-Riemannian structure, which satisfies bracket generating condition and has non-equiregular sub-Riemannian structure (see \cite{barilari2013formula}). The sub-Riemannian structure is defined on $\mathbb R^3$ through the kernel of one-form $\eta :=dz-\frac{1}{2}y^2dx.$ A global orthonormal basis for the horizontal distribution $\mathcal{H}$ adapt the following differential operator representation, in local coordinates $(x,y,z)$,  
\[
X=\frac{\pa}{\pa x}+\frac{y^2}{2}\frac{\pa}{\pa z},\quad Y=\frac{\pa}{\pa y}.
\]
The commutative relation gives
\beaa
[X,Y]=-yZ,\quad [Y,[X,Y]]=Z,\quad \text{where}\quad Z=\frac{\pa}{\pa z}.
\eeaa
To apply in our framework, we take
\beaa
a=\begin{pmatrix}
	1&0\\
	0&1\\
	\frac{y^2}{2}&0
\end{pmatrix},\quad  a^T=\begin{pmatrix}
	1&0&\frac{y^2}{2}\\
	0&1&0
\end{pmatrix},\quad z^{\ts}=(0,0,1),\quad aa^T=\begin{pmatrix}
	1&0&\frac{y^2}{2}\\
	0&1&0\\
	\frac{y^2}{2}&0&\frac{y^4}{4}
\end{pmatrix}.
\eeaa
Thus the sub-Riemanian structure we consider here has the form $(\mathbb M,\mathcal H, (aa^{\ts})^{\dd}_{|\mathcal{H}})$.

\begin{prop}
In this setting, 
\begin{equation*}
\Vol=e^{-\frac{y^2}{2}},   
\end{equation*}
then 
\[
-aa^{\ts}\nabla \log \Vol=a\otimes\nabla a.
\]
\end{prop}
\begin{proof}
\[
a\otimes\nabla a= \begin{pmatrix}
0&
y&
0
\end{pmatrix}^{\ts}, \quad aa^{\ts}\nabla \log\Vol= \begin{pmatrix}
0&
y&
0
\end{pmatrix}^{\ts}.
\]
\end{proof}
\begin{rem}
In this case, $\Vol$ is different from the Popp's volume, which has the form $\frac{1}{|y|}$; see details in \cite{barilari2013formula}. 
\end{rem}

Similar to the previous displacement group case, we have the following identity.
\begin{prop}\label{prop Martinet}  For any smooth function $f\in C^{\infty}(\mathbb{M})$, one has
\beaa
\Gamma_2(f,f)+\Gamma_2^{z,\rho^*}(f,f)&=&|\mathfrak{Hess}_{a,z}^{G}f|^2+\mathfrak{R}^G_{ab}(\nabla f,\nabla f)+\mathfrak{R}_{zb}(\nabla f,\nabla f)+\mathfrak{R}_{\rho^*}(\nabla f,\nabla f),
\eeaa
where 
\beaa
\Lambda_1^{\ts}&=&(0,y\pa_zf/2,0,y\pa_zf/2,0,0,0,0,0); \\
\Lambda_2^{\ts}&=&(0,0,0,0,0,0,-y\pa_yf,\frac{y^3}{2}\pa_zf+y\pa_xf,0);\\
\mathfrak{R}^G_{ab}(\nabla f,\nabla f)&=&\frac{y^2}{2}\Gamma_1^z(f,f)-y^2\Gamma_1(f,f)\\
&&+\frac{\pa f}{\pa z}(a^{\ts}\nabla )_1f+y(a^{\ts}\nabla)_1V \frac{\pa f}{\pa z}(a^{\ts}\nabla)_2f +y \frac{\pa V}{\pa z}(a^{\ts}\nabla)_1f(a^{\ts}\nabla)_2f \\
&&+\sum_{\hat i,k'=1}^3a^{\ts}_{1\hat i}a^{\ts}_{1k'}\frac{\pa^2 V}{\pa x_{\hat i}\pa x_{k'}}|(a^{\ts}\nabla)_1 f|^2+2(\frac{\pa^2 V}{\pa x\pa y}+\frac{y^2}{2}\frac{\pa^2 V}{\pa y \pa z} )(a^{\ts}\nabla)_1 f(a^{\ts}\nabla)_2 f\\
&&+\frac{\pa^2 V}{\pa y\pa y}|(a^{\ts}\nabla)_2 f|^2-y\frac{\pa V}{\pa y}\frac{\pa f}{\pa z}(a^{\ts}\nabla)_1f;\\
\mathfrak{R}_{zb}(\nabla f,\nabla f)&=& (\frac{\pa^2 V}{\pa x\pa z}+\frac{y^2}{2}\frac{\pa^2 V}{\pa z\pa z})(a^{\ts}\nabla)_1 f(z^{\ts}\nabla)_1 f+\frac{\pa^2 V}{\pa y\pa z}(a^{\ts}\nabla)_2 f(z^{\ts}\nabla)_1 f;\\
\mathfrak{R}_{\rho^*}(\nabla f,\nabla f)&=&0.
\eeaa
In particular, we have 
\beaa
\sum_{\hat i,k'=1}^3a^{\ts}_{1\hat i}a^{\ts}_{1k'}\frac{\pa^2 V}{\pa x_{\hat i}\pa x_{k'}}|(a^{\ts}\nabla)_1 f|^2=\Big(\frac{\pa^2 V}{\pa x\pa x}+y^2\frac{\pa^2 V}{\pa x\pa z}+\frac{y^4}{4}\frac{\pa^2 V}{\pa z\pa z} \Big)|(a^{\ts}\nabla)_1 f|^2.
\eeaa 
\end{prop}{}

Similarly, we summarize the sub-Riemannian Ricci tensor in terms of $\mathsf A$ as follows. 
\begin{cor} The matrix $\mathsf{A}$ associated  with Martinet sub-Riemannian structure has the following form
\beaa
\mathsf A_{11}&=&\Big(\frac{\pa^2 V}{\pa x\pa x}+y^2\frac{\pa^2 V}{\pa x\pa z}+\frac{y^4}{4}\frac{\pa^2 V}{\pa z\pa z} \Big)-y^2;\\
\mathsf A_{22}&=&\frac{\pa^2 V}{\pa y\pa y}-y^2;\quad \mathsf A_{33}=\frac{y^2}{2};  \\
\mathsf A_{12}&=&\mathsf A_{21}=\frac{y}{2}\frac{\pa V}{\pa z}+(\frac{\pa^2 V}{\pa x\pa y}+\frac{y^2}{2}\frac{\pa^2 V}{\pa y\pa z});\\
\mathsf A_{13}&=&\mathsf A_{31}=\frac{1}{2}-\frac{y}{2}\frac{\pa V}{\pa y}+\frac{1}{2}(\frac{\pa^2 V}{\pa x\pa z}+\frac{y^2}{2}\frac{\pa^2 V}{\pa z\pa z}) ;\quad \mathsf A_{23}=\mathsf A_{32}=\frac{1}{2}y(a^{\ts}\nabla)_1V+\frac{1}{2}\frac{\pa^2 V}{\pa y\pa z}.
\eeaa
\end{cor}{}



\section{Appendix}\label{section4}
In appendix, we provide the derivation details for our examples.  
\subsection{Proof Of Proposition \ref{prop heisenberg}}
The proof of Proposition of \ref{prop heisenberg} follows from the following three lemmas.
\begin{lem}
For Heisenberg group, we have 
\beaa
Q&=&\left(
\begin{array}{ccccccccc}
 1 & 0 & -\frac{y}{2} & 0 & 0 & 0 & -\frac{y}{2} & 0 & \frac{y^2}{4} \\
 0 & 1 & \frac{x}{2} & 0 & 0 & 0 & 0 & -\frac{y}{2} & -\frac{x y}{4} \\
 0 & 0 & 0 & 1 & 0 & -\frac{y}{2} & \frac{x}{2} & 0 & -\frac{x y}{4} \\
 0 & 0 & 0 & 0 & 1 & \frac{x}{2} & 0 & \frac{x}{2} & \frac{x^2}{4} \\
\end{array}
\right);\\
P&=&\left(
\begin{array}{ccccccccc}
 0 & 0 & 0 & 0 & 0 & 0 & 1 & 0 & -\frac{y}{2} \\
 0 & 0 & 0 & 0 & 0 & 0 & 0 & 1 & \frac{\text{x}}{2} \\
\end{array}
\right);\\
D&=&(0,\frac{1}{2}\pa_zf,-\frac{1}{2}\pa_zf,0)^{\ts};\quad E^{\ts}=(0,0);\quad F^{\ts}=G^{\ts}=(0,0,0,0,0,0,0,0,0);\\
C^T&=&(0,0,\frac{x}{4}\pa_zf+\frac{1}{2}\pa_y f,0,0,\frac{y}{4}\pa_z f-\frac{1}{2}\pa_x f,\frac{x}{4}\pa_zf+\frac{1}{2}\pa_yf,\frac{y}{4}\pa_zf-\frac{1}{2}\pa_xf,-\frac{y}{2}\pa_yf-\frac{x}{2}\pa_xf).
\eeaa
\end{lem}{}
\begin{lem}\label{heisenberg hess} On $\mathbb H^1$, vectors $F$ and $G$ are zero vectors, we have
\beaa
[QX+D]^{\ts}[QX+D]+[PX+E]^{\ts}[PX+E]+2C^{\ts}X
=|\mathfrak{Hess}_{a,z}^{G}f|^2+\mathfrak{R}^{G}(\nabla f,\nabla f).
\eeaa
In particular, we have 
\beaa
|\mathfrak{Hess}_{a,z}^{G}f|^2&=&[X+\Lambda_1]^{\ts}Q^{\ts}Q[X+\Lambda_1]+[X+\Lambda_2]^{\ts}P^{\ts}P[X+\Lambda_2];\\
\Lambda_1^{\ts}&=&(0,0,0,0,0,0,0,0,0); \\
\Lambda_2^{\ts}&=&(0,0,0,0,0,0,(a^{\ts}\nabla)_2f,-(a^{\ts}\nabla)_1f,0);\\
\mathfrak{R}^{G}(\nabla f,\nabla f)&=&-\Gamma_1(f,f)+\frac{1}{2}\Gamma_1^z(f,f).
\eeaa
\end{lem}{}

\begin{lem}\label{heisenberg tensor}
By routine computations, we obtain
\beaa
\mathfrak{R}_{ab}(\nabla f,\nabla f)&=&-(a^{\ts}\nabla)_1V\pa_z f(a^{\ts}\nabla)_2f+(a^{\ts}\nabla)_2V\pa_z f(a^{\ts}\nabla)_1f\\
&&+\Big[\frac{\pa^2 V}{\pa x\pa x}+\frac{y^2}{4}\frac{\pa^2 V}{\pa z\pa z}-y\frac{\pa^2 V}{\pa x\pa z}\Big] |(a^{\ts}\nabla)_1 f|^2\\
&&+\Big[\frac{\pa^2 V}{\pa y\pa y}+\frac{x^2}{4}\frac{\pa^2 V}{\pa z\pa z}+x\frac{\pa^2 V}{\pa y\pa z}\Big] |(a^{\ts}\nabla)_2 f|^2\\
&&+2\Big[\frac{\pa^2 V}{\pa x\pa y}+\frac{x}{2}\frac{\pa^2 V}{\pa x\pa z}-\frac{y}{2}\frac{\pa^2 V}{\pa y\pa z}-\frac{xy}{4}\frac{\pa^2 V}{\pa z\pa z}\Big] (a^{\ts}\nabla)_1 f(a^{\ts}\nabla)_2 f;\\
\mathfrak{R}_{zb}(\nabla f,\nabla f)&=&\Big(\frac{\pa^2 V}{\pa x\pa z }-\frac{y}{2}\frac{\pa^2 V}{\pa z\pa z}\Big)(a^{\ts}\nabla)_1 f(z^{\ts}\nabla)_1 f+\Big(\frac{\pa^2 V}{\pa y\pa z}+\frac{x}{2}\frac{\pa^2 V}{\pa z\pa z}\Big)(z^{\ts}\nabla)_1 f(a^{\ts}\nabla)_2 f;\\
\mathfrak{R}_{\rho^*}(\nabla f,\nabla f)&=&0.
\eeaa
\end{lem}{}

\begin{proof}
[Proof of Lemma \ref{heisenberg hess}]
We first have 
\beaa
2C^{\ts}X&=&\sum_{\hat i,\hat k=1}^32C^{\ts}_{\hat i\hat k}X_{\hat i\hat k}\\
&=&2\Big[\frac{\pa^2 f}{\pa x\pa z}\frac{(a^{\ts}\nabla)_2f}{2}-\frac{\pa^2 f}{\pa y\pa z}\frac{(a^{\ts}\nabla)_1f}{2}+\frac{\pa^2 f}{\pa z\pa x}\frac{(a^{\ts}\nabla)_2f}{2}\Big]\\
&&-2\Big[\frac{\pa^2 f}{\pa z\pa y}\frac{(a^{\ts}\nabla)_1f}{2}+\frac{\pa^2 f}{\pa z\pa z}(\frac{y}{2}\pa_yf+\frac{x}{2}\pa_xf)\Big]\\
&=&2\frac{\pa^2 f}{\pa x\pa z}(a^{\ts}\nabla)_2f-2\frac{\pa^2 f}{\pa y\pa z}(a^{\ts}\nabla)_1f-2\frac{\pa^2 f}{\pa z\pa z}(\frac{y}{2}\pa_yf+\frac{x}{2}\pa_xf)\\
&=&2(a^{\ts}\nabla)_2f\left[\frac{\pa^2 f}{\pa x\pa z}-\frac{y}{2}\frac{\pa^2 f}{\pa z\pa z}\right]-2(a^{\ts}\nabla )_1f\left[\frac{\pa^2 f}{\pa y\pa z}+\frac{x}{2}\frac{\pa^2 f}{\pa z\pa z}\right].
\eeaa
By direct computations, we have
\beaa
&&[QX+D]^{\ts}[QX+D]+[PX+E]^{\ts}[PX+E] +2C^{\ts}X\\
&=& \left[\frac{\pa^2 f}{\pa x\pa x}-y\frac{\pa^2 f}{\pa x\pa z}+\frac{y^2}{4}\frac{\pa^2 f}{\pa z\pa z} \right]^2 +\left[\frac{\pa^2 f}{\pa x\pa y}+\frac{x}{2}\frac{\pa^2 f}{\pa x\pa z}-\frac{y}{2}\frac{\pa^2 f}{\pa y\pa z}-\frac{xy}{4}\frac{\pa^2f}{\pa z\pa z}+\frac{1}{2}\pa_z f \right]^2\\
&&+\left[\frac{\pa^2 f}{\pa x\pa y}-\frac{y}{2}\frac{\pa^2 f}{\pa y\pa z}+\frac{x}{2}\frac{\pa^2 f}{\pa x\pa z}-\frac{xy}{4}\frac{\pa^2f}{\pa z\pa z}-\frac{1}{2}\pa_z f\right]^2
+\left[\frac{\pa^2 f}{\pa y\pa y}+x\frac{\pa^2 f}{\pa y\pa z}+\frac{x^2}{4}\frac{\pa^2f}{\pa z\pa z} \right]^2\\
&&+ \left[\frac{\pa^2 f}{\pa x\pa z}-\frac{y}{2}\frac{\pa^2 f}{\pa z\pa z}\right]^2+\left[\frac{\pa^2 f}{\pa y\pa z}+\frac{x}{2}\frac{\pa^2 f}{\pa z\pa z}\right]^2\\
&&+2(a^{\ts}\nabla )_2f\left[\frac{\pa^2 f}{\pa x\pa z}-\frac{y}{2}\frac{\pa^2 f}{\pa z\pa z}\right]-2(a^{\ts}\nabla )_1f\left[\frac{\pa^2 f}{\pa y\pa z}+\frac{x}{2}\frac{\pa^2 f}{\pa z\pa z}\right].
\eeaa
Completing squares for the cross terms involving the type of $``\nabla f\nabla^2 f"$ and following the reformulation as below
\beaa
&&\left[\frac{\pa^2 f}{\pa x\pa y}+\frac{x}{2}\frac{\pa^2 f}{\pa x\pa z}-\frac{y}{2}\frac{\pa^2 f}{\pa y\pa z}-\frac{xy}{4}\frac{\pa^2f}{\pa z\pa z}+\frac{1}{2}\pa_z f \right]^2\\
&&+\left[\frac{\pa^2 f}{\pa x\pa y}-\frac{y}{2}\frac{\pa^2 f}{\pa y\pa z}+\frac{x}{2}\frac{\pa^2 f}{\pa x\pa z}-\frac{xy}{4}\frac{\pa^2f}{\pa z\pa z}-\frac{1}{2}\pa_z f\right]^2\\
&=&2\left[\frac{\pa^2 f}{\pa x\pa y}-\frac{y}{2}\frac{\pa^2 f}{\pa y\pa z}+\frac{x}{2}\frac{\pa^2 f}{\pa x\pa z}-\frac{xy}{4}\frac{\pa^2f}{\pa z\pa z} \right]^2+\frac{1}{2}|\pa_z f|^2,
\eeaa
we have 
\beaa
&&[QX+D]^{\ts}[QX+D]+[PX+E]^{\ts}[PX+E] +2C^{\ts}X\\
&=& \left[\frac{\pa^2 f}{\pa x\pa x}-y\frac{\pa^2 f}{\pa x\pa z}+\frac{y^2}{4}\frac{\pa^2 f}{\pa z\pa z} \right]^2+2 \left[\frac{\pa^2 f}{\pa x\pa y}-\frac{y}{2}\frac{\pa^2 f}{\pa y\pa z}+\frac{x}{2}\frac{\pa^2 f}{\pa x\pa z}-\frac{xy}{4}\frac{\pa^2f}{\pa z\pa z}\right]^2\\
&&+\left[\frac{\pa^2 f}{\pa y\pa y}+x\frac{\pa^2 f}{\pa y\pa z}+\frac{x^2}{4}\frac{\pa^2f}{\pa z\pa z}\right]^2+ \left[\frac{\pa^2 f}{\pa x\pa z}-\frac{y}{2}\frac{\pa^2 f}{\pa z\pa z}+(a^{\ts}\nabla)_2f\right]^2\\
&&+\left[\frac{\pa^2 f}{\pa y\pa z}+\frac{x}{2}\frac{\pa^2 f}{\pa z\pa z}-(a^{\ts}\nabla)_1f\right]^2-|(a^{\ts}\nabla )_2f|^2-|(a^{\ts}\nabla )_1f|^2+\frac{1}{2} |(z^{\ts}\nabla )_1f|^2.
\eeaa

The sum of square terms give $\|\mathfrak{Hess}_{a,z}\|^2_{\mathrm F}$, hence $\Lambda_1$ and $\Lambda_2$. The remainders generate $-\Lambda_1^{\ts}Q^{\ts}Q\Lambda_1-\Lambda_2^{\ts}P^{\ts}P\Lambda_2+D^{\ts}D+E^{\ts}E$, which equals $-\Gamma_1(f,f)+\frac{1}{2}\Gamma_1^z(f,f)$.
\qed 
\end{proof}

We are now left to compute the tensors. 

\begin{proof}[Proof of Lemma \ref{heisenberg tensor}]
By direct computation, we have 
\beaa
\mathfrak{R}_a(\nabla f,\nabla f)
&=&\sum_{i,k=1}^2\sum_{i',\hat i,\hat k=1}^{3} \la a^{\ts}_{ii'} (\frac{\partial a^{\ts}_{i \hat i}}{\partial x_{i'}} \frac{\partial a^{\ts}_{k\hat k}}{\partial x_{\hat i}}\frac{\partial f}{\partial x_{\hat k}}) ,(a^{\ts}\nabla)_kf\ra_{\hR^2}  \nonumber\\
&&+\sum_{i,k=2}^n\sum_{i',\hat i,\hat k=1}^{3} \la a^{\ts}_{ii'}a^{\ts}_{i \hat i} (\frac{\partial }{\partial x_{i'}} \frac{\partial a^{\ts}_{k\hat k}}{\partial x_{\hat i}})(\frac{\partial f}{\partial x_{\hat k}}) ,(a^{\ts}\nabla)_kf\ra_{\hR^2} \nonumber\\
&&-\sum_{i,k=1}^2\sum_{i',\hat i,\hat k=1}^{3} \la a^{\ts}_{k\hat k}\frac{\partial a^{\ts}_{ii'}}{\partial x_{\hat k}} \frac{\partial a^{\ts}_{i \hat i}}{\partial x_{i'}} \frac{\partial f}{\partial x_{\hat i}})  ,(a^{\ts}\nabla)_kf\ra_{\hR^2}  \nonumber\\
&&-\sum_{i,k=1}^2\sum_{i',\hat i,\hat k=1}^{3} \la a^{\ts}_{k\hat k} a^{\ts}_{ii'} (\frac{\partial }{\partial x_{\hat k}} \frac{\partial a^{\ts}_{i \hat i}}{\partial x_{i'}}) \frac{\partial f}{\partial x_{\hat i}}  ,(a^{\ts}\nabla)_kf\ra_{\hR^2},  \nonumber\\
&=& \cI_1+\cI_2+\cI_3+\cI_4.
\eeaa
For the four terms above, we have 
\beaa
\cI_1&=&\sum_{i=1}^2\sum_{i',\hat i,\hat k=1}^{3}  a^{\ts}_{ii'} (\frac{\partial a^{\ts}_{i \hat i}}{\partial x_{i'}} \frac{\partial a^{\ts}_{1\hat k}}{\partial x_{\hat i}}\frac{\partial f}{\partial x_{\hat k}})(a^{\ts}\nabla)_1f+\sum_{i=1}^2\sum_{i',\hat i,\hat k=1}^{3}  a^{\ts}_{ii'} (\frac{\partial a^{\ts}_{i \hat i}}{\partial x_{i'}} \frac{\partial a^{\ts}_{2\hat k}}{\partial x_{\hat i}}\frac{\partial f}{\partial x_{\hat k}}) (a^{\ts}\nabla)_2f=0\\
\cI_2&=&\sum_{i=2}^n\sum_{i',\hat i,\hat k=1}^{3} a^{\ts}_{ii'}a^{\ts}_{i \hat i} (\frac{\partial }{\partial x_{i'}} \frac{\partial a^{\ts}_{1\hat k}}{\partial x_{\hat i}})(\frac{\partial f}{\partial x_{\hat k}}) (a^{\ts}\nabla)_1f+\sum_{i=2}^n\sum_{i',\hat i,\hat k=1}^{3} a^{\ts}_{ii'}a^{\ts}_{i \hat i} (\frac{\partial }{\partial x_{i'}} \frac{\partial a^{\ts}_{2\hat k}}{\partial x_{\hat i}})(\frac{\partial f}{\partial x_{\hat k}})(a^{\ts}\nabla)_2f=0\\
\cI_3&=&-\sum_{i=1}^2\sum_{i',\hat i,\hat k=1}^{3}  a^{\ts}_{1\hat k}\frac{\partial a^{\ts}_{ii'}}{\partial x_{\hat k}} \frac{\partial a^{\ts}_{i \hat i}}{\partial x_{i'}} \frac{\partial f}{\partial x_{\hat i}})(a^{\ts}\nabla)_1f-\sum_{i=1}^2\sum_{i',\hat i,\hat k=1}^{3}  a^{\ts}_{2\hat k}\frac{\partial a^{\ts}_{ii'}}{\partial x_{\hat k}} \frac{\partial a^{\ts}_{i \hat i}}{\partial x_{i'}} \frac{\partial f}{\partial x_{\hat i}})(a^{\ts}\nabla)_2f=0\\
\cI_4&=&-\sum_{i=1}^2\sum_{i',\hat i,\hat k=1}^{3}  a^{\ts}_{1\hat k} a^{\ts}_{ii'} (\frac{\partial }{\partial x_{\hat k}} \frac{\partial a^{\ts}_{i \hat i}}{\partial x_{i'}}) \frac{\partial f}{\partial x_{\hat i}}(a^{\ts}\nabla)_1f-\sum_{i=1}^2\sum_{i',\hat i,\hat k=1}^{3}  a^{\ts}_{2\hat k} a^{\ts}_{ii'} (\frac{\partial }{\partial x_{\hat k}} \frac{\partial a^{\ts}_{i \hat i}}{\partial x_{i'}}) \frac{\partial f}{\partial x_{\hat i}}(a^{\ts}\nabla)_2f=0.
\eeaa
Similar computation applies to the tensor terms $\mathfrak{R}_{\rho^*}$ and $\mathfrak{R}_{zb}$. Since $z$ is a constant matrix, we get
\beaa
\mathfrak{R}_{zb}(\nabla f,\nabla f)&=&-2\sum_{i=1}^m \sum_{\hat i,\hat k=1}^{n+m}\la (z^T_{i\hat i}\frac{\pa b_{\hat k}}{\pa x_{\hat i}}\frac{\pa f}{\pa x_{\hat k}}-b_{\hat k}\frac{\pa z^T_{i\hat i}}{\pa x_{\hat k}} \frac{\pa f}{\pa x_{\hat i}} ),(z^T\nabla f)_i\ra_{\hR^m},\\
\mathfrak{R}_{\rho^*}&=&0.
\eeaa
We now compute the tensor terms involving the drift $b$. For the drift term in tensor $\mathfrak{R}_{ab}$, taking $b=-\frac{1}{2}aa^{\ts}\nabla V$, which means  $b=-\frac{1}{2}(a_{\hat k k}a^{\ts}_{kk'}\frac{\pa V}{\pa x_{k'}})_{\hat k=1,2,3}$ in local coordinates,
\beaa
\mathfrak{R}_b
&=&\sum_{i,k=1}^2\sum_{\hat i,\hat k,k'=1}^3\left[a^{\ts}_{i\hat i}\frac{\pa a^{\ts}_{k\hat k} }{\pa x_{\hat i}}a^{\ts}_{kk'}\frac{\pa V}{\pa x_{k'}} \frac{\pa f}{\pa x_{\hat k}}(a^{\ts}\nabla)_i f \right]\\
&&+\sum_{i,k=1}^2\sum_{\hat i,\hat k,k'=1}^3\left[a^{\ts}_{i\hat i}\frac{\pa a^{\ts}_{kk'}}{\pa x_{\hat i}}a^{\ts}_{k\hat k}\frac{\pa V}{\pa x_{k'}} \frac{\pa f}{\pa x_{\hat k}}(a^{\ts}\nabla)_i f\right]\\
&&+\sum_{i,k=1}^2\sum_{\hat i,\hat k,k'=1}^3\left[ a^{\ts}_{i\hat i}a^{\ts}_{k\hat k}a^{\ts}_{kk'}\frac{\pa^2 V}{\pa x_{\hat i}\pa x_{k'}}\frac{\pa f}{\pa x_{\hat k}}(a^{\ts}\nabla)_i f\right]\\
&&-\sum_{i,k=1}^2\sum_{\hat i,\hat k,k'=1}^3\left[a^{\ts}_{k\hat k}a^{\ts}_{kk'}\frac{\pa a^{\ts}_{i\hat i}}{\pa x_{\hat k}}\frac{\pa V}{\pa x_{k'}} \frac{\pa f}{\pa x_{\hat i}}(a^{\ts}\nabla)_if\right] \\
&=& \cJ_1+\cJ_2+\cJ_3+\cJ_4.
\eeaa
We now derive the explicit formulas for the above four terms. 
\beaa
\cJ_1
&=&\sum_{\hat i,\hat k,k'=1}^3\left[ a^{\ts}_{1\hat i}\frac{\pa a^{\ts}_{1\hat k} }{\pa x_{\hat i}}a^{\ts}_{1k'}\frac{\pa V}{\pa x_{k'}} \frac{\pa f}{\pa x_{\hat k}}(a^{\ts}\nabla)_1 f+a^{\ts}_{2\hat i}\frac{\pa a^{\ts}_{1\hat k} }{\pa x_{\hat i}}a^{\ts}_{1k'}\frac{\pa V}{\pa x_{k'}} \frac{\pa f}{\pa x_{\hat k}}(a^{\ts}\nabla)_2 f\right]\\
&&+\sum_{\hat i,\hat k,k'=1}^3\left[ a^{\ts}_{1\hat i}\frac{\pa a^{\ts}_{2\hat k} }{\pa x_{\hat i}}a^{\ts}_{2k'}\frac{\pa V}{\pa x_{k'}} \frac{\pa f}{\pa x_{\hat k}}(a^{\ts}\nabla)_1 f+a^{\ts}_{2\hat i}\frac{\pa a^{\ts}_{2\hat k} }{\pa x_{\hat i}}a^{\ts}_{2k'}\frac{\pa V}{\pa x_{k'}} \frac{\pa f}{\pa x_{\hat k}}(a^{\ts}\nabla)_2 f\right]\\
&=&-\frac{1}{2}(a^{\ts}\nabla)_1V\pa_zf(a^{\ts}\nabla)_2f +\frac{1}{2}(a^{\ts}\nabla)_2V\pa_z f(a^{\ts }\nabla)_1f;
\eeaa

\beaa
\cJ_2&=&\sum_{\hat i,\hat k,k'=1}^3\left[a^{\ts}_{1\hat i}\frac{\pa a^{\ts}_{1k'}}{\pa x_{\hat i}}a^{\ts}_{1\hat k}\frac{\pa V}{\pa x_{k'}} \frac{\pa f}{\pa x_{\hat k}}(a^{\ts}\nabla)_1 f+a^{\ts}_{2\hat i}\frac{\pa a^{\ts}_{1k'}}{\pa x_{\hat i}}a^{\ts}_{1\hat k}\frac{\pa V}{\pa x_{k'}} \frac{\pa f}{\pa x_{\hat k}}(a^{\ts}\nabla)_2 f\right]\\
&&\sum_{\hat i,\hat k,k'=1}^3\left[a^{\ts}_{1\hat i}\frac{\pa a^{\ts}_{2k'}}{\pa x_{\hat i}}a^{\ts}_{2\hat k}\frac{\pa V}{\pa x_{k'}} \frac{\pa f}{\pa x_{\hat k}}(a^{\ts}\nabla)_1 f+a^{\ts}_{2\hat i}\frac{\pa a^{\ts}_{2k'}}{\pa x_{\hat i}}a^{\ts}_{2\hat k}\frac{\pa V}{\pa x_{k'}} \frac{\pa f}{\pa x_{\hat k}}(a^{\ts}\nabla)_2 f\right]\\
&=&-\frac{1}{2}\frac{\pa V}{\pa z}(a^{\ts}\nabla)_1f(a^{\ts}\nabla)_2f+\frac{1}{2}\frac{\pa V}{\pa z}(a^{\ts}\nabla)_1f(a^{\ts}\nabla)_2f=0;
\eeaa
\beaa
\cJ_3&=&\sum_{\hat i,\hat k,k'=1}^3\left[ a^{\ts}_{1\hat i}a^{\ts}_{1\hat k}a^{\ts}_{1k'}\frac{\pa^2 V}{\pa x_{\hat i}\pa x_{k'}}\frac{\pa f}{\pa x_{\hat k}}(a^{\ts}\nabla)_1 f+a^{\ts}_{2\hat i}a^{\ts}_{1\hat k}a^{\ts}_{1k'}\frac{\pa^2 V}{\pa x_{\hat i}\pa x_{k'}}\frac{\pa f}{\pa x_{\hat k}}(a^{\ts}\nabla)_2 f\right]\\
&&\sum_{\hat i,\hat k,k'=1}^3\left[ a^{\ts}_{1\hat i}a^{\ts}_{2\hat k}a^{\ts}_{2k'}\frac{\pa^2 V}{\pa x_{\hat i}\pa x_{k'}}\frac{\pa f}{\pa x_{\hat k}}(a^{\ts}\nabla)_1 f+a^{\ts}_{2\hat i}a^{\ts}_{2\hat k}a^{\ts}_{2k'}\frac{\pa^2 V}{\pa x_{\hat i}\pa x_{k'}}\frac{\pa f}{\pa x_{\hat k}}(a^{\ts}\nabla)_2 f\right]\\
&=&\sum_{\hat i,k'=1}^3\left[ a^{\ts}_{1\hat i}a^{\ts}_{1k'}\frac{\pa^2 V}{\pa x_{\hat i}\pa x_{k'}} |(a^{\ts}\nabla)_1 f|^2+a^{\ts}_{2\hat i}a^{\ts}_{1k'}\frac{\pa^2 V}{\pa x_{\hat i}\pa x_{k'}}(a^{\ts}\nabla)_1 f(a^{\ts}\nabla)_2 f\right]\\
&&\sum_{\hat i,k'=1}^3\left[ a^{\ts}_{1\hat i}a^{\ts}_{2k'}\frac{\pa^2 V}{\pa x_{\hat i}\pa x_{k'}}(a^{\ts}\nabla)_2 f(a^{\ts}\nabla)_1 f+a^{\ts}_{2\hat i}a^{\ts}_{2k'}\frac{\pa^2 V}{\pa x_{\hat i}\pa x_{k'}}|(a^{\ts}\nabla)_2 f|^2\right]\\
&=&\Big[\frac{\pa^2 V}{\pa x\pa x}+\frac{y^2}{4}\frac{\pa^2 V}{\pa z\pa z}-y\frac{\pa^2 V}{\pa x\pa z}\Big] |(a^{\ts}\nabla)_1 f|^2+\Big[\frac{\pa^2 V}{\pa y\pa y}+\frac{x^2}{4}\frac{\pa^2 V}{\pa z\pa z}+x\frac{\pa^2 V}{\pa y\pa z}\Big] |(a^{\ts}\nabla)_2 f|^2\\
&&+2\Big[\frac{\pa^2 V}{\pa x\pa y}+\frac{x}{2}\frac{\pa^2 V}{\pa x\pa z}-\frac{y}{2}\frac{\pa^2 V}{\pa y\pa z}-\frac{xy}{4}\frac{\pa^2 V}{\pa z\pa z}\Big] (a^{\ts}\nabla)_1 f(a^{\ts}\nabla)_2 f;
\eeaa

\beaa
\cJ_4&=&-\sum_{\hat i,\hat k,k'=1}^3\left[a^{\ts}_{1\hat k}a^{\ts}_{1k'}\frac{\pa a^{\ts}_{1\hat i}}{\pa x_{\hat k}}\frac{\pa V}{\pa x_{k'}} \frac{\pa f}{\pa x_{\hat i}}(a^{\ts}\nabla)_1f+a^{\ts}_{1\hat k}a^{\ts}_{1k'}\frac{\pa a^{\ts}_{2\hat i}}{\pa x_{\hat k}}\frac{\pa V}{\pa x_{k'}} \frac{\pa f}{\pa x_{\hat i}}(a^{\ts}\nabla)_2f\right]\\
&&-\sum_{\hat i,\hat k,k'=1}^3\left[a^{\ts}_{2\hat k}a^{\ts}_{2k'}\frac{\pa a^{\ts}_{1\hat i}}{\pa x_{\hat k}}\frac{\pa V}{\pa x_{k'}} \frac{\pa f}{\pa x_{\hat i}}(a^{\ts}\nabla)_1f+a^{\ts}_{2\hat k}a^{\ts}_{2k'}\frac{\pa a^{\ts}_{2\hat i}}{\pa x_{\hat k}}\frac{\pa V}{\pa x_{k'}} \frac{\pa f}{\pa x_{\hat i}}(a^{\ts}\nabla)_2f\right]\\
&=&-\frac{1}{2}(a^{\ts}\nabla)_1V\pa_z f(a^{\ts}\nabla)_2f+\frac{1}{2}(a^{\ts}\nabla)_2V\pa_z f(a^{\ts}\nabla)_1f.
\eeaa
Summing up the above formulas, we get $\mathfrak{R}_{ab}$. We now compute the drift tensor term of $\mathfrak{R}_{zb}$.
By taking $b=-\frac{1}{2}aa^{\ts}\nabla V$, we have
\beaa
\mathfrak{R}_{zb}(\nabla f,\nabla f) 
&=&-\sum_{\hat i,\hat k=1}^{3}\left[ z^{\ts}_{1\hat i}\frac{\pa b_{\hat k}}{\pa x_{\hat i}}\frac{\pa f}{\pa x_{\hat k}}(z^{\ts}\nabla f)_1-b_{\hat k}\frac{\pa z^{\ts}_{i\hat i}}{\pa x_{\hat k}} \frac{\pa f}{\pa x_{\hat i}} (z^{\ts}\nabla f)_1\right]\\ 
&=&\sum_{k=1}^2\sum_{\hat i,\hat k,k'=1}^3\left[z^{\ts}_{1\hat i}\frac{\pa a^{\ts}_{k\hat k} }{\pa x_{\hat i}}a^{\ts}_{kk'}\frac{\pa V}{\pa x_{k'}} \frac{\pa f}{\pa x_{\hat k}}(z^{\ts}\nabla)_1 f \right]\\
&&+\sum_{k=1}^2\sum_{\hat i,\hat k,k'=1}^3\left[z^{\ts}_{1\hat i}\frac{\pa a^{\ts}_{kk'}}{\pa x_{\hat i}}a^{\ts}_{k\hat k}\frac{\pa V}{\pa x_{k'}} \frac{\pa f}{\pa x_{\hat k}}(z^{\ts}\nabla)_1 f\right]\\
&&+\sum_{k=1}^2\sum_{\hat i,\hat k,k'=1}^3\left[ z^{\ts}_{1\hat i}a^{\ts}_{k\hat k}a^{\ts}_{kk'}\frac{\pa^2 V}{\pa x_{\hat i}\pa x_{k'}}\frac{\pa f}{\pa x_{\hat k}}(z^{\ts}\nabla)_1 f\right]\\
&&-\sum_{k=1}^2\sum_{\hat i,\hat k,k'=1}^3\left[a^{\ts}_{k\hat k}a^{\ts}_{kk'}\frac{\pa z^{\ts}_{1\hat i}}{\pa x_{\hat k}}\frac{\pa V}{\pa x_{k'}} \frac{\pa f}{\pa x_{\hat i}}(z^{\ts}\nabla)_1f\right] \\
&=& \cJ_1^z+\cJ_2^z+\cJ_3^z+\cJ_4^z.
\eeaa
We further compute as blow by taking advantage of the constant matrix $z$,
\beaa
\cJ_1^z&=&\sum_{k=1}^2\sum_{\hat i,\hat k,k'=1}^3\left[z^{\ts}_{1\hat i}\frac{\pa a^{\ts}_{k\hat k} }{\pa x_{\hat i}}a^{\ts}_{kk'}\frac{\pa V}{\pa x_{k'}} \frac{\pa f}{\pa x_{\hat k}}(z^{\ts}\nabla)_1 f \right]=0;\\
\cJ_2^z&=&\sum_{k=1}^2\sum_{\hat i,\hat k,k'=1}^3\left[z^{\ts}_{1\hat i}\frac{\pa a^{\ts}_{kk'}}{\pa x_{\hat i}}a^{\ts}_{k\hat k}\frac{\pa V}{\pa x_{k'}} \frac{\pa f}{\pa x_{\hat k}}(z^{\ts}\nabla)_1 f\right]=0;\\
\cJ_4^z&=&-\sum_{k=1}^2\sum_{\hat i,\hat k,k'=1}^3\left[a^{\ts}_{k\hat k}a^{\ts}_{kk'}\frac{\pa z^{\ts}_{1\hat i}}{\pa x_{\hat k}}\frac{\pa V}{\pa x_{k'}} \frac{\pa f}{\pa x_{\hat i}}(z^{\ts}\nabla)_1f\right]=0\\
\cJ_3^z&=&\sum_{k=1}^2\sum_{\hat i,\hat k,k'=1}^3\left[ z^{\ts}_{1 \hat i}a^{\ts}_{k\hat k}a^{\ts}_{kk'}\frac{\pa^2 V}{\pa x_{\hat i}\pa x_{k'}}\frac{\pa f}{\pa x_{\hat k}}(z^{\ts}\nabla)_1 f\right]\\
&=& \Big(\frac{\pa^2 V}{\pa x\pa z }-\frac{y}{2}\frac{\pa^2 V}{\pa z\pa z}\Big)(a^{\ts}\nabla)_1 f (z^{\ts}\nabla)_1 f+\Big(\frac{\pa^2 V}{\pa y\pa z}+\frac{x}{2}\frac{\pa^2 V}{\pa z\pa z}\Big)(z^{\ts}\nabla)_1 f(a^{\ts}\nabla)_2 f. 
\eeaa
The proof is thus completed.\qed 
\end{proof}{}

\subsection{Proof Of Proposition \ref{prop se2}}
By routine computations, we derive the following lemma. 
\begin{lem}\label{lemma: vector and mat for se2}
For displacement group $\textbf{SE}(2)$, we have 
\beaa
Q&=&\left(
\begin{array}{ccccccccc}
 1 & 0 & 0 & 0 & 0 & 0 & 0 & 0 & 0 \\
 0 & e^{\beta \theta} & 1 & 0 & 0 & 0 & 0 & 0 & 0 \\
 0 & 0 & 0 & e^{\beta \theta} & 0 & 0 & 1 & 0 & 0 \\
 0 & 0 & 0 & 0 & e^{2 \beta \theta} & e^{\beta \theta} & 0 & e^{\beta \theta} & 1 \\
\end{array}
\right);\\
P&=&\left(
\begin{array}{ccccccccc}
 0 & 0 & 0 & 0 & 0 & 0 & -g(\theta,x,y) & 0 & 0 \\
 0 & 0 & 0 & 0 & 0 & 0 & 0 & -g(\theta,x,y) e^{\beta \theta} & -g(\theta,x,y) \\
\end{array}
\right);\\
D^{\ts}&=&(0,\beta e^{\beta\theta}\pa_xf ,0,0),\quad E^{\ts}=(- \pa_yf \pa_{\theta}g , -\pa_yf \pa_yg-e^{\beta \theta} \pa_yf \pa_xg);\\
C^{\ts}&=&(0,\beta e^{\beta \theta}\pa_y f+\beta e^{2\beta\theta} \pa_xf,0,0,-\beta e^{2\beta \theta}\pa_{\theta} f,-\beta e^{\beta \theta}\pa_{\theta} f,0,0,0).
\eeaa
$F=\left(
\begin{array}{c}
 0 \\
 0 \\
 g\pa_{\theta}g \pa_yf \\
 0 \\
 0 \\
 e^{\beta \theta} g \pa_yf \pa_yg+e^{2 \beta \theta} g \pa_yf \pa_x g \\
 0 \\
 0 \\
 g \pa_yf \pa_yg+e^{\beta \theta} g \pa_y f \pa_x g \\
\end{array}
\right),\quad G=\left(
\begin{array}{c}
 0 \\
 0 \\
 -2 g \pa_yf \pa_{\theta}g \\
 0 \\
 0 \\
 -2 e^{\beta \theta} g \pa_yf \pa_yg-2 e^{2 \beta \theta} g \pa_yf \pa_xg\\
 0 \\
 0 \\
 -2 g \pa_yf \pa_yg-2 e^{\beta \theta} g \pa_yf \pa_xg \\
\end{array}
\right)$.
\end{lem}{}

The proof of Proposition \ref{prop se2} follows from the two lemmas below.
\begin{lem}\label{SE2 hess} On the displacement group, we have
\beaa
[QX+D]^{\ts}[QX+D]+[PX+E]^{\ts}[PX+E]+2[C^{\ts}+F^{\ts}+G^{\ts}]X
=|\mathfrak{Hess}_{a,z}^{G}f|^2+\mathfrak{R}^{G}(\nabla f,\nabla f).
\eeaa
In particular, we have 
\beaa
|\mathfrak{Hess}_{a,z}^{G}f|^2&=&[X+\Lambda_1]^{\ts}Q^{\ts}Q[X+\Lambda_1]+[X+\Lambda_2]^{\ts}P^{\ts}P[X+\Lambda_2];\\
\Lambda_1^{\ts}&=& (0,\beta\pa_xf,\frac{\beta\pa_yf}{2},\beta\pa_xf,0,0,\frac{\beta\pa_yf}{2},0,-\beta\pa_{\theta}f); \\
\Lambda_2^{\ts}&=& (0,0,0,0,0,0,\lambda_6,0,\lambda_9);\\
\lambda_6&=&\frac{\pa_{\theta}g\pa_yf}{g}-\frac{\beta(a^{\ts}\nabla)_2f}{g^2}-\frac{\pa_{\theta}g\pa_yf}{g};\\ \lambda_9&=&\frac{(a^{\ts}\nabla)_2g\pa_yf}{g}+\frac{\beta \pa_{\theta}f}{g^2}-\frac{(a^{\ts}\nabla)_2g\pa_yf}{g};\\
\mathfrak{R}^{G}(\nabla f,\nabla f)&=&\Gamma_1(\log g,\log g)\Gamma_1^z(f,f)-\beta^2(1+\frac{1}{g^2})\Gamma_1(f,f)+\frac{\beta^2}{2g^2}\Gamma_1^z(f,f)
\eeaa
\end{lem}{}
\begin{lem}\label{SE2 tensor}
By routine computations, we obtain
\beaa
\mathfrak{R}_{ab}(\nabla f,\nabla f)&=&\beta^2e^{\beta \theta}\frac{\pa f}{\pa x}(a^{\ts}\nabla )_2f +\beta e^{\beta\theta} (a^{\ts}\nabla)_2V\frac{\pa f}{\pa x}(a^{\ts}\nabla )_1 f  +\beta e^{\beta \theta}\frac{\pa V}{\pa x}(a^{\ts}\nabla)_2f(a^{\ts}\nabla)_1f \\
&&+\frac{\pa^2 V}{\pa \theta\pa \theta} |(a^{\ts}\nabla)_1 f|^2+2(e^{\beta\theta}\frac{\pa^2 V}{\pa \theta\pa x}+\frac{\pa^2 V}{\pa \theta \pa y} )(a^{\ts}\nabla)_1 f(a^{\ts}\nabla)_2 f\\
&&+\sum_{\hat i,k'=1}^3a^{\ts}_{2\hat i}a^{\ts}_{2k'}\frac{\pa^2 V}{\pa x_{\hat i}\pa x_{k'}}|(a^{\ts}\nabla)_2 f|^2-\beta e^{\beta\theta}(a^{\ts}\nabla)_1V \frac{\pa f}{\pa x}(a^{\ts}\nabla)_2f;\\
\mathfrak{R}_{zb}(\nabla f,\nabla f)&=&\sum_{i=1}^2\sum_{i',\hat i=1}^{3}  a^{\ts}_{ii'}a^{\ts}_{i \hat i} \frac{\pa^2 z^{\ts}_{1\hat k}}{\partial x_{i'}\partial x_{\hat i}}\pa_yf (z^{\ts}\nabla)_1f -\sum_{k=1}^2(a^{\ts}\nabla)_kz^{\ts}_{13}(a^{\ts}\nabla)_kV\pa_yf(a^{\ts}\nabla)_1f
\\ &&-g\frac{\pa^2 V}{\pa \theta\pa y}|(a^{\ts}\nabla)_1 f|^2-g(e^{\beta\theta}\frac{\pa^2 V}{\pa x\pa y}+\frac{\pa^2 V}{\pa y\pa y} )(a^{\ts}\nabla)_2 f(a^{\ts}\nabla)_1 f;\\
\mathfrak{R}_{\rho^*}(\nabla f,\nabla f)&=& -2\sum_{l=1}^2\sum_{l',\hat l=1}^3a^{\ts}_{ll'}a^{\ts}_{l\hat l}\frac{\pa^2 z^{\ts}_{13}}{\pa x_{l'}\pa x_{\hat l}}\pa_yf (z^{\ts}\nabla)_1f-2\sum_{l=1}^2\sum_{l',\hat l=1}^3a^{\ts}_{ll'}a^{\ts}_{l\hat l}\frac{\pa z^{\ts}_{13}}{\pa x_{\hat l}}\frac{\pa z^{\ts}_{13}}{\pa x_{l'}} |\pa_yf|^2\\
&&-2 \sum_{l=1}^2\sum_{\hat l=1}^3(a^{\ts}\nabla)_l\log \rho^* a^{\ts}_{l\hat l}\frac{\pa z^{\ts}_{13}}{\pa x_{\hat l}}\pa_yf (z^{\ts}\nabla)_1f.
\eeaa
\end{lem}{}
\begin{proof}[Proof of Lemma \ref{SE2 hess}]
According to Lemma \ref{lemma: vector and mat for se2} and observe the fact that $G=-2F$ and $(a^{\ts}\nabla)_2f=e^{\beta\theta}\pa_xf+\pa_yf$, we first have 
\beaa
2C^{\ts}X&=&2[\beta e^{\beta \theta}\pa_y f+\beta e^{2\beta\theta} \pa_xf]\frac{\pa^2 f}{\pa \theta \pa x}+2[-\beta e^{2\beta \theta}\pa_{\theta} f]\frac{\pa^2 f}{\pa x \pa x}+2[-\beta e^{\beta \theta}\pa_{\theta} f]\frac{\pa^2 f}{\pa x\pa y};\\
2[F^{\ts}+G^{\ts}]X&=&-2\Big(g\pa_{\theta}g\pa_yf\frac{\pa^2 f}{\pa \theta \pa y}+e^{\beta\theta}g (a^{\ts}\nabla)_2g\pa_yf\frac{\pa^2 f}{\pa x\pa y}+g(a^{\ts}\nabla)_2g\pa_yf\frac{\pa^2 f}{\pa y\pa y} \Big).
\eeaa
By direct computations, we end up with 
\beaa
&&[QX+D]^{\ts}[QX+D]+ [PX+E]^{\ts}[PX+E]+2C^{\ts}X+2F^{\ts}X+2G^{\ts}X\\
&=&
\left[ \frac{\pa^2 f}{\pa \theta\pa \theta}\right]^2+\left[ e^{2\beta\theta} \frac{\pa^2 f}{\pa x\pa x}+2e^{\beta\theta}\frac{\pa^2 f}{\pa x\pa y}+\frac{\pa^2 f}{\pa y\pa y}\right]^2+\left[e^{\beta \theta}\frac{\pa^2 f}{\pa \theta \pa x}+\frac{\pa^2 f}{\pa \theta\pa y}+\beta e^{\beta\theta}\frac{\pa f}{\pa x} \right]^2\\
&&+\left[e^{\beta \theta}\frac{\pa^2 f}{\pa \theta \pa x}+\frac{\pa^2 f}{\pa \theta\pa y} \right]^2+\left[-g\frac{\pa^2 f}{\pa \theta\pa y} -\pa_yf\pa_{\theta}g\right]^2+ \left[-g e ^{\beta\theta}\frac{\pa^2 f}{\pa x\pa y} -g\frac{\pa^2f}{\pa y\pa y}-(a^{\ts}\nabla)_2g \pa_yf\right]^2\\
&&+2[\beta e^{\beta \theta}\pa_y f+\beta e^{2\beta\theta} \pa_xf]\frac{\pa^2 f}{\pa \theta \pa x}+2[-\beta e^{2\beta \theta}\pa_{\theta} f]\frac{\pa^2 f}{\pa x \pa x}+2[-\beta e^{\beta \theta}\pa_{\theta} f]\frac{\pa^2 f}{\pa x\pa y}\\
&&-2\Big(g\pa_{\theta}g\pa_yf\frac{\pa^2 f}{\pa \theta \pa y}+e^{\beta\theta}g (a^{\ts}\nabla)_2g\pa_yf\frac{\pa^2 f}{\pa x\pa y}+g(a^{\ts}\nabla)_2g\pa_yf\frac{\pa^2 f}{\pa y\pa y} \Big)\\
&=&\left[ \frac{\pa^2 f}{\pa \theta\pa \theta}\right]^2+\left[ e^{2\beta\theta} \frac{\pa^2 f}{\pa x\pa x}+2e^{\beta\theta}\frac{\pa^2 f}{\pa x\pa y}+\frac{\pa^2 f}{\pa y\pa y}\right]^2+\left[e^{\beta \theta}\frac{\pa^2 f}{\pa \theta \pa x}+\frac{\pa^2 f}{\pa \theta\pa y}+\beta e^{\beta\theta}\frac{\pa f}{\pa x} \right]^2\\
&&+\left[e^{\beta \theta}\frac{\pa^2 f}{\pa \theta \pa x}+\frac{\pa^2 f}{\pa \theta\pa y} \right]^2+\left[-g\frac{\pa^2 f}{\pa \theta\pa y} -\pa_yf\pa_{\theta}g\right]^2+ \left[-g e ^{\beta\theta}\frac{\pa^2 f}{\pa x\pa y} -g\frac{\pa^2f}{\pa y\pa y}-(a^{\ts}\nabla)_2g \pa_yf\right]^2\\
&&+2\beta (a^{\ts}\nabla)_2f\Big[ e^{\beta \theta} \frac{\pa^2 f}{\pa \theta \pa x}+\frac{\pa^2 f}{\pa \theta \pa y}\Big]-2\beta (a^{\ts}\nabla)_2f\frac{\pa^2 f}{\pa \theta \pa y}-2g\pa_{\theta}g\pa_yf\frac{\pa^2 f}{\pa \theta \pa y}\\
&&-2\beta \pa_{\theta}f\Big[2e^{\beta\theta}\frac{\pa^2 f}{\pa x\pa y}+e^{2\beta\theta}\frac{\pa^2 f}{\pa x\pa x}+\frac{\pa^2 f}{\pa y\pa y} \Big]+2\beta \pa_{\theta}f\Big[e^{\beta\theta}\frac{\pa^2 f}{\pa x\pa y}+\frac{\pa^2 f}{\pa y\pa y} \Big]\\
&&-2g(a^{\ts}\nabla)_2g\pa_yf\Big[e^{\beta\theta}\frac{\pa^2f}{\pa x\pa y}+\frac{\pa^2 f}{\pa y\pa y}\Big].
\eeaa
Complete square for the above terms, we end up with 
\beaa 
&&[QX+D]^{\ts}[QX+D]+[PX+E]^{\ts}[PX+E]+2C^{\ts}X+2F^{\ts}X+2G^{\ts}X\\
&=&\left[ \frac{\pa^2 f}{\pa \theta\pa \theta}\right]^2+\left[ e^{2\beta\theta} \frac{\pa^2 f}{\pa x\pa x}+2e^{\beta\theta}\frac{\pa^2 f}{\pa x\pa y}+\frac{\pa^2 f}{\pa y\pa y}-\beta \pa_{\theta}f\right]^2-\beta^2|\pa_{\theta}f|^2\\
&&+\left[e^{\beta \theta}\frac{\pa^2 f}{\pa \theta \pa x}+\frac{\pa^2 f}{\pa \theta\pa y}+\beta e^{\beta\theta}\frac{\pa f}{\pa x} \right]^2
+\left[e^{\beta \theta}\frac{\pa^2 f}{\pa \theta \pa x}+\frac{\pa^2 f}{\pa \theta\pa y}+\beta (a^{\ts}\nabla)_2f \right]^2-\beta^2|(a^{\ts}\nabla)_2f|^2\\
&&+\Big[g\frac{\pa^2 f}{\pa \theta\pa y}+\pa_{\theta}g\pa_yf-\frac{\beta(a^{\ts}\nabla)_2f}{g}-\pa_{\theta}g\pa_yf \Big]^2-\Big[\frac{\beta(a^{\ts}\nabla)_2f}{g}+\pa_{\theta}g\pa_yf \Big]^2\\
&&+\Big[ge^{\beta\theta}\frac{\pa^2 f}{\pa x\pa y}+g\frac{\pa^2 f}{\pa y\pa y} +(a^{\ts}\nabla)_2g\pa_yf+\frac{\beta \pa_{\theta}f}{g}-(a^{\ts}\nabla)_2g\pa_yf \Big]^2-\Big[\frac{\beta \pa_{\theta}f}{g}-(a^{\ts}\nabla)_2g\pa_yf \Big]^2\\
&&+2\Big[\frac{\beta(a^{\ts}\nabla)_2f}{g}+\pa_{\theta}g\pa_yf \Big]\pa_{\theta}g\pa_yf-2\pa_yf(a^{\ts}\nabla)_2g\times \Big[ \frac{\beta \pa_{\theta}f}{g}-(a^{\ts}\nabla)_2g\pa_yf \Big]
\eeaa
The first order terms generate tensor $\mathfrak{R}^G(\nabla f,\nabla f)$ and the sum of square terms generate vectors $\Lambda_1$ and $\Lambda_2$. We further formulate the above two terms as below 
\beaa
&&\left[e^{\beta \theta}\frac{\pa^2 f}{\pa \theta \pa x}+\frac{\pa^2 f}{\pa \theta\pa y}+\beta e^{\beta\theta}\frac{\pa f}{\pa x} \right]^2
+\left[e^{\beta \theta}\frac{\pa^2 f}{\pa \theta \pa x}+\frac{\pa^2 f}{\pa \theta\pa y}+\beta (a^{\ts}\nabla)_2f \right]^2\\
&=&2\left[e^{\beta \theta}\frac{\pa^2 f}{\pa \theta \pa x}+\frac{\pa^2 f}{\pa \theta\pa y}+\beta e^{\beta\theta}\frac{\pa f}{\pa x}+\frac{\beta}{2}\pa_yf \right]^2+\frac{\beta^2}{2}|\pa_yf|^2.
\eeaa
Adding $\frac{\beta^2}{2}|\pa_yf|^2$ into the term $\mathfrak{R}^G(\nabla f,\nabla f)$ again, we further expand the tensor term $\mathfrak{R}^G(\nabla f,\nabla f)$ below, 
\beaa
&&\mathfrak{R}^G(\nabla f,\nabla f)\\
&=&-\beta^2[|\pa_{\theta}f|^2+|(a^{\ts}\nabla)_2f|^2]-\Big[\frac{\beta \pa_{\theta}f}{g}-(a^{\ts}\nabla)_2g\pa_yf \Big]^2-\Big[\frac{\beta(a^{\ts}\nabla)_2f}{g}+\pa_{\theta}g\pa_yf \Big]^2\\
&&+2\Big[\frac{\beta(a^{\ts}\nabla)_2f}{g}+\pa_{\theta}g\pa_yf \Big]\pa_{\theta}g\pa_yf-2\pa_yf(a^{\ts}\nabla)_2g\times \Big[ \frac{\beta \pa_{\theta}f}{g}-(a^{\ts}\nabla)_2g\pa_yf \Big]+\frac{\beta^2}{2}|\pa_yf|^2\\
&=& -\beta^2\Gamma_1(f,f)-\frac{\beta^2}{g^2}|(a^{\ts}\nabla)_1f|^2
-|(a^{\ts}\nabla)_2(\log g)|^2 |(z^{\ts}\nabla)_1f|^2-2\frac{\beta}{g}(a^{\ts}\nabla)_2\log g(a^{\ts}\nabla)_1f(z^{\ts}\nabla)_1f\\
&&-\frac{\beta^2}{g^2}|(a^{\ts}\nabla)_2f|^2-|(a^{\ts}\nabla)_1\log g|^2|(z^{\ts}\nabla)_1f|^2+2\frac{\beta}{g} (a^{\ts}\nabla)_1\log g (a^{\ts}\nabla)_2f(z^{\ts}\nabla)_1f\\
&&-2\frac{\beta}{g}(a^{\ts}\nabla)_1\log g(a^{\ts}\nabla)_2f (z^{\ts}\nabla)_1f
+2|(a^{\ts}\nabla)_1\log g|^2|(z^{\ts}\nabla)_1f|^2\\
&&+2\frac{\beta}{g}(a^{\ts}\nabla)_2\log g(a^{\ts}\nabla)_1f(z^{\ts}\nabla)_1f+2|(a^{\ts}\nabla)_2\log g|^2|(z^{\ts}\nabla)_1f|^2+\frac{\beta^2}{2g^2}\Gamma_1^z(f,f).
\eeaa
By grouping the bilinear terms of $\nabla f$, we get 
\beaa
\mathfrak{R}^G(\nabla f,\nabla f)
=\Gamma_1(\log g,\log g)\Gamma_1^z(f,f)-\beta^2(1+\frac{1}{g^2})\Gamma_1(f,f)+\frac{\beta^2}{2g^2}\Gamma_1^z(f,f).
\eeaa
\qed
\end{proof}{}

We are now left to compute the three tensor terms.

\begin{proof}[Proof of Lemma \ref{SE2 tensor}]
For $\mathbf{SE}(2)$, we have $n=2$ and $m=1$. Recall from Theorem \ref{thm1}, we denote 
$\mathfrak{R}_{ab}(\nabla f,\nabla f)=\mathfrak{R}_a(\nabla f,\nabla f)+\mathfrak{R}_b(\nabla f,\nabla f)$ where $\mathfrak{R}_b(\nabla f,\nabla f)$ represents the tensor term involving drift $b$. We thus have
\beaa
\mathfrak{R}_a(\nabla f,\nabla f)
&=&\sum_{i,k=1}^2\sum_{i',\hat i,\hat k=1}^{3} \la a^{\ts}_{ii'} (\frac{\partial a^{\ts}_{i \hat i}}{\partial x_{i'}} \frac{\partial a^{\ts}_{k\hat k}}{\partial x_{\hat i}}\frac{\partial f}{\partial x_{\hat k}}) ,(a^{\ts}\nabla)_kf\ra_{\hR^2}  \nonumber\\
&&+\sum_{i,k=2}^n\sum_{i',\hat i,\hat k=1}^{3} \la a^{\ts}_{ii'}a^{\ts}_{i \hat i} (\frac{\partial }{\partial x_{i'}} \frac{\partial a^{\ts}_{k\hat k}}{\partial x_{\hat i}})(\frac{\partial f}{\partial x_{\hat k}}) ,(a^{\ts}\nabla)_kf\ra_{\hR^2} \nonumber\\
&&-\sum_{i,k=1}^2\sum_{i',\hat i,\hat k=1}^{3} \la a^{\ts}_{k\hat k}\frac{\partial a^{\ts}_{ii'}}{\partial x_{\hat k}} \frac{\partial a^{\ts}_{i \hat i}}{\partial x_{i'}} \frac{\partial f}{\partial x_{\hat i}})  ,(a^{\ts}\nabla)_kf\ra_{\hR^2}  \nonumber\\
&&-\sum_{i,k=1}^2\sum_{i',\hat i,\hat k=1}^{3} \la a^{\ts}_{k\hat k} a^{\ts}_{ii'} (\frac{\partial }{\partial x_{\hat k}} \frac{\partial a^{\ts}_{i \hat i}}{\partial x_{i'}}) \frac{\partial f}{\partial x_{\hat i}}  ,(a^{\ts}\nabla)_kf\ra_{\hR^2},  \nonumber\\
&=& \cI_1+\cI_2+\cI_3+\cI_4.
\eeaa
By direct computations, we have
\beaa
\cI_1&=&\sum_{i=1}^2\sum_{i',\hat i,\hat k=1}^{3}  \Big[ a^{\ts}_{ii'} (\frac{\partial a^{\ts}_{i \hat i}}{\partial x_{i'}} \frac{\partial a^{\ts}_{1\hat k}}{\partial x_{\hat i}}\frac{\partial f}{\partial x_{\hat k}})(a^{\ts}\nabla)_1f+  a^{\ts}_{ii'} (\frac{\partial a^{\ts}_{i \hat i}}{\partial x_{i'}} \frac{\partial a^{\ts}_{2\hat k}}{\partial x_{\hat i}}\frac{\partial f}{\partial x_{\hat k}}) (a^{\ts}\nabla)_2f\Big]=0 ; \\
\cI_2&=&\sum_{i=2}^n\sum_{i',\hat i,\hat k=1}^{3} \Big[ a^{\ts}_{ii'}a^{\ts}_{i \hat i} (\frac{\partial }{\partial x_{i'}} \frac{\partial a^{\ts}_{1\hat k}}{\partial x_{\hat i}})(\frac{\partial f}{\partial x_{\hat k}}) (a^{\ts}\nabla)_1f+ a^{\ts}_{ii'}a^{\ts}_{i \hat i} (\frac{\partial }{\partial x_{i'}} \frac{\partial a^{\ts}_{2\hat k}}{\partial x_{\hat i}})(\frac{\partial f}{\partial x_{\hat k}})(a^{\ts}\nabla)_2f\Big]\\
&=&a^{\ts}_{11}a^{\ts}_{11}\frac{\pa^2}{\pa \theta \pa \theta}a^{\ts}_{22}\frac{\pa f}{\pa x}(a^{\ts}\nabla )_2f=\beta^2e^{\beta \theta}\frac{\pa f}{\pa x}(a^{\ts}\nabla )_2f;  \\
\cI_3&=&-\sum_{i=1}^2\sum_{i',\hat i,\hat k=1}^{3} \Big[ a^{\ts}_{1\hat k}\frac{\partial a^{\ts}_{ii'}}{\partial x_{\hat k}} \frac{\partial a^{\ts}_{i \hat i}}{\partial x_{i'}} \frac{\partial f}{\partial x_{\hat i}})(a^{\ts}\nabla)_1f+  a^{\ts}_{2\hat k}\frac{\partial a^{\ts}_{ii'}}{\partial x_{\hat k}} \frac{\partial a^{\ts}_{i \hat i}}{\partial x_{i'}} \frac{\partial f}{\partial x_{\hat i}})(a^{\ts}\nabla)_2f\Big]=0;\\
\cI_4&=&-\sum_{i=1}^2\sum_{i',\hat i,\hat k=1}^{3} \Big[ a^{\ts}_{1\hat k} a^{\ts}_{ii'} (\frac{\partial }{\partial x_{\hat k}} \frac{\partial a^{\ts}_{i \hat i}}{\partial x_{i'}}) \frac{\partial f}{\partial x_{\hat i}}(a^{\ts}\nabla)_1f+ a^{\ts}_{2\hat k} a^{\ts}_{ii'} (\frac{\partial }{\partial x_{\hat k}} \frac{\partial a^{\ts}_{i \hat i}}{\partial x_{i'}}) \frac{\partial f}{\partial x_{\hat i}}(a^{\ts}\nabla)_2f\Big]=0.
\eeaa
For the drift term in tensor $\mathfrak{R}_{ab}$, taking $b=-\frac{1}{2}aa^{\ts}\nabla V$, we get
\beaa
\mathfrak{R}_b
&=&\sum_{i,k=1}^2\sum_{\hat i,\hat k,k'=1}^3\left[a^{\ts}_{i\hat i}\frac{\pa a^{\ts}_{k\hat k} }{\pa x_{\hat i}}a^{\ts}_{kk'}\frac{\pa V}{\pa x_{k'}} \frac{\pa f}{\pa x_{\hat k}}(a^{\ts}\nabla)_i f \right]\\
&&+\sum_{i,k=1}^2\sum_{\hat i,\hat k,k'=1}^3\left[a^{\ts}_{i\hat i}\frac{\pa a^{\ts}_{kk'}}{\pa x_{\hat i}}a^{\ts}_{k\hat k}\frac{\pa V}{\pa x_{k'}} \frac{\pa f}{\pa x_{\hat k}}(a^{\ts}\nabla)_i f\right]\\
&&+\sum_{i,k=1}^2\sum_{\hat i,\hat k,k'=1}^3\left[ a^{\ts}_{i\hat i}a^{\ts}_{k\hat k}a^{\ts}_{kk'}\frac{\pa^2 V}{\pa x_{\hat i}\pa x_{k'}}\frac{\pa f}{\pa x_{\hat k}}(a^{\ts}\nabla)_i f\right]\\
&&-\sum_{i,k=1}^2\sum_{\hat i,\hat k,k'=1}^3\left[a^{\ts}_{k\hat k}a^{\ts}_{kk'}\frac{\pa a^{\ts}_{i\hat i}}{\pa x_{\hat k}}\frac{\pa V}{\pa x_{k'}} \frac{\pa f}{\pa x_{\hat i}}(a^{\ts}\nabla)_if\right] \\
&=& \cJ_1+\cJ_2+\cJ_3+\cJ_4.
\eeaa
Plugging into the matrices $a^{\ts}$ and $z^{\ts}$, we get
\beaa
\cJ_1
&=&\sum_{\hat i,\hat k,k'=1}^3\left[ a^{\ts}_{1\hat i}\frac{\pa a^{\ts}_{1\hat k} }{\pa x_{\hat i}}a^{\ts}_{1k'}\frac{\pa V}{\pa x_{k'}} \frac{\pa f}{\pa x_{\hat k}}(a^{\ts}\nabla)_1 f+a^{\ts}_{2\hat i}\frac{\pa a^{\ts}_{1\hat k} }{\pa x_{\hat i}}a^{\ts}_{1k'}\frac{\pa V}{\pa x_{k'}} \frac{\pa f}{\pa x_{\hat k}}(a^{\ts}\nabla)_2 f\right]\\
&&+\sum_{\hat i,\hat k,k'=1}^3\left[ a^{\ts}_{1\hat i}\frac{\pa a^{\ts}_{2\hat k} }{\pa x_{\hat i}}a^{\ts}_{2k'}\frac{\pa V}{\pa x_{k'}} \frac{\pa f}{\pa x_{\hat k}}(a^{\ts}\nabla)_1 f+a^{\ts}_{2\hat i}\frac{\pa a^{\ts}_{2\hat k} }{\pa x_{\hat i}}a^{\ts}_{2k'}\frac{\pa V}{\pa x_{k'}} \frac{\pa f}{\pa x_{\hat k}}(a^{\ts}\nabla)_2 f\right]\\
&=&\beta e^{\beta\theta} (a^{\ts}\nabla)_2V\frac{\pa f}{\pa x}(a^{\ts}\nabla )_1 f;
\eeaa
\beaa
\cJ_2
&=&\sum_{\hat i,\hat k,k'=1}^3\left[a^{\ts}_{1\hat i}\frac{\pa a^{\ts}_{1k'}}{\pa x_{\hat i}}a^{\ts}_{1\hat k}\frac{\pa V}{\pa x_{k'}} \frac{\pa f}{\pa x_{\hat k}}(a^{\ts}\nabla)_1 f+a^{\ts}_{2\hat i}\frac{\pa a^{\ts}_{1k'}}{\pa x_{\hat i}}a^{\ts}_{1\hat k}\frac{\pa V}{\pa x_{k'}} \frac{\pa f}{\pa x_{\hat k}}(a^{\ts}\nabla)_2 f\right]\\
&&+\sum_{\hat i,\hat k,k'=1}^3\left[a^{\ts}_{1\hat i}\frac{\pa a^{\ts}_{2k'}}{\pa x_{\hat i}}a^{\ts}_{2\hat k}\frac{\pa V}{\pa x_{k'}} \frac{\pa f}{\pa x_{\hat k}}(a^{\ts}\nabla)_1 f+a^{\ts}_{2\hat i}\frac{\pa a^{\ts}_{2k'}}{\pa x_{\hat i}}a^{\ts}_{2\hat k}\frac{\pa V}{\pa x_{k'}} \frac{\pa f}{\pa x_{\hat k}}(a^{\ts}\nabla)_2 f\right]\\
&=&\beta e^{\beta \theta}\frac{\pa V}{\pa x}(a^{\ts}\nabla)_2f(a^{\ts}\nabla)_1f;
\eeaa
\beaa
\cJ_3
&=&\sum_{\hat i,k'=1}^3\left[ a^{\ts}_{1\hat i}a^{\ts}_{1k'}\frac{\pa^2 V}{\pa x_{\hat i}\pa x_{k'}} |(a^{\ts}\nabla)_1 f|^2+a^{\ts}_{2\hat i}a^{\ts}_{1k'}\frac{\pa^2 V}{\pa x_{\hat i}\pa x_{k'}}(a^{\ts}\nabla)_1 f(a^{\ts}\nabla)_2 f\right]\\
&&+\sum_{\hat i,k'=1}^3\left[ a^{\ts}_{1\hat i}a^{\ts}_{2k'}\frac{\pa^2 V}{\pa x_{\hat i}\pa x_{k'}}(a^{\ts}\nabla)_2 f(a^{\ts}\nabla)_1 f+a^{\ts}_{2\hat i}a^{\ts}_{2k'}\frac{\pa^2 V}{\pa x_{\hat i}\pa x_{k'}}|(a^{\ts}\nabla)_2 f|^2\right]\\
&=&\frac{\pa^2 V}{\pa \theta\pa \theta} |(a^{\ts}\nabla)_1 f|^2+2(e^{\beta\theta}\frac{\pa^2 V}{\pa \theta\pa x}+\frac{\pa^2 V}{\pa \theta \pa y} )(a^{\ts}\nabla)_1 f(a^{\ts}\nabla)_2 f\\
&&+\sum_{\hat i,k'=1}^3a^{\ts}_{2\hat i}a^{\ts}_{2k'}\frac{pa^2 V}{\pa x_{\hat i}\pa x_{k'}})|(a^{\ts}\nabla)_2 f|^2;
\eeaa

\beaa
\cJ_4
&=&-\sum_{\hat i,\hat k,k'=1}^3\left[a^{\ts}_{1\hat k}a^{\ts}_{1k'}\frac{\pa a^{\ts}_{1\hat i}}{\pa x_{\hat k}}\frac{\pa V}{\pa x_{k'}} \frac{\pa f}{\pa x_{\hat i}}(a^{\ts}\nabla)_1f+a^{\ts}_{1\hat k}a^{\ts}_{1k'}\frac{\pa a^{\ts}_{2\hat i}}{\pa x_{\hat k}}\frac{\pa V}{\pa x_{k'}} \frac{\pa f}{\pa x_{\hat i}}(a^{\ts}\nabla)_2f\right]\\
&&-\sum_{\hat i,\hat k,k'=1}^3\left[a^{\ts}_{2\hat k}a^{\ts}_{2k'}\frac{\pa a^{\ts}_{1\hat i}}{\pa x_{\hat k}}\frac{\pa V}{\pa x_{k'}} \frac{\pa f}{\pa x_{\hat i}}(a^{\ts}\nabla)_1f+a^{\ts}_{2\hat k}a^{\ts}_{2k'}\frac{\pa a^{\ts}_{2\hat i}}{\pa x_{\hat k}}\frac{\pa V}{\pa x_{k'}} \frac{\pa f}{\pa x_{\hat i}}(a^{\ts}\nabla)_2f\right]\\
&=&-\beta e^{\beta\theta}(a^{\ts}\nabla)_1V \frac{\pa f}{\pa x}(a^{\ts}\nabla)_2f.
\eeaa
 Combing the above computations, we get the tensor $\mathfrak{R}_{ab}$.  Now we turn to the second tensor $\mathfrak{R}_{zb}$, which has the following form,
 \beaa
\mathfrak{R}_{zb}(\nabla f,\nabla f)&=&\sum_{i=1}^2\sum_{i',\hat i,\hat k=1}^{3} \la a^{\ts}_{ii'} (\frac{\partial a^{\ts}_{i \hat i}}{\partial x_{i'}} \frac{\partial z^{\ts}_{1\hat k}}{\partial x_{\hat i}}\frac{\partial f}{\partial x_{\hat k}}) ,(z^{\ts}\nabla)_1f\ra_{\hR}\nonumber \\
	&&+\sum_{i=1}^2\sum_{i',\hat i,\hat k=1}^{3} \la a^{\ts}_{ii'}a^{\ts}_{i \hat i} (\frac{\partial }{\partial x_{i'}} \frac{\partial z^{\ts}_{1\hat k}}{\partial x_{\hat i}})(\frac{\partial f}{\partial x_{\hat k}}) ,(z^{\ts}\nabla)_1f\ra_{\hR} \nonumber \\
	&&-\sum_{i=1}^2\sum_{i',\hat i,\hat k=1}^{3} \la z^{\ts}_{1\hat k}\frac{\partial a^{\ts}_{ii'}}{\partial x_{\hat k}} \frac{\partial a^{\ts}_{i \hat i}}{\partial x_{i'}} \frac{\partial f}{\partial x_{\hat i}})  ,(z^{\ts}\nabla)_1f\ra_{\hR} \nonumber \\
&&-\sum_{i=1}^2\sum_{i',\hat i,\hat k=1}^{3} \la z^{\ts}_{1\hat k} a^{\ts}_{ii'} (\frac{\partial }{\partial x_{\hat k}} \frac{\partial a^{\ts}_{i \hat i}}{\partial x_{i'}}) \frac{\partial f}{\partial x_{\hat i}}  ,(z^{\ts}\nabla)_1f\ra_{\hR}\nonumber \\
&&-\sum_{i=1}^2 \sum_{\hat i,\hat k=1}^{3}\la (z^{\ts}_{1\hat i}\frac{\pa b_{\hat k}}{\pa x_{\hat i}}\frac{\pa f}{\pa x_{\hat k}}-b_{\hat k}\frac{\pa z^{\ts}_{1\hat i}}{\pa x_{\hat k}} \frac{\pa f}{\pa x_{\hat i}} ),(z^{\ts}\nabla f)_1\ra_{\hR},\nonumber\\
&=&\cI_1^z+\cI_2^z+\cI_3^z+\cI_4^z+\mathfrak R^{z}_b(\nabla f,\nabla f).
\eeaa 
where we denote further that 
\beaa
\mathfrak{R}_{b}^z(\nabla f,\nabla f)=-\sum_{\hat i,\hat k=1}^{3} (z^{\ts}_{1\hat i}\frac{\pa b_{\hat k}}{\pa x_{\hat i}}\frac{\pa f}{\pa x_{\hat k}}-b_{\hat k}\frac{\pa z^{\ts}_{i\hat i}}{\pa x_{\hat k}} \frac{\pa f}{\pa x_{\hat i}} )(z^{\ts}\nabla f)_1.
\eeaa
By taking $b=-aa^{\ts}\nabla V$, we further obtain that
\beaa
\mathfrak{R}_{b}^z(\nabla f,\nabla f)
&=& -\sum_{\hat i,\hat k=1}^{3}\left[ z^{\ts}_{1\hat i}\frac{\pa b_{\hat k}}{\pa x_{\hat i}}\frac{\pa f}{\pa x_{\hat k}}(z^{\ts}\nabla f)_1-b_{\hat k}\frac{\pa z^{\ts}_{i\hat i}}{\pa x_{\hat k}} \frac{\pa f}{\pa x_{\hat i}} (z^{\ts}\nabla f)_1\right]\\
&=&\sum_{k=1}^2\sum_{\hat i,\hat k,k'=1}^3\left[z^{\ts}_{1\hat i}\frac{\pa a^{\ts}_{k\hat k} }{\pa x_{\hat i}}a^{\ts}_{kk'}\frac{\pa V}{\pa x_{k'}} \frac{\pa f}{\pa x_{\hat k}}(z^{\ts}\nabla)_1 f \right]\\
&&+\sum_{k=1}^2\sum_{\hat i,\hat k,k'=1}^3\left[z^{\ts}_{1\hat i}\frac{\pa a^{\ts}_{kk'}}{\pa x_{\hat i}}a^{\ts}_{k\hat k}\frac{\pa V}{\pa x_{k'}} \frac{\pa f}{\pa x_{\hat k}}(z^{\ts}\nabla)_1 f\right]\\ 
&&+\sum_{k=1}^2\sum_{\hat i,\hat k,k'=1}^3\left[ z^{\ts}_{1\hat i}a^{\ts}_{k\hat k}a^{\ts}_{kk'}\frac{\pa^2 V}{\pa x_{\hat i}\pa x_{k'}}\frac{\pa f}{\pa x_{\hat k}}(z^{\ts}\nabla)_1 f\right]
\eeaa 
\beaa 
&&-\sum_{k=1}^2\sum_{\hat i,\hat k,k'=1}^3\left[a^{\ts}_{k\hat k}a^{\ts}_{kk'}\frac{\pa z^{\ts}_{1\hat i}}{\pa x_{\hat k}}\frac{\pa V}{\pa x_{k'}} \frac{\pa f}{\pa x_{\hat i}}(z^{\ts}\nabla)_1f\right] \\
&=& \cJ_1^z+\cJ_2^z+\cJ_3^z+\cJ_4^z.
\eeaa
By direct computations, it is not hard to observe that
\beaa 
\cI_1^z&=& \sum_{i=1}^2\sum_{i',\hat i,\hat k=1}^{3} \la a^{\ts}_{ii'} (\frac{\partial a^{\ts}_{i \hat i}}{\partial x_{i'}} \frac{\partial z^{\ts}_{1\hat k}}{\partial x_{\hat i}}\frac{\partial f}{\partial x_{\hat k}}) ,(z^{\ts}\nabla)_1f\ra_{\hR}=0 \\
	\cI_2^z&=&\sum_{i=1}^2\sum_{i',\hat i,\hat k=1}^{3} \la a^{\ts}_{ii'}a^{\ts}_{i \hat i} (\frac{\partial }{\partial x_{i'}} \frac{\partial z^{\ts}_{1\hat k}}{\partial x_{\hat i}})(\frac{\partial f}{\partial x_{\hat k}}),(z^{\ts}\nabla)_1f\ra_{\hR}\\
	&&= \sum_{i=1}^2\sum_{i',\hat i=1}^{3}  a^{\ts}_{ii'}a^{\ts}_{i \hat i} \frac{\pa^2 z^{\ts}_{1\hat k}}{\partial x_{i'}\partial x_{\hat i}}\pa_yf (z^{\ts}\nabla)_1f \\
	\cI_3^z&=&-\sum_{i=1}^2\sum_{i',\hat i,\hat k=1}^{3} \la z^{\ts}_{1\hat k}\frac{\partial a^{\ts}_{ii'}}{\partial x_{\hat k}} \frac{\partial a^{\ts}_{i \hat i}}{\partial x_{i'}} \frac{\partial f}{\partial x_{\hat i}})  ,(z^{\ts}\nabla)_1f\ra_{\hR}=0 \nonumber \\
\cI_4^z&=&-\sum_{i=1}^2\sum_{i',\hat i,\hat k=1}^{3} \la z^{\ts}_{1\hat k} a^{\ts}_{ii'} (\frac{\partial }{\partial x_{\hat k}} \frac{\partial a^{\ts}_{i \hat i}}{\partial x_{i'}}) \frac{\partial f}{\partial x_{\hat i}}  ,(z^{\ts}\nabla)_1f\ra_{\hR}=0,\nonumber 
\eeaa 
and 
\beaa 
\cJ_1^z&=&\sum_{k=1}^2\sum_{\hat i,\hat k,k'=1}^3\left[z^{\ts}_{1\hat i}\frac{\pa a^{\ts}_{k\hat k} }{\pa x_{\hat i}}a^{\ts}_{kk'}\frac{\pa V}{\pa x_{k'}} \frac{\pa f}{\pa x_{\hat k}}(z^{\ts}\nabla)_1 f \right]=0;\\
\cJ_2^z&=&\sum_{k=1}^2\sum_{\hat i,\hat k,k'=1}^3\left[z^{\ts}_{1\hat i}\frac{\pa a^{\ts}_{kk'}}{\pa x_{\hat i}}a^{\ts}_{k\hat k}\frac{\pa V}{\pa x_{k'}} \frac{\pa f}{\pa x_{\hat k}}(z^{\ts}\nabla)_1 f\right]=0;\\
\cJ_4^z&=&-\sum_{k=1}^2\sum_{\hat i,\hat k,k'=1}^3\left[a^{\ts}_{k\hat k}a^{\ts}_{kk'}\frac{\pa z^{\ts}_{1\hat i}}{\pa x_{\hat k}}\frac{\pa V}{\pa x_{k'}} \frac{\pa f}{\pa x_{\hat i}}(z^{\ts}\nabla)_1f\right]\\
&=& -\sum_{k=1}^2(a^{\ts}\nabla)_kz^{\ts}_{13}(a^{\ts}\nabla)_kV\pa_yf(z^{\ts}\nabla)_1f.
\eeaa 
\beaa 
\cJ_3^z&=& \sum_{k=1}^2\sum_{\hat i,\hat k,k'=1}^3\left[ z^{\ts}_{1 \hat i}a^{\ts}_{k\hat k}a^{\ts}_{kk'}\frac{\pa^2 V}{\pa x_{\hat i}\pa x_{k'}}\frac{\pa f}{\pa x_{\hat k}}(z^{\ts}\nabla)_1 f\right]\\
&=&\sum_{\hat i,\hat k,k'=1}^3\left[ z^{\ts}_{1\hat i}a^{\ts}_{1\hat k}a^{\ts}_{1k'}\frac{\pa^2 V}{\pa x_{\hat i}\pa x_{k'}}\frac{\pa f}{\pa x_{\hat k}}(z^{\ts}\nabla)_1 f+ z^{\ts}_{1\hat i}a^{\ts}_{2\hat k}a^{\ts}_{2k'}\frac{\pa^2 V}{\pa x_{\hat i}\pa x_{k'}}\frac{\pa f}{\pa x_{\hat k}}(z^{\ts}\nabla)_1 f\right]\\
&=&\sum_{\hat i,k'=1}^3\left[ z^{\ts}_{1\hat i}a^{\ts}_{1k'}\frac{\pa^2 V}{\pa x_{\hat i}\pa x_{k'}} (a^{\ts}\nabla)_1 f(z^{\ts}\nabla)_1 f+ z^{\ts}_{1\hat i}a^{\ts}_{2k'}\frac{\pa^2 V}{\pa x_{\hat i}\pa x_{k'}}(a^{\ts}\nabla)_2 f(z^{\ts}\nabla)_1 f\right]\\
&=&-g\frac{\pa^2 V}{\pa \theta\pa y}(a^{\ts}\nabla)_1 f(z^{\ts}\nabla)_1 f-g(e^{\beta\theta}\frac{\pa^2 V}{\pa x\pa y}+\frac{\pa^2 V}{\pa y\pa y} )(a^{\ts}\nabla)_2 f(z^{\ts}\nabla)_1 f;
\eeaa 
Now we are left to compute the term $\mathfrak{R}_{\rho^*}$. Recall that, 
\beaa
\mathfrak{R}_{\rho^*}(\nabla f,\nabla f)&=&2\sum_{k=1}^1 \sum_{i=1}^2\sum_{k',\hat k,\hat i,i'=1}^{3}\left[\frac{\pa }{\pa x_{k'}} z^{\ts}_{kk'} z^{\ts}_{k\hat k}\frac{\pa}{\pa x_{\hat k}}a^{\ts}_{i\hat i}\frac{\pa f}{\pa x_{\hat i}}a^{\ts}_{ii'}\frac{\pa f}{\pa x_{i'}}\right]\\
 &&+2\sum_{k=1}^1 \sum_{i=1}^2\sum_{k',\hat k,\hat i,i'=1}^{3}\left[z^{\ts}_{kk'}\frac{\pa }{\pa x_{k'}} z^{\ts}_{k\hat k} \frac{\pa}{\pa x_{\hat k}}a^{\ts}_{i\hat i}\frac{\pa f}{\pa x_{\hat i}}a^{\ts}_{ii'}\frac{\pa f}{\pa x_{i'}} \right.\nonumber\\
 &&\quad\quad\quad\quad\quad\quad\quad\quad+z^{\ts}_{kk'} z^{\ts}_{k\hat k} \frac{\pa^2}{\pa x_{k'}\pa x_{\hat k}}a^{\ts}_{i\hat i}\frac{\pa f}{\pa x_{\hat i}}a^{\ts}_{ii'}\frac{\pa f}{\pa x_{i'}}\nonumber \\
&&\left.\quad\quad\quad\quad\quad\quad\quad\quad+z^{\ts}_{kk'} z^{\ts}_{k\hat k} \frac{\pa}{\pa x_{\hat k}}a^{\ts}_{i\hat i}\frac{\pa f}{\pa x_{\hat i}}\frac{\pa }{\pa x_{k'}}a^{\ts}_{ii'}\frac{\pa f}{\pa x_{i'}} \right].\nonumber \\
&&+2\sum_{k=1}^1 \sum_{i=1}^2\sum_{\hat k,\hat i,i'=1}^{3}(z^{\ts}\nabla\log\rho^*)_k \left[ z^{\ts}_{k\hat k}\frac{\pa}{\pa x_{\hat k}}a^{\ts}_{i\hat i}\frac{\pa f}{\pa x_{\hat i}}a^{\ts}_{ii'}\frac{\pa f}{\pa x_{i'}} \right] \nonumber\\
&&-2\sum_{j=1}^1\sum_{l=1}^2\sum_{ l',\hat l,\hat j,j'=1}^{3}\left[\frac{\pa }{\pa x_{l'}} a^{\ts}_{ll'} a^{\ts}_{l\hat l} \frac{\pa}{\pa x_{\hat l}}z^{\ts}_{j\hat j}\frac{\pa f}{\pa x_{\hat j}}z^{\ts}_{jj'}\frac{\pa f}{\pa x_{j'}} \right] \nonumber\\
 &&- 2\sum_{j=1}^1\sum_{l=1}^2\sum_{ l',\hat l,\hat j,j'=1}^{3}\left[ a^{\ts}_{ll'}\frac{\pa }{\pa x_{l'}}a^{\ts}_{l\hat l} \frac{\pa}{\pa x_{\hat l}}z^{\ts}_{j\hat j}\frac{\pa f}{\pa x_{\hat j}}z^{\ts}_{jj'}\frac{\pa f}{\pa x_{j'}}  \right.\nonumber\\
 &&\quad\quad\quad\quad\quad\quad\quad\quad+a^{\ts}_{ll'}a^{\ts}_{l\hat l} \frac{\pa^2}{\pa x_{l'}\pa x_{\hat l}}z^{\ts}_{j\hat j}\frac{\pa f}{\pa x_{\hat j}}z^{\ts}_{jj'}\frac{\pa f}{\pa x_{j'}}\nonumber \\
&&\left.\quad\quad\quad\quad\quad\quad\quad\quad+a^{\ts}_{ll'}a^{\ts}_{l\hat l} \frac{\pa}{\pa x_{\hat l}}z^{\ts}_{j\hat j}\frac{\pa f}{\pa x_{\hat j}}\frac{\partial}{\pa x_{l'}}z^{\ts}_{jj'}\frac{\pa f}{\pa x_{j'}}  \right] \nonumber \\
&&-2\sum_{j=1}^1\sum_{l=1}^2\sum_{\hat l,\hat j,j'=1}^{3}(a^{\ts}\nabla\log\rho^*)_l \left[ a^{\ts}_{l\hat l}\frac{\pa}{\pa x_{\hat l}}z^{\ts}_{j\hat j}\frac{\pa f}{\pa x_{\hat j}}z^{\ts}_{jj'}\frac{\pa f}{\pa x_{j'}} \right]\nonumber\\
&=&\sum_{i=1}^{10}\cK_i.
\eeaa
By direct computation, we get 
\beaa
\cK_1&=&0, \quad \cK_2=0,\quad \cK_3=0,\quad \cK_4=0,\quad \cK_5=0,\quad \cK_6=0,\quad \cK_7=0;\\
\cK_8&=& -2\sum_{l=1}^2\sum_{l',\hat l=1}^3a^{\ts}_{ll'}a^{\ts}_{l\hat l}\frac{\pa^2 z^{\ts}_{13}}{\pa x_{l'}\pa x_{\hat l}}\pa_yf (z^{\ts}\nabla)_1f;\\
\cK_9&=&-2\sum_{l=1}^2\sum_{l',\hat l=1}^3a^{\ts}_{ll'}a^{\ts}_{l\hat l}\frac{\pa z^{\ts}_{13}}{\pa x_{\hat l}}\frac{\pa z^{\ts}_{13}}{\pa x_{l'}} |\pa_yf|^2=-2\Gamma_1(\log g,\log g)|(z^{\ts}\nabla)_1f|^2;\\
\cK_{10}&=&-2 \sum_{l=1}^2\sum_{\hat l=1}^3(a^{\ts}\nabla)_l\log \rho^* a^{\ts}_{l\hat l}\frac{\pa z^{\ts}_{13}}{\pa x_{\hat l}}\pa_yf (z^{\ts}\nabla)_1f=-2\Gamma_1(\log \rho^*,\log g)|(z^{\ts}\nabla)_1f|^2.
\eeaa 
\end{proof}{}

\subsection{Proof Of Proposition \ref{prop Martinet}}
By direct computations, we have the following lemma.
\begin{lem}\label{martinet} For Martinet sub-Riemannian structure $(\mathbb M,\mathcal H, (aa^{\ts})^{\dd}_{|\mathcal{H}})$, we have
\beaa
Q&=&\left(
\begin{array}{ccccccccc}
 1 & 0 & \frac{y^2}{2} & 0 & 0 & 0 & \frac{y^2}{2} & 0 & \frac{y^4}{4} \\
 0 & 1 & 0 & 0 & 0 & 0 & 0 & \frac{y^2}{2} & 0 \\
 0 & 0 & 0 & 1 & 0 & \frac{y^2}{2} & 0 & 0 & 0 \\
 0 & 0 & 0 & 0 & 1 & 0 & 0 & 0 & 0 \\
\end{array}
\right);\\
P&=&\begin{pmatrix}
	0&0&0&0&0&0&1&0&y^2/2\\
	0&0&0&0&0&0&0&1&0
\end{pmatrix};\\
C^{\ts}&=&(0,0,0,0,0,\frac{y^3}{2}\pa_zf+y\pa_xf,-y\pa_yf,0,-\frac{y^3}{2}\pa_yf);\\
D^{\ts}&=&(0,0,y\pa_zf,0),\quad E^{\ts}=(0,0);\\
F^{\ts}&=&G^{\ts}=(0,0,0,0,0,0,0,0,0).
\eeaa
\end{lem}{}

The proof follows from the following two lemmas. 
\begin{lem}\label{martinet hess} For Martinet sub-Riemannian structure, $F$ and $G$ are zero vectors, we have
\beaa
[QX+D]^{\ts}[QX+D]+[PX+E]^{\ts}[PX+E]+2C^{\ts}X=|\mathfrak{Hess}_{a,z}^{G}f|^2+\mathfrak{R}^G(\nabla f,\nabla f).
\eeaa
In particular, we have
\beaa
|\mathfrak{Hess}_{a,z}^{G}f|^2&=&[X+\Lambda_1]^{\ts}Q^{\ts}Q[X+\Lambda_1]+[X+\Lambda_2]^{\ts}P^{\ts}P[X+\Lambda_2];\\
\Lambda_1^{\ts}&=&(0,y\pa_zf/2,0,y\pa_zf/2,0,0,0,0,0); \\
\Lambda_2^{\ts}&=&(0,0,0,0,0,0,-y\pa_yf,\frac{y^3}{2}\pa_zf+y\pa_xf,0);\\
\mathfrak{R}^G(\nabla f,\nabla f)&=&\frac{y^2}{2}\Gamma_1^z(f,f)-y^2\Gamma_1(f,f).
\eeaa
\end{lem}{}
\begin{lem}\label{tensor martinet}
By routine computations, we obtain
\beaa
\mathfrak{R}_{ab}(\nabla f,\nabla f)&=&\frac{\pa f}{\pa z}(a^{\ts}\nabla )_1f+y(a^{\ts}\nabla)_1V \frac{\pa f}{\pa z}(a^{\ts}\nabla)_2f +y \frac{\pa V}{\pa z}(a^{\ts}\nabla)_1f(a^{\ts}\nabla)_2f \\
&&+\sum_{\hat i,k'=1}^3a^{\ts}_{1\hat i}a^{\ts}_{1k'}\frac{\pa^2 V}{\pa x_{\hat i}\pa x_{k'}}|(a^{\ts}\nabla)_1 f|^2+2(\frac{\pa^2 V}{\pa x\pa y}+\frac{y^2}{2}\frac{\pa^2 V}{\pa y \pa z} )(a^{\ts}\nabla)_1 f(a^{\ts}\nabla)_2 f\\
&&+\frac{\pa^2 V}{\pa y\pa y}|(a^{\ts}\nabla)_2 f|^2-y\frac{\pa V}{\pa y}\frac{\pa f}{\pa z}(a^{\ts}\nabla)_1f;\\
\mathfrak{R}_{zb}(\nabla f,\nabla f)&=& (\frac{\pa^2 V}{\pa x\pa z}+\frac{y^2}{2}\frac{\pa^2 V}{\pa z\pa z})(a^{\ts}\nabla)_1 f (z^{\ts}\nabla)_1 f+\frac{\pa^2 V}{\pa y\pa z}(a^{\ts}\nabla)_2 f(z^{\ts}\nabla)_1 f;\\
\mathfrak{R}_{\rho^*}(\nabla f,\nabla f)&=&0.
\eeaa
\end{lem}{}
\begin{proof}[Proof of Lemma \ref{martinet hess}] Since $F$ and $G$ are zero vectors, we have
\beaa
2C^{\ts}X=2\Big[\frac{\pa^2 f}{\pa y\pa z}(\frac{y^3}{2}\pa_zf+y\pa_xf)-\frac{\pa^2 f}{\pa x \pa z}(y\pa_yf)-\frac{\pa^2 f}{\pa z\pa z}(\frac{y^3}{2}\pa_yf)\Big].
\eeaa
By routine computation, we observe that 
\beaa
&&[QX+D]^{\ts}[QX+D]+[PX+E]^{\ts}[PX+E]+2C^{\ts}X\\
&=&\left[\frac{\pa^2 f}{\pa x\pa x}+\frac{y^2}{2}\frac{\pa^2 f}{\pa x\pa z} +\frac{y^2}{2}\frac{\pa^2 f}{\pa z \pa x}+\frac{y^4}{4}\frac{\pa^2 f}{\pa z \pa z}\right]^2 +\left[\frac{\pa^2 f}{\pa y \pa x}+\frac{y^2}{2}\frac{\pa^2 f}{\pa z\pa y}+y\pa_zf\right]^2\\
&&+\left[\frac{\pa^2 f}{\pa y \pa x}+\frac{y^2}{2}\frac{\pa^2 f}{\pa z\pa y}\right]^2+\left[\frac{\pa^2 f}{\pa y\pa y} \right]^2\\
&&+\left[\frac{\pa^2 f}{\pa z\pa x}+\frac{y^2}{2}\frac{\pa^2 f}{\pa z\pa z} \right]^2+\left[ \frac{\pa^2 f}{\pa z\pa y}\right]^2\\
&&+2\frac{\pa^2 f}{\pa y\pa z}(\frac{y^3}{2}\pa_zf+y\pa_xf)-2\frac{\pa^2 f}{\pa x \pa z}(y\pa_yf)-2\frac{\pa^2 f}{\pa z\pa z}(\frac{y^3}{2}\pa_yf)\\
&=&\left[\frac{\pa^2 f}{\pa x\pa x}+\frac{y^2}{2}\frac{\pa^2 f}{\pa x\pa z} +\frac{y^2}{2}\frac{\pa^2 f}{\pa z \pa x}+\frac{y^4}{4}\frac{\pa^2 f}{\pa z \pa z}\right]^2 +\left[\frac{\pa^2 f}{\pa y \pa x}+\frac{y^2}{2}\frac{\pa^2 f}{\pa z\pa y}+y\pa_zf\right]^2\\
&&+\left[\frac{\pa^2 f}{\pa y \pa x}+\frac{y^2}{2}\frac{\pa^2 f}{\pa z\pa y}\right]^2+\left[\frac{\pa^2 f}{\pa y\pa y} \right]^2\\
&&+\left[\frac{\pa^2 f}{\pa z\pa x}+\frac{y^2}{2}\frac{\pa^2 f}{\pa z\pa z}-y\pa_y f \right]^2+\left[ \frac{\pa^2 f}{\pa z\pa y}+(\frac{y^3}{2}\pa_zf+y\pa_xf)\right]^2\\
&&-y^2|\pa_y f|^2-(\frac{y^3}{2}\pa_zf+y\pa_xf)^2\\
&=&|\mathfrak{Hess}_{a,z}^{G}f|^2+\frac{y^2}{2}\Gamma_1^z(f,f)-y^2\Gamma_1(f,f).
\eeaa 
The proof is thus completed.
\qed 
\end{proof}
We are now left to compute the three tensor terms.

\begin{proof}[Proof of Lemma \ref{tensor martinet}]
Similar to the proof of Lemma \ref{SE2 tensor}, we have 
\beaa
\mathfrak{R}_a(\nabla f,\nabla f)
&=&\sum_{i,k=1}^2\sum_{i',\hat i,\hat k=1}^{3} \la a^{\ts}_{ii'} (\frac{\partial a^{\ts}_{i \hat i}}{\partial x_{i'}} \frac{\partial a^{\ts}_{k\hat k}}{\partial x_{\hat i}}\frac{\partial f}{\partial x_{\hat k}}) ,(a^{\ts}\nabla)_kf\ra_{\hR^2}  \nonumber\\
&&+\sum_{i,k=2}^n\sum_{i',\hat i,\hat k=1}^{3} \la a^{\ts}_{ii'}a^{\ts}_{i \hat i} (\frac{\partial }{\partial x_{i'}} \frac{\partial a^{\ts}_{k\hat k}}{\partial x_{\hat i}})(\frac{\partial f}{\partial x_{\hat k}}) ,(a^{\ts}\nabla)_kf\ra_{\hR^2} \nonumber\\
&&-\sum_{i,k=1}^2\sum_{i',\hat i,\hat k=1}^{3} \la a^{\ts}_{k\hat k}\frac{\partial a^{\ts}_{ii'}}{\partial x_{\hat k}} \frac{\partial a^{\ts}_{i \hat i}}{\partial x_{i'}} \frac{\partial f}{\partial x_{\hat i}})  ,(a^{\ts}\nabla)_kf\ra_{\hR^2}  \nonumber\\
&&-\sum_{i,k=1}^2\sum_{i',\hat i,\hat k=1}^{3} \la a^{\ts}_{k\hat k} a^{\ts}_{ii'} (\frac{\partial }{\partial x_{\hat k}} \frac{\partial a^{\ts}_{i \hat i}}{\partial x_{i'}}) \frac{\partial f}{\partial x_{\hat i}}  ,(a^{\ts}\nabla)_kf\ra_{\hR^2},  \nonumber\\
&=& \cI_1+\cI_2+\cI_3+\cI_4.
\eeaa
By direct computations, we have
\beaa
\cI_1&=&\sum_{i=1}^2\sum_{i',\hat i,\hat k=1}^{3} \Big[ a^{\ts}_{ii'} (\frac{\partial a^{\ts}_{i \hat i}}{\partial x_{i'}} \frac{\partial a^{\ts}_{1\hat k}}{\partial x_{\hat i}}\frac{\partial f}{\partial x_{\hat k}})(a^{\ts}\nabla)_1f+  a^{\ts}_{ii'} (\frac{\partial a^{\ts}_{i \hat i}}{\partial x_{i'}} \frac{\partial a^{\ts}_{2\hat k}}{\partial x_{\hat i}}\frac{\partial f}{\partial x_{\hat k}}) (a^{\ts}\nabla)_2f\Big]=0;\\
\cI_2&=&\sum_{i=2}^n\sum_{i',\hat i,\hat k=1}^{3}\Big[  a^{\ts}_{ii'}a^{\ts}_{i \hat i} (\frac{\partial }{\partial x_{i'}} \frac{\partial a^{\ts}_{1\hat k}}{\partial x_{\hat i}})(\frac{\partial f}{\partial x_{\hat k}}) (a^{\ts}\nabla)_1f+ a^{\ts}_{ii'}a^{\ts}_{i \hat i} (\frac{\partial }{\partial x_{i'}} \frac{\partial a^{\ts}_{2\hat k}}{\partial x_{\hat i}})(\frac{\partial f}{\partial x_{\hat k}})(a^{\ts}\nabla)_2f\Big]\\
&=&a^{\ts}_{22}a^{\ts}_{22}\frac{\pa^2}{\pa y \pa y}a^{\ts}_{13}\frac{\pa f}{\pa z}(a^{\ts}\nabla )_1f=\frac{\pa f}{\pa z}(a^{\ts}\nabla )_1f ;  \\
\cI_3&=&-\sum_{i=1}^2\sum_{i',\hat i,\hat k=1}^{3} \Big[  a^{\ts}_{1\hat k}\frac{\partial a^{\ts}_{ii'}}{\partial x_{\hat k}} \frac{\partial a^{\ts}_{i \hat i}}{\partial x_{i'}} \frac{\partial f}{\partial x_{\hat i}})(a^{\ts}\nabla)_1f+  a^{\ts}_{2\hat k}\frac{\partial a^{\ts}_{ii'}}{\partial x_{\hat k}} \frac{\partial a^{\ts}_{i \hat i}}{\partial x_{i'}} \frac{\partial f}{\partial x_{\hat i}})(a^{\ts}\nabla)_2f\Big]=0;\\
\cI_4&=&-\sum_{i=1}^2\sum_{i',\hat i,\hat k=1}^{3} \Big[  a^{\ts}_{1\hat k} a^{\ts}_{ii'} (\frac{\partial }{\partial x_{\hat k}} \frac{\partial a^{\ts}_{i \hat i}}{\partial x_{i'}}) \frac{\partial f}{\partial x_{\hat i}}(a^{\ts}\nabla)_1f+  a^{\ts}_{2\hat k} a^{\ts}_{ii'} (\frac{\partial }{\partial x_{\hat k}} \frac{\partial a^{\ts}_{i \hat i}}{\partial x_{i'}}) \frac{\partial f}{\partial x_{\hat i}}(a^{\ts}\nabla)_2f\Big]=0.
\eeaa
For the drift term, we take $b=-\frac{1}{2}aa^{\ts}\nabla V$, we have
\beaa
\mathfrak{R}_b
&=&\sum_{i,k=1}^2\sum_{\hat i,\hat k,k'=1}^3\left[a^{\ts}_{i\hat i}\frac{\pa a^{\ts}_{k\hat k} }{\pa x_{\hat i}}a^{\ts}_{kk'}\frac{\pa V}{\pa x_{k'}} \frac{\pa f}{\pa x_{\hat k}}(a^{\ts}\nabla)_i f +a^{\ts}_{i\hat i}\frac{\pa a^{\ts}_{kk'}}{\pa x_{\hat i}}a^{\ts}_{k\hat k}\frac{\pa V}{\pa x_{k'}} \frac{\pa f}{\pa x_{\hat k}}(a^{\ts}\nabla)_i f\right]\\
&&+\sum_{i,k=1}^2\sum_{\hat i,\hat k,k'=1}^3\left[ a^{\ts}_{i\hat i}a^{\ts}_{k\hat k}a^{\ts}_{kk'}\frac{\pa^2 V}{\pa x_{\hat i}\pa x_{k'}}\frac{\pa f}{\pa x_{\hat k}}(a^{\ts}\nabla)_i f\right]\\
&&-\sum_{i,k=1}^2\sum_{\hat i,\hat k,k'=1}^3\left[a^{\ts}_{k\hat k}a^{\ts}_{kk'}\frac{\pa a^{\ts}_{i\hat i}}{\pa x_{\hat k}}\frac{\pa V}{\pa x_{k'}} \frac{\pa f}{\pa x_{\hat i}}(a^{\ts}\nabla)_if\right] \\
&=& \cJ_1+\cJ_2+\cJ_3+\cJ_4.
\eeaa
Plugging into the matrices of $a^{\ts}$ and $z^{\ts}$, we get
\beaa
\cJ_1
&=&\sum_{\hat i,\hat k,k'=1}^3\left[ a^{\ts}_{1\hat i}\frac{\pa a^{\ts}_{1\hat k} }{\pa x_{\hat i}}a^{\ts}_{1k'}\frac{\pa V}{\pa x_{k'}} \frac{\pa f}{\pa x_{\hat k}}(a^{\ts}\nabla)_1 f+a^{\ts}_{2\hat i}\frac{\pa a^{\ts}_{1\hat k} }{\pa x_{\hat i}}a^{\ts}_{1k'}\frac{\pa V}{\pa x_{k'}} \frac{\pa f}{\pa x_{\hat k}}(a^{\ts}\nabla)_2 f\right]\\
&&+\sum_{\hat i,\hat k,k'=1}^3\left[ a^{\ts}_{1\hat i}\frac{\pa a^{\ts}_{2\hat k} }{\pa x_{\hat i}}a^{\ts}_{2k'}\frac{\pa V}{\pa x_{k'}} \frac{\pa f}{\pa x_{\hat k}}(a^{\ts}\nabla)_1 f+a^{\ts}_{2\hat i}\frac{\pa a^{\ts}_{2\hat k} }{\pa x_{\hat i}}a^{\ts}_{2k'}\frac{\pa V}{\pa x_{k'}} \frac{\pa f}{\pa x_{\hat k}}(a^{\ts}\nabla)_2 f\right]\\
&=& a^{\ts}_{22}\frac{\pa a^{\ts}_{13}}{\pa y}(a^{\ts}\nabla)_1V \frac{\pa f}{\pa z}(a^{\ts}\nabla)_2f=y(a^{\ts}\nabla)_1V \frac{\pa f}{\pa z}(a^{\ts}\nabla)_2f;
\eeaa
\beaa
\cJ_2
&=&\sum_{\hat i,\hat k,k'=1}^3\left[a^{\ts}_{1\hat i}\frac{\pa a^{\ts}_{1k'}}{\pa x_{\hat i}}a^{\ts}_{1\hat k}\frac{\pa V}{\pa x_{k'}} \frac{\pa f}{\pa x_{\hat k}}(a^{\ts}\nabla)_1 f+a^{\ts}_{2\hat i}\frac{\pa a^{\ts}_{1k'}}{\pa x_{\hat i}}a^{\ts}_{1\hat k}\frac{\pa V}{\pa x_{k'}} \frac{\pa f}{\pa x_{\hat k}}(a^{\ts}\nabla)_2 f\right]\\
&&+\sum_{\hat i,\hat k,k'=1}^3\left[a^{\ts}_{1\hat i}\frac{\pa a^{\ts}_{2k'}}{\pa x_{\hat i}}a^{\ts}_{2\hat k}\frac{\pa V}{\pa x_{k'}} \frac{\pa f}{\pa x_{\hat k}}(a^{\ts}\nabla)_1 f+a^{\ts}_{2\hat i}\frac{\pa a^{\ts}_{2k'}}{\pa x_{\hat i}}a^{\ts}_{2\hat k}\frac{\pa V}{\pa x_{k'}} \frac{\pa f}{\pa x_{\hat k}}(a^{\ts}\nabla)_2 f\right]\\
&=&y \frac{\pa V}{\pa z}(a^{\ts}\nabla)_1f(a^{\ts}\nabla)_2f;
\eeaa
\beaa
\cJ_3
&=&\sum_{\hat i,k'=1}^3\left[ a^{\ts}_{1\hat i}a^{\ts}_{1k'}\frac{\pa^2 V}{\pa x_{\hat i}\pa x_{k'}} |(a^{\ts}\nabla)_1 f|^2+a^{\ts}_{2\hat i}a^{\ts}_{1k'}\frac{\pa^2 V}{\pa x_{\hat i}\pa x_{k'}}(a^{\ts}\nabla)_1 f(a^{\ts}\nabla)_2 f\right]\\
&&+\sum_{\hat i,k'=1}^3\left[ a^{\ts}_{1\hat i}a^{\ts}_{2k'}\frac{\pa^2 V}{\pa x_{\hat i}\pa x_{k'}}(a^{\ts}\nabla)_2 f(a^{\ts}\nabla)_1 f+a^{\ts}_{2\hat i}a^{\ts}_{2k'}\frac{\pa^2 V}{\pa x_{\hat i}\pa x_{k'}}|(a^{\ts}\nabla)_2 f|^2\right]\\
&=&\sum_{\hat i,k'=1}^3a^{\ts}_{1\hat i}a^{\ts}_{1k'}\frac{\pa^2 V}{\pa x_{\hat i}\pa x_{k'}}|(a^{\ts}\nabla)_1 f|^2+2(\frac{\pa^2 V}{\pa x\pa y}+\frac{y^2}{2}\frac{\pa^2 V}{\pa y \pa z} )(a^{\ts}\nabla)_1 f(a^{\ts}\nabla)_2 f\\
&&+\frac{\pa^2 V}{\pa y\pa y}|(a^{\ts}\nabla)_2 f|^2;
\eeaa

\beaa
\cJ_4
&=&-\sum_{\hat i,\hat k,k'=1}^3\left[a^{\ts}_{1\hat k}a^{\ts}_{1k'}\frac{\pa a^{\ts}_{1\hat i}}{\pa x_{\hat k}}\frac{\pa V}{\pa x_{k'}} \frac{\pa f}{\pa x_{\hat i}}(a^{\ts}\nabla)_1f+a^{\ts}_{1\hat k}a^{\ts}_{1k'}\frac{\pa a^{\ts}_{2\hat i}}{\pa x_{\hat k}}\frac{\pa V}{\pa x_{k'}} \frac{\pa f}{\pa x_{\hat i}}(a^{\ts}\nabla)_2f\right]\\
&&-\sum_{\hat i,\hat k,k'=1}^3\left[a^{\ts}_{2\hat k}a^{\ts}_{2k'}\frac{\pa a^{\ts}_{1\hat i}}{\pa x_{\hat k}}\frac{\pa V}{\pa x_{k'}} \frac{\pa f}{\pa x_{\hat i}}(a^{\ts}\nabla)_1f+a^{\ts}_{2\hat k}a^{\ts}_{2k'}\frac{\pa a^{\ts}_{2\hat i}}{\pa x_{\hat k}}\frac{\pa V}{\pa x_{k'}} \frac{\pa f}{\pa x_{\hat i}}(a^{\ts}\nabla)_2f\right]\\
&=&-y\frac{\pa V}{\pa y}\frac{\pa f}{\pa z}(a^{\ts}\nabla)_1f.
\eeaa
 Combing the above computations, we get the tensor $\mathfrak{R}_{ab}$. Now we turn to the second tensor $\mathfrak{R}_{zb}$.
Since $z^{\ts}=(0,0,1)$, it is obvious to see that only the drift term of the tensor $\mathfrak{R}_{zb}$ remains, where we denote
\beaa
\mathfrak{R}_{zb}(\nabla f,\nabla f)=-2\sum_{\hat i,\hat k=1}^{n+m} (z^T_{1\hat i}\frac{\pa b_{\hat k}}{\pa x_{\hat i}}\frac{\pa f}{\pa x_{\hat k}}-b_{\hat k}\frac{\pa z^T_{1\hat i}}{\pa x_{\hat k}} \frac{\pa f}{\pa x_{\hat i}} )(z^T\nabla )_1f.
\eeaa
By taking $b=-\frac{1}{2}aa^{\ts}\nabla V$, we further obtain that
\beaa
\mathfrak{R}_{zb}(\nabla f,\nabla f)
&=&-2\sum_{\hat i,\hat k=1}^{3}\left[ z^T_{1\hat i}\frac{\pa b_{\hat k}}{\pa x_{\hat i}}\frac{\pa f}{\pa x_{\hat k}}(z^T\nabla f)_1-b_{\hat k}\frac{\pa z^T_{1\hat i}}{\pa x_{\hat k}} \frac{\pa f}{\pa x_{\hat i}} (z^T\nabla f)_1\right]\\
&=&\sum_{k=1}^2\sum_{\hat i,\hat k,k'=1}^3\left[z^{\ts}_{1\hat i}\frac{\pa a^{\ts}_{k\hat k} }{\pa x_{\hat i}}a^{\ts}_{kk'}\frac{\pa V}{\pa x_{k'}} \frac{\pa f}{\pa x_{\hat k}}(z^{\ts}\nabla)_1 f\right]\\
&&+\sum_{k=1}^2\sum_{\hat i,\hat k,k'=1}^3\left[z^{\ts}_{1\hat i}\frac{\pa a^{\ts}_{kk'}}{\pa x_{\hat i}}a^{\ts}_{k\hat k}\frac{\pa V}{\pa x_{k'}} \frac{\pa f}{\pa x_{\hat k}}(z^{\ts}\nabla)_1 f\right]\\ 
&&+\sum_{k=1}^2\sum_{\hat i,\hat k,k'=1}^3\left[ z^{\ts}_{1\hat i}a^{\ts}_{k\hat k}a^{\ts}_{kk'}\frac{\pa^2 V}{\pa x_{\hat i}\pa x_{k'}}\frac{\pa f}{\pa x_{\hat k}}(z^{\ts}\nabla)_1 f\right]\\
&&-\sum_{k=1}^2\sum_{\hat i,\hat k,k'=1}^3\left[a^{\ts}_{k\hat k}a^{\ts}_{kk'}\frac{\pa z^{\ts}_{1\hat i}}{\pa x_{\hat k}}\frac{\pa V}{\pa x_{k'}} \frac{\pa f}{\pa x_{\hat i}}(z^{\ts}\nabla)_1f\right] \\
&=& \cJ_1^z+\cJ_2^z+\cJ_3^z+\cJ_4^z.
\eeaa
By direct computations, it is not hard to observe that
\beaa
\cJ_1^z&=&\sum_{k=1}^2\sum_{\hat i,\hat k,k'=1}^3\left[z^{\ts}_{1\hat i}\frac{\pa a^{\ts}_{k\hat k} }{\pa x_{\hat i}}a^{\ts}_{kk'}\frac{\pa V}{\pa x_{k'}} \frac{\pa f}{\pa x_{\hat k}}(a^{\ts}\nabla)_1 f \right]=0;\\
\cJ_2^z&=&\sum_{k=1}^2\sum_{\hat i,\hat k,k'=1}^3\left[z^{\ts}_{1\hat i}\frac{\pa a^{\ts}_{kk'}}{\pa x_{\hat i}}a^{\ts}_{k\hat k}\frac{\pa V}{\pa x_{k'}} \frac{\pa f}{\pa x_{\hat k}}(a^{\ts}\nabla)_1 f\right]=0;\\
\cJ_4^z&=&-\sum_{k=1}^2\sum_{\hat i,\hat k,k'=1}^3\left[a^{\ts}_{k\hat k}a^{\ts}_{kk'}\frac{\pa z^{\ts}_{1\hat i}}{\pa x_{\hat k}}\frac{\pa V}{\pa x_{k'}} \frac{\pa f}{\pa x_{\hat i}}(a^{\ts}\nabla)_1f\right]=0.
\eeaa
The only non-zero term has the following form,
\beaa
\cJ_3^z&=&\sum_{k=1}^2\sum_{\hat i,\hat k,k'=1}^3\left[ z^{\ts}_{1 \hat i}a^{\ts}_{k\hat k}a^{\ts}_{kk'}\frac{\pa^2 V}{\pa x_{\hat i}\pa x_{k'}}\frac{\pa f}{\pa x_{\hat k}}(z^{\ts}\nabla)_1 f\right]\\
&=&\sum_{\hat i,\hat k,k'=1}^3\left[ z^{\ts}_{1\hat i}a^{\ts}_{1\hat k}a^{\ts}_{1k'}\frac{\pa^2 V}{\pa x_{\hat i}\pa x_{k'}}\frac{\pa f}{\pa x_{\hat k}}(z^{\ts}\nabla)_1 f+ z^{\ts}_{1\hat i}a^{\ts}_{2\hat k}a^{\ts}_{2k'}\frac{\pa^2 V}{\pa x_{\hat i}\pa x_{k'}}\frac{\pa f}{\pa x_{\hat k}}(z^{\ts}\nabla)_1 f\right]\\
&=&\sum_{\hat i,k'=1}^3\left[ z^{\ts}_{1\hat i}a^{\ts}_{1k'}\frac{\pa^2 V}{\pa x_{\hat i}\pa x_{k'}} (a^{\ts}\nabla)_1 f(z^{\ts}\nabla)_1 f+ z^{\ts}_{1\hat i}a^{\ts}_{2k'}\frac{\pa^2 V}{\pa x_{\hat i}\pa x_{k'}}(a^{\ts}\nabla)_2 f(z^{\ts}\nabla)_1 f\right]\\
&=&(\frac{\pa^2 V}{\pa x\pa z}+\frac{y^2}{2}\frac{\pa^2 V}{\pa z\pa z})(a^{\ts}\nabla)_1 f(z^{\ts}\nabla)_1 f+\frac{\pa^2 V}{\pa y\pa z}(a^{\ts}\nabla)_2 f(z^{\ts}\nabla)_1 f.
\eeaa
Since matrix $z^{\ts}$ is a constant matrix and matrix $a^{\ts}$ contains only variable $y$, it is easy to observe that 
\[
\mathfrak{R}_{\rho^*}(\nabla f,\nabla f)=0.
\]
\qed 
\end{proof}{}

\section*{Acknowledgement} The authors would like to thank Professor Fabrice Baudoin for many helpful discussions and valuable suggestions. The authors would also like to thank Luca Rizzi for helpful discussion on Popp's volume.


\begin{thebibliography}{10}

\bibitem{agrachev2009optimal}
A.~Agrachev and P.~Lee.
\newblock Optimal transportation under nonholonomic constraints.
\newblock {\em Transactions of the American Mathematical Society}, 361(11):6019--6047, 2009.

\bibitem{bakryemery1985}
D.~Bakry and M.~{\'E}mery.
\newblock Diffusions hypercontractives.
\newblock {\em S{\'e}minaire de Probabilit{\'e}s XIX 1983/84},
177--206. Springer, 1985.

\bibitem{barilari2013formula}
D.~Barilari and L.~Rizzi.
\newblock A formula for Popp’s volume in sub-Riemannian geometry.
\newblock {\em Analysis and Geometry in Metric Spaces},1:42--57, 2013.

\bibitem{barilari2019bakry}
D.~Barilari and L.~Rizzi.
\newblock Bakry-\'{E}mery curvature and model spaces in sub-Riemannian
  geometry.
\newblock {\em Mathematische Annalen}, 377(1-2), 435-–482, 
  2020.

\bibitem{baudoin2016wasserstein}
F.~Baudoin.
\newblock Wasserstein contraction properties for hypoelliptic diffusions.
\newblock {\em arXiv preprint arXiv:1602.04177}, 2016.

\bibitem{Baudoin2017}
F.~Baudoin.
\newblock Bakry--{\'E}mery meet Villani.
\newblock {\em Journal of Functional Analysis}, 273(7): 2275--2291, 2017.

\bibitem{baudoin2012log}
F.~Baudoin and M.~Bonnefont.
\newblock Log-sobolev inequalities for subelliptic operators satisfying a
  generalized curvature dimension inequality.
\newblock {\em Journal of Functional Analysis}, 262(6): 2646--2676, 2012.

\bibitem{BBG}
F.~Baudoin, M.~Bonnefont, and N.~Garofalo.
\newblock A sub-{{Riemannian}} curvature-dimension inequality, volume doubling
  property and the {{Poincar\'e}} inequality.
\newblock {\em Mathematische Annalen}, 358(3-4): 833--860, 2014.

\bibitem{baudoin2015subelliptic}
F.~Baudoin and M.~Cecil.
\newblock The subelliptic heat kernel on the three-dimensional solvable lie groups.
\newblock {\em 
Forum Mathematicum}, 27(4):2051--2086, 2015.

\bibitem{baudoin2015log}
F.~Baudoin and Q.~Feng.
\newblock Log-sobolev inequalities on the horizontal path space of a totally
  geodesic foliation.
\newblock {\em arXiv preprint arXiv:1503.08180}, 2015.

\bibitem{BaudoinGarofalo09}
F.~Baudoin and N.~Garofalo.
\newblock Curvature-dimension inequalities and \textsc{R}icci lower bounds for
  sub-Riemannian manifolds with transverse symmetries.
\newblock {\em Journal of the European Mathematical Society,} 19(1): 151--219, 2017.

\bibitem{baudoin2019gamma}
F.~Baudoin, M.~Gordina, and D.~P. Herzog.
\newblock Gamma calculus beyond \textsc{V}illani and explicit convergence
  estimates for langevin dynamics with singular potentials.
\newblock {\em arXiv preprint arXiv:1907.03092}, 2019.

\bibitem{BGK}
F.~Baudoin, E.~Grong, K.~Kuwada, and A.~Thalmaier.
\newblock Sub-{{Laplacian}} comparison theorems on totally geodesic
  {{Riemannian}} foliations.
\newblock {\em arXiv:1706.08489}, 2017.

\bibitem{baudoinwang2012}
F.~Baudoin and J.~Wang.
\newblock Curvature dimension inequalities and subelliptic heat kernel gradient bounds on contact manifolds.
\newblock {\em Potential Analysis},  40(2): 163–193. 2014.

\bibitem{Feng}
Q.~Feng.
\newblock Harnack inequalities on totally geodesic foliations with transverse
  {{Ricci}} flow.
\newblock {\em arXiv:1712.02275}, 2017.

\bibitem{FL}
Q.~Feng and W.~Li.
\newblock Generalized gamma $ z $ calculus via sub-Riemannian density manifold.
\newblock {\em arXiv preprint arXiv:1910.07480}, 2019.

\bibitem{GL2016}
M.~Gordina and T.~Laetsch.
\newblock Sub-laplacians on sub-Riemannian manifolds.
\newblock {\em Potential Analysis}, 44(4): 811--837, 2016.

\bibitem{hebisch2009coercive}
W. Hebisch and B. Zegarlinski.
\newblock Coercive inequalities on metric measure spaces.
\newblock {\em Journal of Functional Analysis}, 258(3): 814--851, 2010.



\bibitem{inglis2009logarithmic}
J. Inglis and I. Papageorgiou.
\newblock Logarithmic Sobolev inequalities for infinite dimensional H{\"o}rmander type generators on the Heisenberg group.
\newblock {\em Potential Analysis}, 31(1): 79--102, 2009.




\bibitem{jungel2016entropy}
A.~J{\"u}ngel.
\newblock {\em Entropy methods for diffusive partial differential equations}.
\newblock Springer, 2016.


\bibitem{KL}
B.~Khesin and P.~Lee.
\newblock A nonholonomic {{Moser}} theorem and optimal transport.
\newblock {\em Journal of Symplectic Geometry}, 7(4): 381--414, 2009.

\bibitem{LiG}
W. Li.
\newblock Geometry of probability simplex via optimal transport.
\newblock {\em arXiv:1803.06360}, 2018.

\bibitem{Li2019_diffusion}
W. Li.
\newblock Diffusion Hypercontractivity via Generalized Density
  Manifold.
\newblock {\em arXiv:1907.12546}, 2019.

\bibitem{MV}
P. A. Markowich, C. Villani.
\newblock{On The Trend To Equilibrium For The Fokker-Planck Equation: An Interplay Between Physics And Functional Analysis.} 
\newblock{\em Physics and Functional Analysis}, Matematica Contemporanea, 1999. 

\bibitem{OV}
F.~Otto and C.~Villani. 
\newblock{Generalization of an inequality by Talagrand and links with the logarithmic Sobolev inequality.}
\newblock{\em Journal of Functional Analysis}, 173 (2): 361--400, 2000.

\bibitem{wang1997logarithmic}
F.-Y. WANG.
\newblock Logarithmic sobolev inequalities on noncompact Riemannian manifolds.
\newblock {\em Probability Theory and Related Fields}, 109(3): 417--424, 1997.


\end{thebibliography}
\end{document}